\documentclass[11pt,reqno]{amsart}
\usepackage{amssymb}
\usepackage{amscd,enumerate,amsfonts,calc,amsmath,verbatim}
\usepackage{mathtools}
\usepackage{amsmath, amssymb, amsfonts, amstext, amsthm, amscd}
\usepackage[mathscr]{eucal}
\usepackage{enumerate}
\usepackage{tikz,color,soul}
\usepackage{tikz-cd}
\usetikzlibrary{matrix,arrows}
\usepackage[utf8]{inputenc}
\usepackage[english]{babel}

\usepackage{thmtools} % necessary for cleveref (Cref) to work properly

\usepackage{mathrsfs}

\usepackage{verbatim}

\usepackage[new]{old-arrows}

\usepackage{amsmath,calligra,mathrsfs}
\DeclareMathOperator{\Hom}{\mathscr{H}\text{\kern -3pt {\calligra\large om}}\,} 

\usepackage{longtable}
\usepackage{hyperref}
\usepackage[all]{xy}
\usepackage{xcolor}

\makeatletter
\def\@tocline#1#2#3#4#5#6#7{\relax
  \ifnum #1>\c@tocdepth % then omit
  \else
    \par \addpenalty\@secpenalty\addvspace{#2}%
    \begingroup \hyphenpenalty\@M
    \@ifempty{#4}{%
      \@tempdima\csname r@tocindent\number#1\endcsname\relax
    }{%
      \@tempdima#4\relax
    }%
    \parindent\z@ \leftskip#3\relax \advance\leftskip\@tempdima\relax
    \rightskip\@pnumwidth plus4em \parfillskip-\@pnumwidth
    #5\leavevmode\hskip-\@tempdima
      \ifcase #1
       \or\or \hskip 1em \or \hskip 2em \else \hskip 3em \fi%
      #6\nobreak\relax
    \hfill\hbox to\@pnumwidth{\@tocpagenum{#7}}\par% <---- \dotfill -> \hfill
    \nobreak
    \endgroup
  \fi}
\makeatother
\usepackage[capitalise]{cleveref}
\usetikzlibrary{arrows.meta, positioning,arrows}
\newtheorem{theorem}{Theorem}[section]

\newtheorem{corollary}[theorem] {Corollary}
\newtheorem{lemma}[theorem]{Lemma}

\newtheorem{proposition}[theorem]{Proposition}

\newtheorem*{theorem*}{Theorem}

\theoremstyle{definition}
\newtheorem{definition}[theorem]{Definition}
\newtheorem{example}[theorem]{Example}
\newtheorem{remark}[theorem]{Remark}

\newtheorem*{notation*}{Notation}

\newcommand{\C}{\mathbb C}

\newcommand{\N}{\mathbb N}
\newcommand{\Q}{\mathbb Q}
\newcommand{\R}{\mathbb R}
\newcommand{\Z}{\mathbb Z}
\newcommand{\GL}{\mathrm{GL}}

\newcommand{\id}{\mathrm{id}}
\newcommand{\pr}{\mathrm{pr}}
\newcommand{\B}{\widetilde{\mathfrak{B}}}

\newcommand{\scrX}{\mathscr{X}}
\newcommand{\scrXE}{\mathscr{X}(\boldsymbol{\Sigma})}
\newcommand{\scrP}{\mathscr{P}}

\newcommand{\Aut}{\mathrm{Aut}}

\newcommand{\Gm}{\mathbb{G}_m}

\def \mc{\mathcal}

\allowdisplaybreaks
\begin{document}

	\title{Equivariant principal bundles over toric Deligne--Mumford stacks}

    \author{Ramandeep Singh Arora}
	\address{Department of Mathematics, Indian Institute of Science Education and Research Pune, Pune, India}
	\email{ramandeepsingh.arora@students.iiserpune.ac.in}

    \author{Chandranandan Gangopadhyay}
	\address{Department of Mathematics, Shiv Nadar University, Greater Noida, Uttar Pradesh, India}
	\email{chandranandan.g@snu.edu.in}

    \author{Mainak Poddar}
	\address{Department of Mathematics, Indian Institute of Science Education and Research Pune, Pune, India}
	\email{mainak@iiserpune.ac.in}

	

	\begin{abstract}
		Let $\mathscr{X}(\boldsymbol{\Sigma})$ be the toric Deligne–Mumford stack with stacky torus $\mathcal{T}$ associated to the stacky fan $\boldsymbol{\Sigma} = (N,\Sigma, \beta)$.
		A toric principal $H$-bundle over $\mathscr{X}(\boldsymbol{\Sigma})$ is a principal $H$-bundle $\mathscr{P}$ equipped with a $\mathcal{T}$-action lifting the $\mathcal{T}$-action on $\mathscr{X}(\boldsymbol{\Sigma})$, such that the $\mathcal{T}$-action and the $H$-action on $\mathscr{P}$ commute. For a connected reductive algebraic group $H$ over $\mathbb{C}$, we classify the isomorphism classes of framed toric principal $H$-bundles over $\mathscr{X}(\boldsymbol{\Sigma})$ in terms of piecewise linear maps from $|\Sigma|$ to the cone over the Tits building of $H$.
        Using this classification, we describe the equivariant automorphism group of a toric principal $H$-bundle, and give a necessary and sufficient condition for an equivariant reduction of the structure group.
        We also show that every toric principal $\mathrm{GL}(r)$-bundle over the weighted projective stack $\mathbb{P}(w_0,\dots,w_n)$ splits equivariantly whenever $r < n$, and that every toric principal $H$-bundle over the weighted stacky projective line $\mathbb{P}(a,b)$ splits equivariantly when $H$ is a connected reductive algebraic group.
	\end{abstract}
	
	\maketitle
	
	% \tableofcontents

%--------------------------------------------------------
%--------------------------------------------------------
\section{Introduction}
%--------------------------------------------------------
%--------------------------------------------------------

The classification of toric vector bundles was first established by Kaneyama \cite{Kaneyama}. 
He showed that every equivariant vector bundle on a smooth complete toric variety admits a torus linearization and can be encoded by combinatorial data $(m,P)$ associated with the cone complex of the variety. 
More precisely, the bundle is determined by a system of characters on each maximal cone together with transition matrices satisfying cocycle conditions, thereby reducing the geometric problem to one of linear algebra over the fan. 
Subsequently, Klyachko \cite{Klyachko} provided a more general description in terms of families of filtrations on a fixed vector space. 
Klyachko's approach, which applies to possibly singular toric varieties, establishes an equivalence between equivariant vector bundles and collections of decreasing $\Z$-filtrations indexed by the rays of the fan, satisfying certain compatibility conditions. 
These foundational results have become central tools in toric geometry and equivariant sheaf theory.
More recently, Kaveh and Manon \cite{MR4492498} have extended the classification to principal bundles. 
In their paper, Kaveh and Manon introduced the notion of a piecewise linear map from the support of a fan to the cone $\widetilde{\mathfrak{B}}(H)$ over the Tits building of a connected reductive group $H$. 
Their main theorem establishes that, for a connected reductive group $H$, isomorphism classes of framed toric principal $H$-bundles on a toric variety are classified by integral piecewise linear maps. 
This classification recovers Klyachko's classification as the special case $H = \mathrm{GL}_r$, while also yielding analogous classification results for other classical groups. 
Several other recent works have also studied the category of toric principal bundles from various viewpoints (\cite{MR3453974},  \cite{MR4166679}, \cite{MR4960069} ).

The main objective of this paper is to build upon the framework established by Kaveh and Manon for toric varieties and extend this correspondence from toric varieties to toric Deligne--Mumford stacks. 
Toric Deligne--Mumford stacks, introduced by Borisov, Chen, and Smith \cite{BCS}, generalize the classical theory of toric varieties to the category of Deligne--Mumford stacks. 
They are constructed from stacky fans, which enrich the combinatorial data of a fan by assigning lattice vectors to its rays, thereby encoding the stack structure and its isotropy groups. 
Fantechi, Mann, and Nironi \cite{FMN} developed the geometric definition of smooth toric Deligne--Mumford stacks and showed that their construction is equivalent to the stacky-fan construction of Borisov, Chen, and Smith \cite{BCS}.
For a general theory of toric stacks, we refer the reader to \cite{ToricStacksI} and \cite{ToricStacksII}.
Classification results for equivariant vector bundles over toric stacks, analogous to Klyachko's classification, have been obtained in \cite[\S 9]{IS} and \cite{Singh2026}.
More recently, Kundu \cite{kundu2026fixed} developed a combinatorial description of torsion-free toric sheaves in arbitrary dimension on smooth toric Deligne--Mumford stacks, extending the work of Klyachko, Perling, and Kool.

Our main result provides a combinatorial classification of equivariant principal bundles over smooth toric Deligne--Mumford stacks in terms of piecewise linear maps adapted to the stacky fan data defining the stack.
This recovers the classification for toric varieties when the stack is representable and provides a unified framework for studying equivariant principal bundles on smooth toric Deligne--Mumford stacks. 
To state our main results precisely, we briefly recall the description of toric Deligne–Mumford stacks.
Let $\boldsymbol{\Sigma} = (N, \Sigma, \beta)$ be a stacky fan (\Cref{defn: stacky fan}), where $N$ is a finitely generated abelian group, $\Sigma$ is a rational simplicial fan in $N_{\Q} = N \otimes_{\Z} \Q$ with $n$ rays, and $\beta \colon \Z^n \to N$ is a group homomorphism. 
If $\rho_1,\dots,\rho_n$ are the rays of $\Sigma$ with primitive generators $v_{\rho_1},\dots,v_{\rho_n}$, respectively, then $\beta$ satisfies $\beta(e_i) = a_{\rho_i} v_{\rho_i}$ for some $ a_{\rho_i} \in \Z_{>0}$.
The morphism $\beta$ induces a homomorphism $\beta^{\vee} \colon G \to (\C^*)^n$, where $G$ is a diagonalizable group. 
See \cite[\S 2]{BCS} for the construction of the morphism $\beta^{\vee}$.
The toric Deligne--Mumford stack $\mathscr{X}(\boldsymbol{\Sigma})$ associated with the stacky fan $\boldsymbol{\Sigma}$ admits a presentation as a quotient stack $[Z/G]$, where $Z$ is a smooth quasi-affine toric variety in $\mathbb{C}^n$ with torus $(\mathbb{C}^*)^n$, and $G$ acts on $Z$ via the morphism $\beta^\vee$.
The quotient stack $\mathcal{T} = [(\mathbb{C}^{*})^n/G]$ determined by $\beta^\vee$ is the Deligne--Mumford torus associated to $\mathscr{X}(\boldsymbol{\Sigma})$ (see \cite[\S 2]{FMN}).
Moreover, the toric variety $X(\Sigma)$ is the coarse moduli space of $\mathscr{X}(\boldsymbol{\Sigma})$ \cite[Proposition 3.7]{BCS}, and the coarse moduli space $\mathbb{T}$ of $\mathcal{T}$ is the dense open torus of $X(\Sigma)$.

\begin{theorem*}
Let $H$ be a connected reductive algebraic group over $\C$. 
There is a one-to-one correspondence between the isomorphism classes of framed $\mathcal{T}$-equivariant principal $H$-bundles on $\mathscr{X}(\boldsymbol{\Sigma})$ and piecewise linear maps $\Phi \colon |\Sigma| \to \widetilde{\mathfrak{B}}(H)$ such that $\Phi(a_{\rho} v_{\rho})$ is a lattice point of the building, where $\widetilde{\mathfrak {B}}(H)$ is the cone over the Tits building of $H$.
\end{theorem*}

The theorem above follows from the description of the toric Deligne--Mumford stack $\mathscr{X}(\mathbf{\Sigma})$ as the quotient stack $[Z/G]$ and the following theorem, proved in \Cref{sec: action of picard stacks}.

\begin{theorem*}
Let $G^{\bullet}=[G^{-1} \xrightarrow{d} G^{0}]$ be a morphism of abelian group schemes over $S$ such that $G^{-1},\,G^0$ are flat and locally of finite presentation over $S$. 
% If we assume flat and locally of finite presentation, then Picard stack is an algebraic stack. Then we can use theorem of SGA Expose
Let $\mathcal{G}$ be the associated Picard stack. 
Let $X$ be a scheme over $S$ and let $G^{0}$ act on $X$ over $S$. 
Let $H$ be a group scheme over $S$.  
Then there is an equivalence of categories between the category of $\mathcal{G}$-equivariant principal $H$-bundles on $[X/G^{-1}]$ and the category of $G^0$-equivariant principal $H$-bundles on $X$.
\end{theorem*}

As applications of our classification theorem, we obtain the following results, which generalize theorems of \cite{MR4767486} and \cite{MR4960069}. 

% \begin{theorem}
% The set of isomorphism classes of (framed) $\mathcal{T}$-equivariant line bundles over the stacky projective line $\mathbb{P}(a,b)$ is in bijection with the set $\Z \oplus \Z$.
% \end{theorem}

\begin{theorem*}
Let $H$ be a connected reductive algebraic group over $\C$, and let $\scrP$ be a $\mathcal{T}$-equivariant principal $H$-bundle over $\scrXE$.
Let $(\scrP_Z,p_0)$ be a framing of $\scrP$, and let $\Phi \colon |\Sigma| \to \widetilde{\mathfrak{B}}(H)$ be the corresponding piecewise linear map, see \Cref{thm: equivalence with piecewise linear maps}. 
If $\Aut_{\mathcal{T}}(\scrP)$ is the group of $\mathcal{T}$-equivariant automorphisms of $\scrP$ in the sense of \Cref{defn: equi morphism of principal bundles}, then
$$
    \Aut_{\mathcal{T}}(\scrP) 
        \cong \bigcap_{\rho \in \Sigma(1)} P_{\rho}\,,
$$
where $P_{\rho}$ is the parabolic subgroup in $H$ corresponding to the lattice point $\Phi(a_{\rho}v_{\rho}) \in \widetilde{\mathfrak{B}}_{\Z}(H)$. 
\end{theorem*}

\begin{theorem*}[Criterion for equivariant reduction of structure group]
Let $H$ be a connected reductive algebraic group over $\C$, and let $K$ be a closed connected reductive subgroup of $H$. 
A $\mathcal{T}$-equivariant principal $H$-bundle $\mathscr{P}$ over $\mathscr{X}(\boldsymbol{\Sigma})$ has an equivariant reduction of structure group to $K$ if and only if there exists a framing $(\mathscr{P}_Z,p_0)$ of $\mathscr{P}$ such that the image of the corresponding piecewise linear map $\Phi \colon |\Sigma| \to \B(H)$, satisfying $\Phi(a_{\rho}v_{\rho})$ is a lattice point of the building, lies in $\B(K)$.
\end{theorem*}

As an application of the preceding theorem, we obtain the following equivariant splitting results for weighted projective stacks in \Cref{ex: low rank bundle splits over weighted projective stack} and \Cref{ex: equivariant bundle on stacky projective line}, respectively.

\begin{theorem*}
Every toric principal $\GL(r)$-bundle over the weighted projective stack $\mathbb{P}(w_0,\dots,w_n)$ splits equivariantly whenever $r<n$.
\end{theorem*}

\begin{theorem*}
For any connected reductive algebraic group $H$, every toric principal $H$-bundle over the weighted stacky projective line $\mathbb{P}(a,b)$ splits equivariantly.    
\end{theorem*}

%--------------------------------------------------------
%--------------------------------------------------------
\section{Preliminaries}
%--------------------------------------------------------
%--------------------------------------------------------

%--------------------------------------------------------
%--------------------------------------------------------
\subsection{Tits building}
\label{sec: Tits building}
\hfill
%\\ \vspace{-0.1em}
%--------------------------------------------------------
%--------------------------------------------------------

In this section, we recall the construction of the Tits building associated with a linear algebraic group. 
We will focus on the Tits building as a simplicial complex and its geometric realizations, culminating in the definition of the cone over the building and piecewise linear maps.
For a comprehensive treatment of buildings, see \cite{TitsBuilding}, \cite{BrownBuildings} and \cite{garrett1997buildings}.

% \subsubsection*{Buildings: Axiomatic Definition}
% \hfill

% We begin with the axiomatic definition of a building. 
% For a comprehensive treatment, see \cite{TitsBuilding}, \cite{BrownBuildings} and \cite{garrett1997buildings}.

% An \emph{$n$-dimensional building} $\Delta$ is an abstract simplicial complex together with a collection of distinguished subcomplexes, called \emph{apartments}, that satisfy the following axioms:
% \begin{enumerate}
% \item Every $k$-simplex of $\Delta$, with $k < n$, is contained in at least three $n$-simplices.

% \item Any $(n – 1)$-simplex in an apartment $A$ lies in exactly two adjacent $n$-simplices of $A$ and the graph of adjacent $n$-simplices is connected.

% \item Any two simplices of $\Delta$ lie in a common apartment.

% \item If two simplices are contained in both apartments $A_1$ and $A_2$, then there exists a simplicial isomorphism from $A_1$ onto $A_2$ that fixes the vertices of those simplices.
% \end{enumerate}
% An $n$-simplex in an apartment is called a \emph{chamber}, and the \emph{rank} of the building is defined to be $n + 1$.

% Intuitively, a building is a complex that is glued together from apartments, each of which is a Coxeter complex (a triangulation of a sphere by Weyl chambers). Axiom (4) ensures that the gluing is consistent.

%--------------------------------------------------------
%--------------------------------------------------------
\subsubsection*{The Tits building of a linear algebraic group}
\hfill 
%--------------------------------------------------------
%--------------------------------------------------------

Let $H$ be a linear algebraic group over $\C$.
The Tits building $\Delta(H)$ of $H$ is an abstract simplicial complex constructed from the proper parabolic subgroups of $H$.

\begin{itemize}
\item \textbf{Simplices and parabolic subgroups:}
There is a bijection between the simplices of $\Delta(H)$ and the proper parabolic subgroups of $H$. 
The correspondence is inclusion-reversing: if $P_1$ and $P_2$ are parabolic subgroups with $P_2 \subset P_1$, then the simplex $\Delta_{P_1}$ corresponding to $P_1$ is a face of the simplex $\Delta_{P_2}$ corresponding to $P_2$. 
Consequently, the maximal simplices (chambers) correspond to minimal parabolic subgroups, i.e., Borel subgroups of $H$.

\item \textbf{Apartments and maximal tori:}
Apartments in $\Delta(H)$ correspond to maximal tori of $H$. 
Specifically, for a maximal torus $T \subset H$, the associated apartment $A_T$ consists of all simplices $\Delta_P$ such that the parabolic subgroup $P$ contains $T$.

\item \textbf{Reduction to the semisimple case:} 
Since every parabolic subgroup contains the solvable radical $R(H)$ of $H$, it follows that the parabolic subgroups of $H$ are in bijection with those of the semisimple group $H/R(H)$. 
Hence, the building $\Delta(H)$ is canonically identified, as a simplicial complex, with the building $\Delta(H/R(H))$. 
Therefore, when studying the building, we may assume without loss of generality that $H$ is semisimple.
\end{itemize}

% \begin{remark}
% \label{rem: intersection of any two borels contains a maximal tori}
% For any linear algebraic group $H$, we have any two simplices in $\Delta(H)$ lie in a common apartment.
% Without loss of generality, we can assume that $H$ is semisimple.
% Since simplices correspond to parabolic subgroups, it is enough to show that any two simplices corresponding to Borel subgroups lie a common apartment.
% % This is because parabolic if and only if contained in a Borel subgroup.
% This follows from the fact that intersection of any two Borel subgroups in a reductive group contains a maximal tori, see \cite[Page 173 Corollary]{HumphreysLAG}.
% % Every semisimple group is reductive.
% \end{remark}

%--------------------------------------------------------
%--------------------------------------------------------
\subsubsection*{Geometric realization of apartments via cocharacter lattices}
\hfill
%--------------------------------------------------------
%--------------------------------------------------------

To work with the Tits building, it is useful to have a concrete geometric model. 
Fix a maximal torus $T$ of a connected reductive linear algebraic group $H$ over $\C$.
Let $\Phi(H,T)$ be the root system lying in the character lattice $\Lambda(T) = \operatorname{Hom}(T,\mathbb{G}_m)$.
%, and let the Weyl group $W = N_H(T)/T$ act on the real vector space $V = \Lambda(T) \otimes_{\mathbb{Z}} \mathbb{R}$.
Let $\Lambda^\vee(T) = \operatorname{Hom}(\mathbb{G}_m, T)$ denote the cocharacter lattice of $T$, and set
$$
    \widetilde{A}_T 
        = \Lambda^\vee_{\mathbb{R}}(T) 
        = \Lambda^\vee(T) \otimes_{\mathbb{Z}} \mathbb{R}.
$$
We refer to the vector space $\widetilde{A}_T$ as the \emph{extended apartment} corresponding to $T$.
% be the real vector space of dimension $r = \dim(T) = \operatorname{rank}(G)$.

\begin{itemize}
\item \textbf{Weyl Chambers and Apartments:} 
The (spherical) apartment $A_T$ corresponding to $T$ can be realized as the triangulation of the unit sphere $S_T$ in $\Lambda^\vee_{\mathbb{R}}(T)$ obtained by intersecting it with the (closures of the) Weyl chambers and their faces. 
% More concretely, the apartment is the spherical Coxeter complex associated to the root system. 
The chambers in $A_T$ correspond to the Weyl chambers in $\Lambda^\vee_{\mathbb{R}}(T)$.

\item \textbf{Parabolic Subgroups and Faces:} 
There is a one-to-one correspondence between the faces of the Weyl chambers in $\Lambda^\vee_{\mathbb{R}}(T)$ and the parabolic subgroups of $H$ that contain $T$. 
% A parabolic subgroup corresponds to the face that is the intersection of the walls of the Weyl chambers determined by the simple roots that generate the Levi factor of the parabolic.
\end{itemize}

%--------------------------------------------------------
%--------------------------------------------------------
\subsubsection*{The underlying space and the cone over the Tits building}
\hfill
%--------------------------------------------------------
%--------------------------------------------------------

Using the apartment realizations, we can glue the spherical apartments along common faces to obtain the \emph{underlying space of the Tits building of $H$}, denoted $\mathfrak{B}(H)$. 
This space is formed by gluing the spherical apartments $A_T$, for all maximal tori $T$ of $H$, along the simplices that correspond to the same parabolic subgroup. 
The result is a topological space $\mathfrak{B}(H)$ that is homeomorphic to the geometric realization of the simplicial complex $\Delta(H)$.

A related and often more convenient object is the \emph{cone over the Tits building of $H$}, denoted $\widetilde{\mathfrak{B}}(H)$. 
It is obtained by gluing the vector spaces $\Lambda^\vee_{\mathbb{R}}(T)$, for all maximal tori $T$ of $H$, along the faces of the Weyl chambers corresponding to the same parabolic subgroups.
\subsubsection*{One-parameter subgroups and the lattice points in the cone over the Tits building}
\label{sec: one-para-subgroups and lattice points}
\hfill
%--------------------------------------------------------
%--------------------------------------------------------

Define two one-parameter subgroups $\lambda_1,\lambda_2 \colon \mathbb{G}_m \to H$ to be equivalent if $\lim_{s \to 0} \lambda_1(s)\lambda_2(s)^{-1}$ exists in $H$, i.e., the morphism 
$$
    \lambda_1\lambda_2^{-1} \colon \mathbb{G}_m \to H
$$ 
extends to a regular morphism from the affine line $\mathbb{A}^1$ to $H$.
Given a one-parameter subgroup $\lambda \colon \mathbb{G}_m \to H$, we associate a parabolic subgroup
$$
    P_{\lambda} 
        = \left\{h \in H 
        \,\big\vert\, \lim_{s \to 0} \lambda(s) h \lambda(s)^{-1} \text{ exists in } H\right\} \,,
$$
see \cite[Section 2.2, Definition 2.3/Proposition 2.6]{GIT}.
% where $\lim_{s \to 0} \lambda(s) h \lambda(s)^{-1}$ exists in $H$ means that the morphism $\lambda h \lambda^{-1} \colon \mathbb{G}_m \to H$, given by $s \mapsto \lambda(s) h \lambda(s)^{-1}$, extends to regular morphism from the affine line $\mathbb{A}^1$ to $H$.
We note that the image $\lambda(\mathbb{G}_m)$ is a closed subgroup of $P_{\lambda}$.
The parabolic subgroup $P_{\lambda}$ can also be described as follows:
$$
    P_{\lambda} 
        = \left\{ h \in H 
        \mid h \lambda h^{-1} \sim \lambda \right\}.
$$
It is straightforward to see that if $\lambda_1$ is equivalent to $\lambda_2$, then $P_{\lambda_1} = P_{\lambda_2}$.

The following proposition was established in \cite[Proposition 1.8]{MR4492498}. 
We present an alternative proof that is more elementary.

\begin{proposition}
Let $\lambda_1, \lambda_2 \colon \mathbb{G}_m \to H$ be two one parameter subgroups.
Then the following are equivalent:
\begin{enumerate}
\item[(1)] $\lambda_1$ is equivalent to $\lambda_2$.

\item[(2)] There exists $g \in P_{\lambda_1}$ such that $g \lambda_1 g^{-1} = \lambda_2$.
\end{enumerate}
\end{proposition}

\begin{proof}
$(1) \implies (2)$ Since $\lambda_1 \sim \lambda_2$, we have $P_{\lambda_1} = P_{\lambda_2}$. 
Denote this common parabolic subgroup by $P$.
Since $\lambda_i(\mathbb{G}_m)$ is a closed subgroup of $P$, it is contained in a maximal torus of $P$.
Let $T_i$ be a maximal tori of $P$ containing $\lambda_i(\mathbb{G}_m)$.
Since any two maximal tori are conjugate, there exists $g \in P$ such that $gT_1g^{-1} = T_2$.
Consider the one-parameter subgroup $g \lambda_1 g^{-1}$.
Then $g \lambda_1 g^{-1} \sim \lambda_1$ since $g \in P = P_{\lambda_1}$.
Hence, $g \lambda_1 g^{-1} \sim \lambda_2$.
Moreover, $g \lambda_1 g^{-1}$ and $\lambda_2$ can be considered as one-parameter subgroups of $T_2$.
But, two one-parameter subgroups of a torus are equivalent if and only if they are equal.
Hence, $g \lambda_1 g^{-1} = \lambda_2$.

$(2) \implies (1)$ Since $g \in P_{\lambda_1}$, it follows that $\lambda_1 \sim g\lambda_1g^{-1} = \lambda_2$.
\end{proof}

Let $T$ be a maximal torus of $H$, and let $P$ be a parabolic subgroup of $H$ containing $T$.
Then
$$
    \Lambda_{P}^{\vee}(T) 
        = \left\{ \lambda \in \Lambda^{\vee}(T) \mid P_{\lambda} = P \right\}
$$
gives the face corresponding to $P$ in the extended apartment $\widetilde{A}_T$.
In other words, $\Lambda_{P}^{\vee}(T)$ consists precisely of those cocharacters of $T$ that determine the parabolic subgroup $P$.
An immediate consequence of the above proposition is the following.
If $T_1$ and $T_2$ are maximal tori of $H$ contained in a common parabolic subgroup $P$ of $H$, then there exists a unique $\Z$-linear isomorphism between $\Lambda^{\vee}(T_1)$ and $\Lambda^{\vee}(T_2)$ such that the induced $\R$-linear isomorphism between $\Lambda^\vee_{\mathbb{R}}(T_1)$ and $\Lambda^\vee_{\mathbb{R}}(T_2)$ maps the face corresponding to $P$ in $\Lambda^\vee_{\mathbb{R}}(T_1)$ onto the face corresponding to $P$ in $\Lambda^\vee_{\mathbb{R}}(T_2)$, see \cite[Proposition 1.10]{MR4492498}.

The set of \emph{lattice points in the cone over the Tits building of $H$}, denoted $\widetilde{\mathfrak{B}}_{\mathbb{Z}}(H)$, is the subset of $\widetilde{\mathfrak{B}}(H)$ obtained by taking the union of the lattices $\Lambda^\vee(T) \subset \Lambda^\vee_{\mathbb{R}}(T)$, for all maximal tori $T$ of $H$, with the understanding that if two maximal tori $T_1$ and $T_2$ are contained in a common parabolic subgroup $P$, then a lattice point in $\Lambda^\vee(T_1)$ is identified with its image in $\Lambda^{\vee}(T_2)$ under the canonical isomorphism $\Lambda^\vee(T_1) \cong \Lambda^\vee(T_2)$ induced by $P$.
Equivalently, $\widetilde{\mathfrak{B}}_{\mathbb{Z}}(H)$ is the subset of $\widetilde{\mathfrak{B}}(H)$ obtained by gluing the lattices $\Lambda^{\vee}(T) \subset \Lambda^\vee_{\mathbb{R}}(T)$, for all maximal tori $T$ of $H$ along these identifications.
Hence, $\widetilde{\mathfrak{B}}_{\Z}(H)$ can be identified with the set of equivalence classes of one-parameter subgroups of $H$, see \cite[Corollary 1.11]{MR4492498}.
Under the above identification, every element of $\B_{\Z}(H)$ determines a parabolic subgroup of $H$.

\begin{remark}
We emphasize that the assumption that $H$ is reductive is essential for the definition of the cone over the Tits building of $H$.
Indeed, the Weyl chamber decomposition of the extended apartment $\widetilde{A}_T$ relies on the existence of the root system $\Phi(H,T)$.
Which is defined when $H$ is reductive.
Furthermore, the definition of $P_{\lambda}$ as a parabolic also requires 
$H$ to be reductive, see \cite[Section 2.2, Definition 2.3/Proposition 2.6]{GIT}.
Hence, throughout this paper, we work only with the cone over the Tits building associated with connected reductive groups.
\end{remark}

% In view of \cite[Proposition 1.10]{MR4492498}, we can define the cone over the Tits building and its lattice points for an arbitrary linear algebraic group $H$ via equivalence classes of one-parameter subgroups of $H$.
% This description of the cone over the Tits building is the one used in \cite{MR4492498} and \cite{MR4767486}.

%--------------------------------------------------------
%--------------------------------------------------------
\subsubsection*{Functoriality of the cone over the Tits building}
\hfill
%--------------------------------------------------------
%--------------------------------------------------------

If $\alpha \colon H \to H'$ is a morphism of linear algebraic groups, then it induces a map  $\widehat{\alpha} \colon \B_{\Z}(H) \to \B_{\Z}(H')$, defined by $\widehat{\alpha}([\lambda]) = [\alpha \circ \lambda]$, for any one-parameter subgroup $\lambda \colon \mathbb{G}_m \to H$.
This extends to a map $\widehat{\alpha} \colon \B(H) \to \B(H')$ as follows. 
Suppose $T$ is a maximal torus in $H$. 
Then $\alpha(T)$ is a torus in $H'$, and hence is contained in a maximal torus, say $T'$, of $H'$.
If $\lambda \colon \mathbb{G}_m \to T$ is a one-parameter subgroup in $\Lambda^{\vee}(T)$, then $\alpha \circ \lambda \colon \mathbb{G}_m \to T'$ is a one-parameter subgroup in $\Lambda^{\vee}(T')$.
We note that $\Z$-linear map $\Lambda^{\vee}(T) \to \Lambda^{\vee}(T')$, given by $\lambda \mapsto \alpha \circ \lambda$, is the restriction of $\widehat{\alpha} \colon \B_{\Z}(H) \to \B_{\Z}(H')$, since no two one-parameter subgroups in $\Lambda^{\vee}(T)$ and in  $\Lambda^{\vee}(T')$ are equivalent.
Moreover, the map $\Lambda^{\vee}(T) \to \Lambda^{\vee}(T')$ extends to an $\R$-linear map $\Lambda^{\vee}_{\R}(T) \to \Lambda^{\vee}_{\R}(T')$.
This gives us a map from $\bigsqcup_{T} \Lambda_{\R}^{\vee}(T) \to \B(H')$, where $T$ varies over all maximal tori of $H$.
% Now it is enough to show that the lattice points corresponding to a parabolic subgroup $P$ contained in two different maximal tori $T_1$ and $T_2$ of $H$ maps to the same element in $\B(H)$.
Suppose $T_1$ and $T_2$ are maximal tori of $H$ contained in a common parabolic subgroup $P$ of $H$.
Then $T_2 = hT_1 h^{-1}$ for some $h \in P$.
Let $T_1'$ be the maximal torus of $H'$ such that $\alpha(T_1) \subset T_1'$, and let $T_2' = \alpha(h) T_1 \alpha(h)^{-1}$.
Then $\alpha(T_2) \subset T_2'$.
Consider the diagram
\[
% https://tikzcd.yichuanshen.de/#N4Igdg9gJgpgziAXAbVABwnAlgFyxMJZABgBpiBdUkANwEMAbAVxiRAB12AZOgWwCModAHrBONGDAC+ACgAqAfQCMAShBTS6TLnyEUZJVVqMWbTjwFDR4ybMUAmNRq3Y8BIkvJH6zVog7cfIIiYuwS0vLKAOROmiAYrroepIbUPqb+5kFWoeF2CvYx6kYwUADm8ESgAGYAThC8SGQgOBBInsa+ZuwA7liwABZ0OMADUiDUDHT8MAwACtpueiC1WGUDOOpxdQ1N1K1I9mkmfgF9g8OhjGhD484gO42IHQeIAMzHXZm9-TBDI5xrrctjV6k8ji02u9Phkzr9-lcGDc6DIBio7hQpEA
\begin{tikzcd}
\Lambda^{\vee}(T_1) \arrow[d, "\widehat{h}"'] \arrow[r, "\widehat{\alpha}"] & \Lambda^{\vee}(T_1') \arrow[d, "\widehat{\alpha(h)}"] \\
\Lambda^{\vee}(T_2) \arrow[r, "\widehat{\alpha}"]                           & \Lambda^{\vee}(T_2')                                 
\end{tikzcd}
\]
where $\widehat{h} \colon \Lambda^{\vee}(T_1) \to \Lambda^{\vee}(T_2)$ is the isomorphism, given by $\lambda \mapsto h \lambda h^{-1}$ (see\cite[Proposition 1.10]{MR4492498}), and $\widehat{\alpha(h)}$ defined similarly.
Since the $\R$-linear isomorphism $\widehat{h} \colon \Lambda^{\vee}_{\R}(T_1) \to \Lambda^{\vee}_{\R}(T_2)$ takes the face corresponding to $P$ in $\Lambda^{\vee}_{\R}(T_1)$ onto the face corresponding to $P$ in $\Lambda^{\vee}_{\R}(T_2)$ and the diagram above commutates, it follows that we get a well-defined map $\widehat{\alpha} \colon \B(H) \to \B(H')$.

We note that if $K$ is a closed subgroup of $H$, then the induced map $\B(K) \to \B(H)$ is injective.
Indeed, suppose that $\lambda_1,\lambda_2 \colon \mathbb{G}_m \to K$ are one-parameter subgroups of $K$ such that $j \circ \lambda_1 \sim_H j \circ \lambda_2$, where $j \colon K \hookrightarrow H$ is the inclusion homomorphism.
That is, $j \circ (\lambda_1\lambda_2^{-1}) = (j \circ \lambda_1)(j \circ \lambda_2)^{-1}$ extends to a regular function from the affine line $\mathbb{A}^1 \to H$.
Since $K$ is a closed subgroup, it follows that the extension factors through $K$.
% Use $f(\overline{A}) \subseteq \overline{f(A)}$.
Hence, $\lambda_1 \sim_K \lambda_2$.

%--------------------------------------------------------
%--------------------------------------------------------
\subsubsection*{Piecewise Linear Maps into the Cone}
\hfill
%--------------------------------------------------------
%--------------------------------------------------------

Let $N$ be a lattice, and let $\Sigma$ be a fan in the $\mathbb{R}$-vector space $N_{\mathbb{R}}:= N \otimes_{\Z} \R$ with support $|\Sigma|$.

\begin{definition}[Piecewise linear map]
A map $\Phi \colon | \Sigma | \to \widetilde{\mathfrak{B}}(H)$ is said to be a \emph{piecewise linear map} if it satisfies the following two conditions:
\begin{enumerate}
\item For each cone $\sigma \in \Sigma$, there exists a maximal torus $T_{\sigma}$ of $H$ such that $\Phi(\sigma) \subset \Lambda^{\vee}_{\mathbb{R}}(T_{\sigma})$.

\item For each cone $\sigma \in \Sigma$, the restriction $\left.\Phi\right|_{\sigma} \colon \sigma \to \Lambda^{\vee}_{\mathbb{R}}(T_{\sigma})$ is $\mathbb{R}$-linear.
\end{enumerate}
We say that $\Phi$ is \emph{integral}, if for each cone $\sigma$, the restriction $\left.\Phi\right|_{\sigma}$ induces a $\Z$-linear map 
$$
    \left.\Phi\right|_{\sigma} \colon \sigma \cap N \to \Lambda^{\vee}(T_{\sigma}),
$$
that is, $\Phi$ sends lattice points to lattice points.
\end{definition}

% In other words, $\Phi$ is locally linear with respect to a choice of apartment (i.e., a maximal torus) on each cone, and the linear maps on overlapping cones are compatible with the gluing that defines $\widetilde{\mathfrak{B}}(H)$. 

%--------------------------------------------------------
%--------------------------------------------------------
\subsection{Toric principal $H$-bundles}
\hfill
% \\ \vspace{-0.4em}
%--------------------------------------------------------
%--------------------------------------------------------

Let $H$ be a linear algebraic group over $\C$.
A principal $H$-bundle over a scheme $X$ is a scheme $P$ equipped with a map $\pi \colon P \to X$ satisfying the following conditions:
\begin{enumerate}
\item the morphism $\pi$ is flat, locally of finite presentation, and surjective.
    
\item the scheme $P$ is equipped with a right $H$-action $\rho \colon P \times H \to P$ such that the following diagram commutes:
\[
% https://tikzcd.yichuanshen.de/#N4Igdg9gJgpgziAXAbVABwnAlgFyxMJZABgBpiBdUkANwEMAbAVxiRAAUACAHW7wFt4nABIgAvqXSZc+QigCM5KrUYs27cZJAZseAkTLzl9Zq0QdNU3bKKKj1E2vMANXqQB045TCgBzeESgAGYAThD8SGQgOBBIiiqmbLxoIQD68pYgoeGR1DFIAEwOqmYgvCEAFrHUDHQARjAM7NJ6ciAhWL4VOJnZEYhF0bGIAMzFiebJWL1h-fH5o+NOZdxo02IUYkA
\begin{tikzcd}
P \times H \arrow[r, "\pr_1"] \arrow[d, "\rho"'] & P \arrow[d, "\pi"] \\
P \arrow[r, "\pi"]                               & {X\,.}            
\end{tikzcd}
\]

\item the map $(\pr_1,\rho) \colon P \times H \to P \times_X P$ is an isomorphism.
\end{enumerate}

Suppose $P$ is a principal $H$-bundle over $X$ and $Q$ is a principal $H'$-bundle over $X$.
A \emph{morphism of principal bundles with respect to a group homomorphism of linear algebraic groups} $\alpha \colon H \to H'$ is a map $F \colon P \to Q$ such that the following diagrams commutes:
\[
% https://tikzcd.yichuanshen.de/#N4Igdg9gJgpgziAXAbVABwnAlgFyxMJZABgBpiBdUkANwEMAbAVxiRAAUQBfU9TXfIRQAmclVqMWbAIrdeIDNjwEiARlKrx9Zq0QgAGnL5LBRACxjq2qXvYACADoO8AW3h2AEkYX9lQ5BaaVpK6HN6KAiooAKwaWiEyTqQAdOG+pjGWEjoyjs5YbnCeAOTc4jBQAObwRKAAZgBOEC5IZCA4EEjq2TYgAGLejc2t1B1Ioj2hTmhYAPqc1Ax0AEYwDOzpUSANWJUAFjiDTS2I3WOIE9ZTDjOzsjz1x0gAzKOdiBaTbE4NexDzIEWKzWGxMWx2+0ODxAQxOn3OsS+egG0NhSAAbG8kIirt8HL9-vd5GjEK92u9MUj+nlXO4nIw0Hs6GUuEA
\begin{tikzcd}
P \arrow[rr, "F"] \arrow[rd, "\pi_P"'] &   & Q \arrow[ld, "\pi_Q"] &  & P \times H \arrow[d, "\rho_P"'] \arrow[r, "F \times \alpha"] & Q \times H' \arrow[d, "\rho_Q"] \\
                                       & X &                       &  & P \arrow[r, "F"]                                             & {Q\,.}                         
\end{tikzcd}
\]

Let $\Sigma$ be a fan in an $\R$-vector space $N_{\R}$, and let $X(\Sigma)$ be the toric variety associated with $\Sigma$, with torus $\mathbb{T}$. 
Let $x_0 \in X$ be a fixed point in the open orbit of the $\mathbb{T}$-action on $X(\Sigma)$. 
It identifies the torus $\mathbb{T}$ with the open orbit via $t \mapsto t\cdot x_0$.
We say that a principal $H$-bundle over $X(\Sigma)$ is a \emph{toric principal $H$-bundle} if $P$ is equipped with a left $\mathbb{T}$-action such that the actions of $\mathbb{T}$ and $H$ commute, and $\pi$ is a $\mathbb{T}$-equivariant morphism.
A \emph{framed toric principal $H$-bundle} over $X(\Sigma)$ is a toric principal $H$-bundle $P$ equipped with a distinguished point $p_0$ in the fiber $P_{x_0}$; such a pair is denoted by $(P, p_0)$.
Specifying $p_0 \in P_{x_0}$ is equivalent to choosing an $H$-equivariant isomorphism between the fiber $P_{x_0}$ and $H$, where $H$ acts on itself by right multiplication.
A \emph{morphism of toric principal bundles} is a morphism of principal bundles (with respect to a group homomorphism as above) that is also $\mathbb{T}$-equivariant.
A \emph{morphism of framed toric principal bundles} $(P,p_0) \to (Q,q_0)$ is a morphism of toric principal bundles which maps $p_0 \in P_{x_0}$ to $q_0 \in Q_{x_0}$.

\begin{remark}
\label{rem: change of framing}
Suppose $(P,p_0)$ is a framed toric principal $H$-bundle over $X(\Sigma)$.
If $\alpha_{h_0} \colon H \to H$ is the conjugation homomorphism, given by $h \mapsto h_0^{-1}hh_0$, then the morphism $F \colon (P,p_0) \to (P,p_0 \cdot h_0)$, given by $F(p) = p \cdot h_0$, is a morphism of framed toric principal $H$-bundles with respect to the morphism $\alpha_{h_0}$, since
$
    F(p \cdot h)
        = (p \cdot h) \cdot h_0
        = (p \cdot h_0) \cdot (h_0^{-1}hh_0)
        = F(p) \cdot \alpha_{h_0}(h).
$
\end{remark}

\begin{definition}[Equivariant triviality]
Let $P$ be a toric principal $H$-bundle $P$ over $X(\Sigma)$.
\begin{enumerate}

\item[(1)] We say $P$ is \emph{equivariantly trivial} if there exists a $\mathbb{T}$-equivariant isomorphism of principal $H$-bundle between $P$ and a trivial toric principal bundle $X(\Sigma) \times H$, where $\mathbb{T}$ acts on $X(\Sigma) \times H$ by
$$
    t \cdot (x,h) = (t \cdot x , \phi(t)h)\,,
$$
for some algebraic group morphism $\phi \colon \mathbb{T} \to H$.

\item[(2)] We say $P$ is \emph{locally equivariantly trivial} if for any cone $\sigma \in \Sigma$, the restriction $\left.P\right|_{U_{\sigma}}$ is equivariantly trivial, where $U_{\sigma}$ denotes the affine open subset in $X(\Sigma)$ corresponding to the cone $\sigma \in \Sigma$.
\end{enumerate}
\end{definition}

\begin{remark}
In \cite[Lemma 2.8]{BDP2016}, it is shown that over $\C$, every toric principal $H$-bundle over $X(\Sigma)$ is locally equivariantly trivial.
In contrast, \cite[Theorem 2.5]{MR4492498} proves the same result over an arbitrary field, provided that $H$ is reductive.
\end{remark}

The following theorem is the main classification result for toric principal bundles over a toric variety, proved in \cite[Theorem 2.2]{MR4492498}.

\begin{theorem}
\label{thm: main-theorem kaveh-manon}
Let $X(\Sigma)$ be a toric variety and $H$ be a connected reductive algebraic group over $\C$.
Then there is a one-to-one correspondence between the isomorphism classes of framed toric principal $H$-bundles over $X(\Sigma)$ and the integral piecewise linear maps $\Phi \colon |\Sigma| \to \B(H)$.

Moreover, let $\alpha \colon H \to H'$ be a homomorphism of connected reductive algebraic groups over $\C$.
Let $P$ (respectively $Q$) be a framed toric principal $H$-bundle (respectively $H'$-bundle) over $X(\Sigma)$ with corresponding piecewise linear maps $\Phi \colon |\Sigma| \to \B(H)$ (respectively $\Phi' \colon |\Sigma| \to \B(H')$).
Then there exists a morphism of framed toric principal bundles $P \to Q$, which is equivariant with respect to $\alpha$, if and only if $\Phi' = \widehat{\alpha} \circ \Phi$, where $\widehat{\alpha} \colon \B (H) \to \B (H')$ is the map induced by $\alpha$ between the cones over the Tits buildings.
\end{theorem}

The idea of the proof of above theorem is as follows. 
Let $\Phi\colon |\Sigma|\to \widetilde{\mathfrak{B}}(H)$ be an integral piecewise linear map. 
For each cone $\sigma\in\Sigma$, the restriction $\left.\Phi\right|_{\sigma}$ determines a homomorphism $\phi_\sigma\colon \mathbb{T}_{\sigma}\to H$, where $\mathbb{T}_{\sigma}$ denotes the stabilizer of the unique orbit $O_{\sigma}$ corresponding to $\sigma$.
We then extend it to a homomorphism $\phi_\sigma\colon \mathbb{T} \to H$. 
On each affine toric chart $X_\sigma$, consider the trivial principal $H$-bundle $\mathcal{P}_\sigma = X_\sigma \times H$, where $\mathbb{T}$ acts diagonally on $X_\sigma$ in the usual way and on $H$ via $\phi_\sigma$. 
For two cones $\sigma,\sigma'\in\Sigma$ with $\tau=\sigma\cap\sigma'$, define the transition function $\psi_{\sigma,\sigma'}\colon X_\tau\to H$ on the open torus orbit by  
\[
    \psi_{\sigma,\sigma'}(t\cdot x_0)
        =\phi_{\sigma'}(t)\phi_\sigma(t)^{-1}.
\]
The compatibility condition on $\Phi$ ensures that $\psi_{\sigma,\sigma'}$ extends to a regular map on all of $X_\tau$; moreover, the transition functions satisfy the cocycle condition. 
Thus the local bundles $\mathcal{P}_\sigma$ glue together to give a framed toric principal $G$-bundle $\mathcal{P}_\Phi$. 
Conversely, given a locally equivariantly trivial framed toric principal $H$-bundle $\mathcal{P}$, the equivariant local trivializations yield homomorphisms $\phi_\sigma$, and hence integral linear maps on the cones of $\Sigma$ to $\B(H)$.
After choosing the trivializations compatibly with the framing, these maps agree on intersections as maps to $\widetilde{\mathfrak{B}}(H)$, and therefore glue together to give an integral piecewise linear map $\Phi_{\mathcal{P}}\colon |\Sigma|\to \widetilde{\mathfrak{B}}(H)$. 
Finally, one checks that the two constructions are inverse to each other, establishing the desired bijection.

%--------------------------------------------------------
%--------------------------------------------------------
\section{Preliminaries on action of a Picard stack} 
%--------------------------------------------------------
%--------------------------------------------------------

%--------------------------------------------------------
%--------------------------------------------------------
\subsection{Picard stacks}
%--------------------------------------------------------
%--------------------------------------------------------

Throughout the article, by a stack over a scheme $S$ we mean a stack on the site $({\rm Sch}/S)_{\rm \acute{e}tale}$  of $S$-schemes with \'etale topology \cite[Definition 4.6.1] {Olsson}. 
Recall that a Picard stack over $S$ is a stack $\mathcal{G}$ over $S$ equipped with a composition law given by morphism of stacks 
$$
    m \colon \mathcal{G}\times_S \mathcal{G} \to \mathcal{G}
$$
along with an associativity $2$-arrow and a commutativity $2$-arrow satisfying some compatibility conditions, see \cite[Expose XVIII, Definition 1.4.2]{SGA4} and \cite[Definition B.1]{FMN}. 
By \cite[Expose XVIII, 1.4.4]{SGA4}, there exists a couple $(e,\epsilon)$, where $e \colon S \to \mathcal{G}$ is a neutral section and $\epsilon \colon e \cdot e \Rightarrow e$ is a $2$-arrow. 
In particular, a Picard stack over $S$ is a group stack in the sense of \cite[Definition 3.1.1]{Breen-Theorie}. 
Any Picard stack over $S$ can be constructed as follows \cite[Expose XVIII, \S 1.4.11 and Proposition 1.4.15]{SGA4}: 
Let $G^{\bullet}=[G^{-1}\xrightarrow{d} G^{0}]$ be a complex of sheaves of abelian groups on $S$. 
% On $S$'' in the last line mean ``on the site $\mathrm{Sch}/S$ with \'etale topology
The Picard stack $\mathcal{G}$ associated to $G^{\bullet}$ is the stackification of the prestack $\mathcal{G}^{\rm pre}$ given by
\begin{itemize}
  \item[(I)] For $ U \to S $, we have $ \text{Ob}(\mathcal{G}^{\rm pre}(U)) = G^0(U) $.
  
  % \item[(II)] If $ x, y \in G^0(U) $, a morphism from $ x $ to $ y $ is an element $ f \in G^{-1}(U) $ such that $ df = y - x $.

  \item[(II)] If $ g_0, g_0' \in G^0(U) $, a morphism from $ g_0 $ to $ g_0' $ is an element $ g_1 \in G^{-1}(U) $ such that $ dg_1 = g_0' - g_0 $.
  
  \item[(III)] The composition law for morphisms is given by addition in $ G^{-1}(U) $.
  
  \item[(IV)] The functor %$ + $
  $\mathcal{G}^{\rm pre} \times_S \mathcal{G}^{\rm pre} \to \mathcal{G}^{\rm pre}$ is defined by the addition laws in $ G^0(U) $ and $ G^{-1}(U) $.
  
  \item[(V)] The associativity and commutativity isomorphisms are the identity morphisms, represented by the zero element of $G^{-1}(U)$.
\end{itemize}

% {\color{red}{FMN in remark 1.12 only states that the stack mentioned above is the quotient stack $[G^0/G^{-1}]$  for $G^0, G^{-1}$ diagonalizable groups, why? Deligne's construction is for any homomorphism of any sheaf of abelian groups, and it seems to me that that construction is nothing but the quotient stack.}}

% \textcolor{blue}{Yes. I also agree this is the quotient stack. In the construction of $[X/G^{-1}]$ below, if we replace $X$ by $G^{0}$, we will get back $[G^{0}/G^{-1}]$. The equation 
% $$
%     \sigma(d(f),x)=y \text{ becomes } d(f)+x = y \iff d(f) = y-x.
% $$
% This is because now the $\sigma \colon G^{0} \times G^{0} \to G^{0}$ is the ``addition'' map, since $G^0$ is a sheaf of abelian groups. The extra structure comes since $G^{0}$ is a sheaf of abelian groups. }

%--------------------------------------------------------
%--------------------------------------------------------
\subsection{Action of Picard stacks}
\label{sec: action of picard stacks}
%--------------------------------------------------------
%--------------------------------------------------------

Since a Picard stack $\mathcal{G}$ over $S$ is a group stack, we have the notion of $\mathcal{G}$ acting on a stack $\mathscr X$ over $S$ consisting of the following data: 
\begin{itemize}
    \item a morphism of $S$-stacks called an action map:
    \begin{align*}
        a \colon \mathcal{G} \times_S \mathscr X & \to \mathscr X \\
        (g,x) & \to g \cdot x
    \end{align*}

    \item an associativity $2$-arrow 
    $
    \mu \colon % a \circ (m \times \mathrm{id}_{\mathscr{X}}) \implies a \circ (\mathrm{id}_{\mathcal{G}} \times a)
    $
\begin{equation}
\label{diag: action of a picard stack}
% https://tikzcd.yichuanshen.de/#N4Igdg9gJgpgziAXAbVABwnAlgFyxMJZABgBpiBdUkANwEMAbAVxiRAB12BbOnACwDGjYAHEAvgAJOeLvAD6AZSndeg4eOUz5Szj35wBAJ2AANMSDGl0mXPkIoyARiq1GLNrtVCGoydKyycIrKenwGxmYWViAY2HgERABM5C70zKyIHCr83r6aAdohquGm5pbWcXZJpM7Uae6ZnvpGpRYuMFAA5vBEoABmhhBcSGQgOBBIjnVuGSBc+YFF-IZcwFhQYnLATWEtZubUDHQARjAMAAo28fYghlidfDhR-YPDiFNjE4gAzNPpbHRniABkMRtRxkhkq5-o1snwVmsNlsdrlxH52Fo4BJAeVga9IeCvr9oQ0QDjoiC3h8IYgoQwsGBZlAIExjgxWNQ+DA6FAkGAmAwGOC6FgGGxIIyQH9SbomG0xEA
\begin{tikzcd}
\mathcal{G} \times_S \mathcal{G} \times_S \mathscr{X} \arrow[d, "m \times \mathrm{id}_{\mathscr{X}}"'] \arrow[rr, "\mathrm{id}_{\mathcal{G}} \times a"] &  & \mathcal{G} \times_S \mathscr{X} \arrow[d, "a"] \\
\mathcal{G} \times_S \mathscr{X} \arrow[rr, "a"] \arrow[rru, "\mu", Rightarrow,shorten <=8pt,  shorten >=8pt]                                                                         &  & \mathscr{X},                                    
\end{tikzcd}
\end{equation}
also denoted by $\mu \colon (g_1\cdot g_2) \cdot x \to g_1 \cdot (g_2 \cdot x)$.

    \item an identity $2$-arrow 
$
    \eta \colon 
    % a \circ (e \times \mathrm{id}_{\mathscr{X}}) \implies \mathrm{id}_{\mathscr{X}}
$ 
\[
% https://tikzcd.yichuanshen.de/#N4Igdg9gJgpgziAXAbVABwnAlgFyxMJZABgBpiBdUkANwEMAbAVxiRAB12BbOnACwDGjYAHEAvgAJOeLvAD6AZSndefOAIBOwABpiQY0uky58hFGQCMVWoxZtOPfuq279hkBmx4CRC+Wv0zKyIHCpOmjp6BkZepr6kVtSBdiH61jBQAObwRKAAZhoQXEh+IDgQSGQ2QWwwyjLwyo58GlzAWFBicsAOqs6RUe4FRZXU5UgATEm2wSB0bvmFxYil44hT1Smhza3tnd294S5ietQMdABGMAwACsbeZiAaWJl8OAsgw8tVawDMZ1gwLMoBAmBcGKxqHwYHQoEgwEwGAwxnQsAw2JAgSBpjUQpwYDh5mIKGIgA
\begin{tikzcd}
\mathcal{G} \times_S \mathscr{X} \arrow[r, "a"] \arrow[rd, "\eta", swap, Rightarrow, pos= .1, shorten >=40pt]                       & \mathscr{X} \\
\mathscr{X} \arrow[u, "e \times \mathrm{id}_{\mathscr{X}}"] \arrow[ru, "\mathrm{id}_{\mathscr{X}}"'] & {}         ,
\end{tikzcd}
\]
also denoted by $\eta \colon e \cdot x \to x$.
% , which induces an isomorphism $a(e,x) = x$ where $e \colon S\to \mathcal{G}$ is the identity section \cite[Remark B.7]{FMN}.
\end{itemize} 
These data satisfy some compatibility relations, see \cite[Definition B.12]{FMN} and \cite[\S 3.4]{aldrovandi-NoohiII}.

The key example that we are going to consider in this section is the following: Suppose that we have a complex $G^{\bullet}=[G^{-1}\xrightarrow{d} G^0]$ of sheaves of abelian groups on $S$. 
Let $X$ be scheme over $S$ and let $G^0$ act on $X$ via an action map $\sigma \colon G^0\times X\to X$. 
Then we have an action of the Picard stack $\mathcal{G}$ associated to $G^{\bullet}$ on the quotient stack $[X/G^{-1}]$ as follows: 
Let us denote by $[X/G^{-1}]^{\rm pre}$ the quotient prestack (\cite[\href{https://stacks.math.columbia.edu/tag/044O}{Tag 044O}]{stacks-project}). 
Recall that 
\begin{comment}
\begin{itemize}
    \item[(I)] For $U\to S$, we have {\rm Ob}($[X/G^{-1}]^{\rm pre}(U))=X(U)$.
    \item[(II)] For $x,y\in X(U)$ a morphism from $x$ to $y$ is given by an element $f\in G^{-1}(U)$ such that $f\cdot x:=\sigma(d(f),x)=y$.
\end{itemize}
\end{comment}
\begin{itemize}
    \item[(I)] $\mathrm{Ob}([X/G^{-1}]^{\rm pre})$ consists of pairs $(U,x)$ where $U$ is a scheme over $S$ and $x$ is an $S$-morphism from $U$ to $X$. 
    In particular, we have {\rm Ob}($[X/G^{-1}]^{\rm pre}(U))=X(U)$.
    
    \item[(II)] A morphism $(U,x) \to (U',x')$ consists of a pair $(f,g)$ where $f \colon U \to U'$  is a $S$-morphism and $g_1' \in G^{-1}(U')$ such that $g_1' \cdot x := \sigma(dg_1',x) = x' \circ f$.
    In particular, for $x,x' \in X(U)$ a morphism $x \to x'$ is given by an element $g_1' \in G^{-1}(U)$ such that 
    $$
        g_1' \cdot x := \sigma(d(g_1'),x) = x'.
    $$
\end{itemize}
Note that we have a morphism of prestacks
  \begin{equation}
  \label{eqn:action on prestacks}
        \mathcal{G}^{\rm pre}\times_S [X/G^{-1}]^{\rm pre}\to [X/G^{-1}]^{\rm pre}
  \end{equation}
  
  \begin{itemize}
  \item[(I)] On objects it is defined as follows: For $U\to S$ and $g_0\in G^0(U), x\in X(U)$ we define 
  $$
    (g,x) \mapsto \sigma(g,x).
    $$
  
  \item[(II)] On morphisms it is defined as follows: For $U\to S$, $g_0,g_0'\in G^0(U)$ and $g_1 \in G^{-1}(U)$ such that $d(g_1)=g_0'-g_0$ (i.e., $g_1$ is a morphism from $g_0$ to $g_0'$) and $x,x' \in X(U)$ with $g_1'\in G^{-1}(U)$ such that $g_1'\cdot x=x'$ (i.e., $g_1'$ is a morphism from $x$ to $x'$) we define 
  $$
    (g_1,g_1') \mapsto g_1g_1'.
  $$ 
  It can be checked easily that $g_1g_1'$ is a morphism from $g_0x \to g_0'x'$.
  
  \item[(III)] The associativity and identity isomorphisms are provided by the zero element of $ G^{-1}(U) $.
  \end{itemize} 
Since stackification commutes with products \cite[\href{https://stacks.math.columbia.edu/tag/04Y1}{Tag 04Y1}]{stacks-project}, by stackifiying the morphism (\ref{eqn:action on prestacks}) we get an action of $\mathcal{G}$ on the stack $[X/G^{-1}]$. 
Note that the associativity and identity isomorphism and the compatibility conditions follow from \cite[\href{https://stacks.math.columbia.edu/tag/04W9}{Tag 04W9}]{stacks-project}.

We now define the notion of a $\mathcal{G}$-equivariant morphism, following \cite[Eqn.~(6.1.5)]{Breen-bitorseur}:

\begin{definition}
\label{defn: equivariant map}
Let $\mathscr{X},\mathscr{Y}$ be two stacks over $S$ with an action of a Picard stack $\mathcal{G}$. 
Let $a_{\mathscr X}, a_{\mathscr{Y}}$ be the corresponding action maps. 
A \emph{$\mathcal{G}$-equivariant map} from $\mathscr{X}$ to $\mathscr{Y}$ consists of the following data:
    \begin{itemize}
        \item a morphism of $S$-stacks 
        $
            \pi \colon \mathscr{X}\to \mathscr{Y}
        $
        
        \item a $2$-arrow $q$:
% $q \colon a_{\mathscr Y}\circ (\mathrm{id}_{\mathcal{G}} \times \pi) \implies \pi\circ a_{\mathscr{X}}$
\begin{equation}
\label{diag: equivariant map 2-isomorphism}
% https://tikzcd.yichuanshen.de/#N4Igdg9gJgpgziAXAbVABwnAlgFyxMJZABgBpiBdUkANwEMAbAVxiRAB12BbOnACwDGjYAHEAvgAJOeLvCndefOAIBOwABpiQY0uky58hFAEZyVWoxZtOPfsrWbtukBmx4CRMsfP1mrRBwK-EIMopLSWLJw8rZKqsAAmlo6em6GRKbe1L5WATaK9onJ5jBQAObwRKAAZioQXEhkIDgQSKYWfmx0APrA+XbxmsnOtfWN1C1IAEzZlv6BsSpcwFhQYr39gsLi4ewycmgg1Ax0AEYwDAAK+u5GICpYZXw4TjV1DYgzza2IAMyznQCPT6QTiaiSwzeY0Q7UmfwBuRcrxAow+Xzh7QYWDA8ygECYpwYrGofBgdCgSDATAYDAmdCwDDYkBxRw6iIAjqyTucrjd0gEidUXmIKGIgA
\begin{tikzcd}
\mathcal{G} \times \mathscr{X} \arrow[r, "a_{\mathscr{X}}"] \arrow[d, "\mathrm{id}_{\mathcal{G}} \times \pi"'] & \mathscr{X} \arrow[d, "\pi"] \\
\mathcal{G} \times \mathscr{Y} \arrow[r, "a_{\mathscr{Y}}"] \arrow[ru, "q", Rightarrow, shorten <=8pt, shorten >=8pt]                      & \mathscr{Y}   ,      
\end{tikzcd}
\end{equation}
also denoted by $q \colon g\cdot \pi(x) \to \pi(g \cdot x)$.
\end{itemize}
These data must satisfy the following conditions:
\begin{enumerate}
    \item For every chart $U$,  objects $g,g' \in \mathcal{G}(U)$ and $x \in \mathscr{X}(U)$, the following diagram commutes
\begin{equation}
\label{diag equivariant 1}
% https://tikzcd.yichuanshen.de/#N4Igdg9gJgpgziAXAbVABwnAlgFyxMJZABgBpiBdUkANwEMAbAVxiRAAoBzAAgB1eAxlAg5unAOQBKPoOGj+aLOwAekkAF9S6TLnyEUARnJVajFm078hIruKtyFS1dP50ATsjeleIS7JHcAI4+FBpaIBjYeAREAEzG1PTMrIi+9jaOtumizmHaUXpEZAYmSeapmVzZYlLVqnkROtH6yPEliWYpIJl+1ji2Mn3cueomMFCc8ESgAGZuEAC2SGQgOBBIRqbJbPwLTAD6wLt0OAAWcAJu3ACa6iDUDHQARjAMAApNhakMMDM4DXNFstqGskABmDrbVJBe4gR4vd6fGKpNxYTinf6aWbzJaIeKrdaIAAskPKIECsPhrw+BWRcN+mPCgNxEIJSBJWzJmV2ByOvAWJ3Ol24AA11GoHs9qUj9PS-hoKOogA
\begin{tikzcd}
(g \cdot g') \cdot \pi(x) \arrow[r, "\mu_{\mathscr Y}"] \arrow[d, " q"'] & {g\cdot(g'\cdot\pi(x)) \ar[r,"g\cdot q"]} & g\cdot(\pi(g'\cdot x)) \arrow[d, "q"] \\
\pi((g\cdot g')\cdot x) \arrow[rr, "\pi(\mu_{\mathscr X})"]              &                                           & \pi(g\cdot(g' \cdot x))              
\end{tikzcd}
\end{equation}
where $\mu_{\mathscr X}$ and $\mu_{\mathscr Y}$ are the  associativity 2-arrows for the action of $\mathcal{G}$ on $\mathscr{X}$ and $\mathscr{Y}$ respectively.
    
    \item For every chart $U$ and every object $x \in \mathscr{X}(U)$, the following diagram commutes:
\begin{equation}
\label{diag equivariant 2}  
% https://tikzcd.yichuanshen.de/#N4Igdg9gJgpgziAXAbVABwnAlgFyxMJZABgBpiBdUkANwEMAbAVxiRBgAIAdLgYygg5uXNFgAUADwCUIAL6l0mXPkIoAjOSq1GLNj1FiYPfoI7S5CkBmx4CRDWq31mrRCH3jzsrTCgBzeCJQADMAJwgAWyQyEBwIJA1tFzYARwsQ8KjEGLikACZqZ103HhgcOgB9YB4IuhwACzheUOAATVlZEGoGOgAjGAYABSVbVRBQrD96nHSQMMiE6lzEAqTi9xFxUvKKmrrG5uAADVkZb1kgA
\begin{tikzcd}
e \cdot \pi(x) \arrow[r, "q"] \arrow[rd, "\eta_{\mathscr{Y}}"'] & \pi(e\cdot x) \arrow[d, "\pi(\eta_\mathscr{X})"] \\
                                                                & \pi(x)                                          
\end{tikzcd}
\end{equation}
where $\eta_{\mathscr X}$ and $\eta_{\mathscr Y}$ are the  identity 2-arrows for the action of $\mathcal{G}$ on $\mathscr{X}$ and $\mathscr{Y}$ respectively.
\end{enumerate} 
\end{definition}

\begin{remark}
    If the Picard stack is a group scheme, then our definition of the $\mathcal{G}$-equivariant map is compatible with the one given by Romagny in \cite[Definition 1.3(ii)]{Romagny}.
    %\item {\color{red} In the definition of $\mathcal{G}$-torsors in \cite[\S 3.4]{aldrovandi-NoohiII}, where $\mathcal{G}$ is a group stack, the condition (2) is not included.}
\end{remark}

\begin{definition}
[{\cite[Definition 1.3(i), Definition 2.1(ii)]{Romagny}}]
\label{defn: strict action}
Let $H$ be a group scheme over $S$, and let $\mathscr{X}$ and $\mathscr{Y}$ be stacks over $S$ equipped with $H$-actions.
\begin{enumerate}
    \item An action of $H$ on $\mathscr{X}$ is said to be \emph{strict} if the associativity $2$-arrow $\mu_H$ and the identity $2$-arrow $\eta_H$ are identity $2$-isomorphisms.

    % {\color{blue}
    %     strict action means $\mu_H \colon x \cdot (h_1 \cdot h_2) \to (x \cdot h_1) \cdot h_2$ and $\eta_H \colon x \cdot e \to x$ are identity (for my reference).
    % }
         
    \item An $H$-equivariant morphism from $\mathscr{X}$ to $\mathscr{Y}$ is said to be \emph{strict} if the $2$-arrow $q_H$ is the identity $2$-isomorphism.

    % {\color{blue}
    %     strict morphism means $q_H \colon \pi(x) \cdot h = \pi(x) \to \pi(x \cdot h)$ is identity (for my reference).
    % }
\end{enumerate}
\end{definition}

\begin{remark}
If $\pi \colon \mathscr{X} \to \mathscr{Y}$ is an $H$-equivariant morphism of stacks equipped with strict $H$-actions, then $q_H \colon e \cdot \pi(x) \to \pi(e \cdot x)$ is the identity map by diagram \eqref{diag equivariant 2}.
% \textcolor{red}{I don't think I can conclude that $q_H \colon g \cdot \pi(x) \to \pi(g \cdot x)$ is identity in this case, i.e, a morphism between two stacks equipped with strict action can be non-strict.}
\end{remark}

Following Romagny \cite[Section 1]{Romagny}, any action of a group scheme $H$ on an algebraic stack $\mathcal{X}$ can be replaced by an equivalent stack $\mathcal{X}^{\mathrm{str}}$ equipped with a strict $H$-action. 
The strictification construction encodes these coherence isomorphisms into the object data, producing a stack on which the associativity and identity axioms hold strictly. Moreover, the
construction is functorial and yields a $2$-equivalence $\mathcal{X}^{\mathrm{str}} \simeq \mathcal{X}$, which is itself an $H$-morphism.
Therefore, throughout this paper we will consider only strict $H$-actions.
We describe Romagny's result in more detail in \Cref{sec: Appendix (Strict Actions)}.

%--------------------------------------------------------
%--------------------------------------------------------
\subsection{$\mathcal{G}$-equivariant principal $H$-bundles}
%--------------------------------------------------------
%--------------------------------------------------------

Let $\mathscr{X}$ be an algebraic stack over $S$ \cite[8.1.4]{Olsson}.
Recall that for a group scheme $H$ over $S$ we have the notion of (right) $H$-torsors over $\mathscr{X}$, that is, a sheaf over the big \'etale site $\mathscr{X}_{\rm \acute{e}tale}$ with a action of $H$ satisfying the usual torsor properties.
Here, $\mathscr{X}_{\rm \acute{e}tale}$  is the structure of site on $\mathscr X$ inherited from $({\rm Sch}/S)_{\rm \acute{e}tale}$ \cite[\href{https://stacks.math.columbia.edu/tag/06TP}{Definition 06TP}]{stacks-project}; see also \cite[\href{https://stacks.math.columbia.edu/tag/03AH}{Definition 03AH}]{stacks-project} and \cite[\S 4.5.1]{Olsson} for the definition of torsor on a site. 
For a more general discussion where $H$ is any group stack, see \cite[Definition 6.1]{Breen-bitorseur} \cite[\S 3.4]{aldrovandi-NoohiII} for $H$ any group stack.
On the other hand, just as in the case of schemes \cite[Definition 4.5.4]{Olsson}, one can define principal $H$-bundles on $\mathscr X$ as follows:

\begin{definition}
\label{defn: principal bundle}
A \emph{(right) principal $H$-bundle} over $\mathscr X$ is a pair $(\mathscr P \xrightarrow{\pi} \mathscr X, \rho \colon \mathscr{P}\times_S H \to \mathscr{P})$ such that 
\begin{enumerate}
    \item $\pi$ is representable by schemes,
    % algebraic spaces
    flat, locally of finite presentation, and surjective.
         
    \item $\rho$ is a strict group action of $H$ on $\mathscr{P}$ and $\pi$ is a strict $H$-invariant morphism.
         
    \item The map $(\pr_1, \rho) \colon \mathscr P\times_S H\to \mathscr{P}\times_{\mathscr X} \mathscr P$ is an equivalence.
\end{enumerate}
\end{definition}

\begin{remark}
Although not all morphisms between strict $H$-algebraic stacks are themselves strict, each is equivalent to a strict $H$-morphism; see \Cref{prop: strictification morphism}. 
Since we will be working with the category of principal $H$-bundles over $\mathscr{X}$ (see \Cref{sec: cat of principal bundles}), which is obtained from the corresponding $2$-category by taking 2-equivalence classes of 1-morphisms, we may, without loss of generality, assume that the projection to the base $\pi \colon \mathscr{P} \to \mathscr{X}$ is strict. 
Hence, we will assume that all principal $H$-bundles are strict $H$-algebraic stacks and that their projections to the base are strict $H$-morphisms.
\end{remark}

\begin{remark}
\begin{enumerate}
\item If the structure morphism $H \to S$ is flat, affine and locally of finite presentation, then the natural functor from the category of principal $H$-bundles on $\mathscr X$ to the category of $H$-torsors on $\mathscr X$ is an equivalence, as in \cite[Proposition 4.5.6]{Olsson}.

\item There is a canonical equivalence between the category of principal $H$-bundles over $\mathscr X$ in the sense of \Cref{defn: principal bundle} and in the sense of \cite[\S 1.2]{Biswas-Majumdar-Wong}.
Indeed, by \cite[Proposition 1.2]{Biswas-Majumdar-Wong}, a principal bundle in the latter sense is equivalent to a morphism of stacks $\phi \colon \mathscr X \to BH$.
Given a bundle $(\pi \colon \mathscr P \to \mathscr X,\rho)$ that satisfies our definition, the representability of $\pi$ ensures that for every object $x \colon U\to\mathscr X$, the pullback $\mathscr P_U := \mathscr P \times_{\mathscr X} U$ is a scheme over $U$, and the action $\rho$ equips it with the structure of a classical principal $H$-bundle, thus defining a morphism $\phi \colon \mathscr X \to BH$.
Conversely, pulling back the universal $H$-bundle $\mathrm{pt} \to BH$ along any morphism $\phi:\mathscr X \to BH$ yields a bundle $\mathscr P \to \mathscr X$ whose total space is representable by schemes and which satisfies the torsor axioms, and hence gives an object of our category.

% \item Perhaps the strictness of the actions and the morphism is not necessary. 
% The actions can be strictified by \cite[Proposition 1.5]{Romagny}. 
% For a $G$-equivariant morphism $\mc M\to \mc N$ (with $G$ acting on $\mc M$ and $\mc N$ via strict action) we can perhaps replace $\mc M$ by the stack $\mc M\times_{\mc N}\mc N$ and consider the projection $\mc M\times_{\mc N}\mc N\to \mc N$. 
% It seems to me that this is a strict $G$-equivariant morphism and this morphism is isomorphic to the morphism $\mc M\to \mc N$.
\end{enumerate}
\end{remark}

\begin{lemma}
\label{lemma: pullback of a principal bundle is a principal bundle}
If $\pi \colon \mathscr{P} \to \mathscr{X}$ is a principal $H$-bundles over an algebraic stack $\mathscr{X}$ and $X \to \mathscr{X}$ is a smooth atlas, then the induced morphism $\mathscr{P} \times_{\mathscr{X}} X \to X$ is a principal $H$-bundle.
\end{lemma}

\begin{proof}
Let $P := \mathscr{P}\times_{\mathscr X}X$. 
Note that we have the following $2$-cartesian diagram
\begin{equation*}
% https://tikzcd.yichuanshen.de/#N4Igdg9gJgpgziAXAbVABwnAlgFyxMJZABgBpiBdUkANwEMAbAVxiRAAUACAHW7wFt4AfQDKnABIgAvqXSZc+QigCM5KrUYs2XXgOEANTu2myQGbHgJEATGur1mrRBxNyLiomWXqHW57346HAALOABjACdgdikePixBOFEJVzN5SyVkVW97TScQAKDQyOipXQThYEKQ8Kj9KVjq4qiY1PMFKxRbHI1HNiba0rb0jxQAZjtevxB9YfdO5Ame33yBkvreUgA6aXUYKABzeCJQADMIiH4kMhAcCCQJqacwJgYGagY6ACMYBnYRzogBgwU44VLnS7Xah3JCqJ5IF5vD7fX7-eZKIEgsEyM4XK6IOEwxAAFlyfUQiPeQJRfwBGOBoPBeNh0PuiFs8Iprypnx+tPRbAZ2NMEPxHKJAFYyX5Kci+WiOvSsUzIYhHkTSZyCtwIsF7nLUXTBcqcSBRUhNZLpc9uQb+YrjYzTeb2aykAA2Z3MxDut2IADsXtVUtubMDFCkQA
\begin{tikzcd}
P \times_S H \arrow[d] \arrow[r]         & P \times_X P \arrow[d] \arrow[r]                      & P \arrow[d] \arrow[r] & X \arrow[d]      \\
\mathscr{P} \times_S H \arrow[r, "\rho"] & \mathscr{P}\times_{\mathscr{X}} \mathscr{P} \arrow[r] & \mathscr{P} \arrow[r] & {\mathscr{X}\,.}
\end{tikzcd}
\end{equation*}
Hence, it follows from \cite[\href{https://stacks.math.columbia.edu/tag/04XD}{Tag 04XD}]{stacks-project} that the morphism $P \to X$ is a principal $H$-bundle, where the action morphism $P \times_S H \to P$ is obtained by pulling back the action morphism $\mathscr{P} \times_S H \to \mathscr{P}$ along the morphism $P \to \mathscr{P}$.  
\end{proof}

\begin{definition}
Let $\pi \colon \mathscr{P} \to \mathscr{X}$ and $\pi' \colon \mathscr{Q} \to \mathscr{X}$ be principal $H$-bundles over a stack $\mathscr{X}$.
A \emph{morphism of principal $H$-bundles} over $\mathscr X$ is an $H$-equivariant morphism of stacks $f \colon \mathscr{P} \to \mathscr{Q}$, together with a $2$-isomorphism $\zeta \colon \pi \Rightarrow \pi' \circ f$ such that for every chart $U$, every object $p \in \mathscr{P}(U)$, and every $h \in  H(U)$ the following diagram commutes:
\[
% https://tikzcd.yichuanshen.de/#N4Igdg9gJgpgziAXAbVABwnAlgFyxMJZABgBpiBdUkANwEMAbAVxiRAB120sAKNTgMZQIOAAQALAJQgAvqXSZc+QigBM5KrUYs2nbgHIeAMz6DhYqdLkLseAkXUBGTfWatEHLlkMm0k0WYiEpKcpNTWIBi2ykRkztSuOh56vH6y8pGKdirIjqTxWm66Xj58klaaMFAA5vBEoEYAThAAtkhkIDgQSHmFSZ4AXjA4dOkNzW2IHV1IAMzUDFhg7iDCTABGDKzU4jB0UGyQyyDUI1gMhwSsEU2tc6fdiAAsCdornEMjYyC3ky+dj3UIEWxw8a022xAu32l2Opzo51h1wyvx6DyQQMSKwAjgB9IyyCgyIA
\begin{tikzcd}
\pi(p\cdot h) \arrow[rr, "\zeta"] \arrow[d, Rightarrow, no head] &                                           & \pi'(f(p\cdot h)) \arrow[d, "q_f"] \\
\pi(p) \arrow[r, "\zeta"]                                        & \pi'(f(p)) \arrow[r, Rightarrow, no head] & {\pi'(f(p) \cdot h)\,,}           
\end{tikzcd}
\]
where $q_f$ is the $2$-arrow corresponding to the $H$-equivariant morphism $f$, see diagram \eqref{diag: equivariant map 2-isomorphism}.
\end{definition}

\begin{lemma}
\label{lemma: morphism of principal bundles is an isomorphism}
If $f \colon \mathscr{P} \to \mathscr{Q}$ is a morphism of principal $H$-bundles over an algebraic stack $\mathscr{X}$, then the morphism $f$ is an equivalence.
\end{lemma}

\begin{proof}
Let $X \to \mathscr{X}$ be a smooth atlas.
Consider the $2$-cartesian diagram
\[
% https://tikzcd.yichuanshen.de/#N4Igdg9gJgpgziAXAbVABwnAlgFyxMJZABgBoBGAXVJADcBDAGwFcYkQAdDgW3pwAs4AYwBOwAAoBfEJNLpMufIRTkK1Ok1bsuvAcLEBFabPnY8BIgCY1NBizaJOPPoNHAAGpK6kAdDLkgGGZKVqTE6nZaju7+pooWKmERmg5Ouq6GkgAEXHjc8AD6wDou+h6S2TEmgQrmyiRJtinaznpuUjkceYXFrRnllTLqMFAA5vBEoABmIhDcSGQgOBBIqhr27FOxIDNzqzTLSJbVu-OIAMwHK4jHAadIACxXqyezZwCsz4jEr3uIn0trg9fmcnoCkOdJJRJEA
\begin{tikzcd}
\mathscr{P} \times_{\mathscr{X}} X \arrow[d] \arrow[r] & \mathscr{Q} \times_{\mathscr{X}} X \arrow[d] \arrow[r] & X \arrow[d]      \\
\mathscr{P} \arrow[r, "f"]                             & \mathscr{Q} \arrow[r]                                  & {\mathscr{X}\,.}
\end{tikzcd}
\]
Since $\mathscr{P} \times_{\mathscr{X}} X \to \mathscr{Q} \times_{\mathscr{X}} X$ is a morphism of principal $H$-bundles over $X$, it is an equivalence.
Hence, by \cite[\href{https://stacks.math.columbia.edu/tag/04XD}{Tag 04XD}]{stacks-project}, it follows that the morphism $f$ is an equivalence.
\end{proof}

%-------------------------------------------------------------------

Let $\mathcal{G}$ be a Picard stack over $S$ acting on a stack $\mathscr X$ over $S$.
Following \cite[Definition~3.1.8]{Breen-Theorie}, we define the notion of a (right) $\mathcal{G}$-equivariant $H$-principal bundles as follows:

\begin{definition}
\label{def: equivariant bundle}
A \emph{(right) $\mathcal{G}$-equivariant principal $H$-bundle} over $\mathscr{X}$ consists of the following data:
\begin{itemize}
    \item a principal $H$-bundle $\mathscr{P}$ over $\mathscr{X}$ with action morphism $a_H \colon \mathscr{P} \times_S H \to \mathscr{P}$,

    \item an action of the Picard stack $\mathcal{G}$ on $\mathscr{P}$ with action morphism $a_{\mathscr{P}} \colon \mathcal{G}\times_S \mathscr{P} \to \mathscr{P}$,

    \item a commutativity $2$-arrow $b$ (i.e., the two actions of $\mathscr{P}$ commutes):
\begin{equation}
\label{diag: G action commutes with H action}
% https://tikzcd.yichuanshen.de/#N4Igdg9gJgpgziAXAbVABwnAlgFyxMJZABgBpiBdUkANwEMAbAVxiRAB12BbOnACwDGjYAHEAvgAJOeLvAD6AZSndefOAIBOwAAqTpWWXEUSAEiDGl0mXPkIoyARiq1GLNpx78hDUXvYz5JQ9VdS1dc0sQDGw8AiIAJnJnemZWRA4VflCdPwCjJTMLKxjbBNInahS3dOCszRzOUmoxZxgoAHN4IlAAMw0ILiQyEBwIJAdqBjoAIxgGbWtYuxANLHa+HBBK1zSMzz4NLmAsKDE5YFrBYXFcg3gJOnOTMQje-sHERJGxxABmbdSbEewGerxAfQG42ooyQ-xcgPSwMu3l8LyK4PeQ2hPy+VV2SMyVx8N2UeWU+0Ox1OckKkQhHwm3yQuJ2bGmWxADCwYF2UAgTGmDFY1D4MDoUCQYCYDAY0LoWAYbEgPPMFDEQA
\begin{tikzcd}
\mathcal{G} \times_S \mathscr{P} \times_S H \arrow[d, "\mathrm{id}_{\mathcal{G}} \times a_{H}"'] \arrow[rr, "a_{\mathcal{G}} \times \mathrm{id}_H"] &  & \mathscr{P} \times_S H \arrow[d, "a_{H}"] \\
\mathcal{G} \times_S \mathscr{P} \arrow[rr, "a_{\mathcal{G}}"] \arrow[rru, "b", Rightarrow,  shorten <=8pt, shorten >=8pt]                                                         &  & {\mathscr{P}\,,}                         
\end{tikzcd}
\end{equation}
also denoted by $b \colon g\cdot (x \cdot h) \to (g \cdot x) \cdot h$, 
\end{itemize}
such that the projection $\pi \colon \mathscr{P} \to \mathscr{X}$ is a $\mathcal{G}$-equivariant morphism (in the sense of Definition \ref{defn: equivariant map}, with $2$-arrow $q \colon a_{\mathscr X} \circ (\mathrm{id}_{\mathcal{G}} \times \pi) \implies \pi\circ a_{\mathscr{P}}$).
These data must satisfy the following conditions:
\begin{enumerate}
    \item For every chart $U$,  objects $g,g' \in \mathcal{G}(U)$, $x \in \mathscr{P}(U)$ and $h \in H(U)$, the following diagram commutes
    \begin{equation}
    \label{diag: equi-prin-bundle diag 1}
    % https://tikzcd.yichuanshen.de/#N4Igdg9gJgpgziAXAbVABwnAlgFyxMJZABgBpiBdUkANwEMAbAVxiRAAoBzTgcgEp2ADwAWfEAF9S6TLnyEUAJnJVajFm3ZdefQX2ESpIDNjwEiZAIwr6zVohCcuPIaLGTpJuUSVXqN9fZcTrp6Bh6yZigWpL6qtmyOwXpuKjBQnPBEoABmAE4QALZIZCA4EEjRcQEgAEYg1Ax0NTAMAAoypvIgDDDZOGEgeYXF1GVISlV2IAA60wVMAPrAswV0OMJwAMa5wK3i4vXdTS3tnpEguVicwv3ug-lFiJVjiADMfmpTK4vLc2sb212+wABLNNlAIDhgfoGsc2h0vPYen0BkNHhMXgAWD7xeycUHTcGQ4F1O5opDY0rlN446qkijiIA
\begin{tikzcd}
(gg')(xh) \arrow[rr, "b"] \arrow[d, "\mu_{\mathscr{P}}"'] &                          & ((gg')x)h \arrow[d, "\mu_{\mathscr{P}} \cdot h"] \\
g(g'(xh)) \arrow[r, "g \cdot b"]                          & g((g'x)h) \arrow[r, "b"] & (g(g'x))h                                       
\end{tikzcd},
    \end{equation}
where $\mu_{\mathscr P}$ and $\mu_{\mathscr X}$ are the  associativity 2-arrows for the action of $\mathcal{G}$ on $\mathscr{P}$ and $\mathscr{X}$ respectively.

\item For every chart $U$, every object $g \in \mathcal{G}(U)$, $x \in \mathscr{P}(U)$ and $h,h' \in H(U)$, the following diagram commutes
    \begin{equation}
    \label{diag: equi-prin-bundle diag 2}
    % https://tikzcd.yichuanshen.de/#N4Igdg9gJgpgziAXAbVABwnAlgFyxMJZABgBpiBdUkANwEMAbAVxiRAHMAKAD04As+AcgCUwkAF9S6TLnyEUZAIxVajFm07tuw-kLGTp2PASIAmcivrNWiEJp59RQiVJAYjcs6WXUr625ya2o7OBm4yxvLIiha+ajYcgdwhIhIqMFDs8ESgAGYAThAAtkhkIDgQSDGq1mwARiDUDHR1MAwAChGetvlY7Hw4LnmFJYjm5ZWIAMxxtbZ1AAQAOksAxlAQOAuhrgXFpdQVSAAsTVhgCRtMdQys1HwwdFBskBeN5XRYDC8ErGF7o1OEyQ4z8CQa-xGVUOkxmIAY50uEGut3eDyePzeh0+31srz+FHEQA
\begin{tikzcd}
g(x(hh')) \arrow[d, "b"'] \arrow[r, Rightarrow, no head] & g((xh)h') \arrow[r, "b"] & (g(xh))h' \arrow[d, "b \cdot h'"] \\
(gx)(hh') \arrow[rr, Rightarrow, no head]                &                          & ((gx)h)h'                        
\end{tikzcd}
    \end{equation}

    \item For every chart $U$ and every object $x \in \mathscr{P}(U)$, the following diagram commutes

\begin{equation}
\label{diag: equi-prin-bundle diag 3}
% https://tikzcd.yichuanshen.de/#N4Igdg9gJgpgziAXAbVABwnAlgFyxMJZABgBpiBdUkANwEMAbAVxiRBgH1gAdbgWzo4AFgGNGwAOIBfKbxFQIOAAQAKAB5K5C5Z2AAJKQEoQU0uky58hFAEZyVWoxZsNWxUt0GTZkBmx4CIgAme2p6ZlZEEDVvc38rIjIbB3DnKJVdXgFhMQZJGU1ueXc1Q0LinQ49WN8LAOtkO2Swp0j2LizBUXFpKXLtJRjTOMtAlBDmxwiXEwcYKABzeCJQADMAJwg+JDIQHAgkOym0kF4YHDoO-i64EXXgAAUZEGoGOgAjGAYHuoSohhgqxwNQ2Wx21H2SAAzC1plF3i8QG9Pt9fmMQOssAshMDhiBQdtEEdIYgQsdImAmAwGK8Pl8fvF0QCgYiGFgwG0FEx3gDEUIYHQoGxIByQZtCWSSQBWWFpSnU2kohmjaxIwHA17szkQbm86j8wXCgisPEE6EQg6IAAssopVJpSLpqMZquZGqRWrYXJ5rH1AqFURFJp8ZutFqQMvJbDOFyu2SEt3uTykrKdyvqbDdsykQA
\begin{tikzcd}
e_{\mathcal{G}}\cdot (x \cdot e_{H}) \arrow[r, "\eta_{\mathscr{P}}"] \arrow[d, "b"'] & x \cdot e_{H} \arrow[r, Rightarrow, no head]            & x \arrow[d, Rightarrow, no head] \\
(e_{\mathcal{G}} \cdot x) \cdot e_H \arrow[r, Rightarrow, no head]                   & e_{\mathcal{G}} \cdot x \arrow[r, "\eta_{\mathscr{P}}"] & x                               
\end{tikzcd}
\end{equation}
where $\eta_{\mathscr P}$ and $\eta_{\mathscr X}$ are the identity 2-arrows for the action of $\mathcal{G}$ on $\mathscr{P}$ and $\mathscr{X}$ respectively.
% {\color{red}
% \item May need to include if the $H$-action/morphism are non-strict, since the proposition below won't be valid.
% \[
% % https://tikzcd.yichuanshen.de/#N4Igdg9gJgpgziAXAbVABwnAlgFyxMJZABgBpiBdUkANwEMAbAVxiRAB120sAKAcwAEnAMZQIOATwAeQ9qPECAFgEplIAL6l0mXPkIoyARiq1GLNp249+s+RKnLbYiSo1aQGbHgJFD5E-TMrIgggiLOslYy4Qqumtpeer6kxtSB5iGWvGFyEQ5OsW4Juj4oAEz+aWbBoQUSPFnSjjEuavEeOt76yBWppkFsNi2RvA7NuYXqJjBQfPBEoABmAE4QALZIZCA4EEiG7Svrm9Q7SGUHqxuIFdu7iAAsF0eIfrdIAMxPV+8ndwCsXyQ91+SABFHUQA
% \begin{tikzcd}
% \pi(g \cdot (x \cdot h)) \arrow[d] \arrow[r] & g \cdot \pi(x \cdot h) \arrow[r] & g \cdot (\pi(x) \cdot h) \arrow[d] \\
% \pi((g \cdot x) \cdot h) \arrow[r]           & \pi(g \cdot x) \cdot h \arrow[r] & (g \cdot \pi(x)) \cdot h          
% \end{tikzcd}
% \]
% }
\end{enumerate}
\end{definition}

\begin{proposition}
The diagrams \eqref{diag: equi-prin-bundle diag 1}, \eqref{diag: equi-prin-bundle diag 2} and \eqref{diag: equi-prin-bundle diag 3} are commutative if the following diagram is commutative
\begin{equation}
\label{diag equivariant bundle}
% https://tikzcd.yichuanshen.de/#N4Igdg9gJgpgziAXAbVABwnAlgFyxMJZABgBpiBdUkANwEMAbAVxiRAB120sAKAc04BjKBBwACHgA8hI8QAsAlArGc6AJ2Rq1pdhy68ARgt0UQAX1LpMufIRQAmclVqMWbTtx78ZosZOPswr6KKuzqyFCkMACOTIymFlbYeAREZACMzvTMrIggAoGyoZ7ShcHKqhraMXEMCZYgGMm2ROmkmdTZbnkFQeIevP7mDU02qQ7tWa65ep69RUNmzjBQfPBEoABmahAAtkgAzNQ4EEgALJ3TbNHDWzv7iI4gJ0hkLjnX5hRmQA
\begin{tikzcd}
{\pi(g\cdot (x\cdot h)) \ar[rr,"\pi(b)"]}          &                              & {\pi((g\cdot x)\cdot h) \ar[d,equal]} \\
{g\cdot \pi(x\cdot h) \ar[r,equal]} \arrow[u, "q"] & g\cdot \pi(x) \arrow[r, "q"] & \pi(g\cdot x) ,                       
\end{tikzcd}
\end{equation}
for every chart $U$, every object $g \in \mathcal{G}(U)$, $x \in \mathscr{P}(U)$ and $h \in H(U)$.
\end{proposition}

\begin{proof}
    As the map $\pi \colon \mathscr P\to \mathscr X$ is representable by schemes, it is a faithful functor \cite[\href{https://stacks.math.columbia.edu/tag/04Y5}{Tag 04Y5}]{stacks-project}. 
    In particular, the diagrams \eqref{diag: equi-prin-bundle diag 1}, \eqref{diag: equi-prin-bundle diag 2} and \eqref{diag: equi-prin-bundle diag 3} are commutative if and only if they are commutative after applying $\pi$.
    
    Consider the diagram \eqref{diag: equi-prin-bundle diag 1} after applying $\pi$:
    \[
    % https://tikzcd.yichuanshen.de/#N4Igdg9gJgpgziAXAbVABwnAlgFyxMJZARgBoAGAXVJADcBDAGwFcYkQAdDtLACl4DmAgOQBKXgA8AFqNEgAvqXSZc+QigDMFanSat2XHv0EjRE0TIVKQGbHgJEyxHQxZtEnbnwGDhkmbJWynZqRFrONK76HoZ8gr7mFnKKwaoOKABMpBkueu4gArH8IuaWKTYq9urI5NqReewmYkXSyda2adVZEbpuBl7xCUltqVWO2bl9HgIABEUi-oHlHWMotTn1UwW+LQFBFSHpyAAsdb3RnkZNZiMHnUSnPVH580PJOjBQAvBEoABmACcIABbJC1EA4CBIMjnF4DABGchojHo8JgjAACpVQh5GDA-jh9oCQWCaJCkFlYf0jFxgcwAPrAWn0HBSOAAYwBwAx8nkSJAKLRmOx6RAAKwAikhPKxNBiBh5MQWipMQGtIZTI4wJZbM53N5cw47KgEBwM0syNR6Kxh3UAvx0ussqQAFYyVDEODnuwAI5EoFyymKgBsmwu80NxtNM0R-pJiFDEI9yu9qqMsZlAaQpyTSETqZAfsz8YA7O6KWH8kWnVnEGXc4gcwXZlwo2airHLUKbfdcQ643K3Q2AByV-p0xnM1kcrkADV5IC71pFdvFksd-1ro4b9ebkZNZurm-jCo9AE5kVgwPkTcx4XjFyApDB6FB2JBr4+cPQsIx3wQ2BoOApCwAkwWLOVlUVYhwUYK8bwgO8HxoZ9X3-T8yR-P8PA-NgIKQC8Gxgsc0z4dVJy1HUZ31PkFEoeQgA
\begin{tikzcd}
(gg')\pi(xh) \arrow[r, "q"] \arrow[dd, "\mu_{\mathscr{X}}"'] & \pi((gg')(xh)) \arrow[rr, "\pi(b)"] \arrow[d, "\pi(\mu_{\mathscr{P}})"'] &                                    & \pi(((gg')x)h) \arrow[d, "\pi(\mu_{\mathscr{P}} \cdot h)"] \arrow[r, Rightarrow, no head] & \pi((gg')x) \arrow[d, "\pi(\mu_{\mathscr{P}})"] \\
                                                             & \pi(g(g'(xh))) \arrow[r, "\pi(g \cdot b)"]                               & \pi(g((g'x)h)) \arrow[r, "\pi(b)"] & \pi((g(g'x))h) \arrow[r, Rightarrow, no head]                                             & \pi(g(g'x))                                     \\
g(g'\pi(xh)) \arrow[r, "g \cdot q"]                          & g \pi(g'(xh)) \arrow[u, "q"] \arrow[r, "g \cdot \pi(b)"]                 & g\pi((g'x)h) \arrow[u, "q"]        &                                                                                           &                                                
\end{tikzcd}
    \]
    where the left commutative block is the diagram \eqref{diag equivariant 1}, and the bottom commutative square expresses that $q$ is a natural transformation.
    Hence, the diagram \eqref{diag: equi-prin-bundle diag 1} is commutative if and only if the outer diagram is commutative.
    After resolving the outer diagram using the relation $\pi(b) \circ q = q$, we get the following diagram 
    \[
    % https://tikzcd.yichuanshen.de/#N4Igdg9gJgpgziAXAbVABwnAlgFyxMJZAZgBoAGAXVJADcBDAGwFcYkQAdDtLACl4DmAgOQBKAB6iQAX1LpMufIRRkAjNTpNW7Lj0GDhkqbPnY8BIuQoaGLNohCCRo3X3EALY3JAYzSy6TqNLbaDgIGrrweol6mihYoAEyBNlr2IAKRBpKeMhowUALwRKAAZgBOEAC2SFYgOBBIqsFpOtx8XFXMAPrAnfQ47nAAxuXAAArS0lI0jPQARjCM4wrmyiCMMKU4Mt4V1UjJ9Y2IZJp2bV29-YMjYwAaUyCzC0srfgkg5VgC7jsmIH2NVONAaSAALC0LmEuMMoBAcAACACOuzKlWBkOOTShoRAqIBQMOoJOdRC6QJlGkQA
\begin{tikzcd}
(gg')\pi(xh) \arrow[d, "\mu_{\mathscr{X}}"'] \arrow[rrr, "q"] &  &                             & \pi((gg')x) \arrow[d, "\pi(\mu_{\mathscr{P}})"] \\
g(g'\pi(xh)) \arrow[rr, "g\cdot q"]                           &  & g\pi((g'x)h) \arrow[r, "q"] & \pi(g(g'x))                    ,              
\end{tikzcd}
    \]
 which is the diagram \eqref{diag equivariant 1}.

    Now consider the diagram \eqref{diag: equi-prin-bundle diag 2} after applying $\pi$:
    \[
    % https://tikzcd.yichuanshen.de/#N4Igdg9gJgpgziAXAbVABwnAlgFyxMJZARgBpiBdUkANwEMAbAVxiRAB120sAKAcx4APHgAsRAcgCU0kAF9S6TLnyEUZAExVajFm07ce-QZNESZ8xdjwEiAZnJb6zVog5de-ISOlm5CkBhWKnakmtROuq76HkaS3r4WAUrWqsjqDuE6Lm4GAl5xUpJ+lso2KAAsGdrOeu6egt5FiYGlqZVh1ZE5HnzG3sVJQWUkpAAMjllsfJyk0UKmhQMtKUSjoRM1UXW9Tf7LwSjp45mbINPss3X5CXvJB8hrx53Z55cGxnJaMFB88ESgADMAE4QAC2SDWIBwECQZGetQMACMitQGHRETAGAAFO5lEAMGAAnADYFgpDpKEwxD2eFbJEAAk4AGMoBAcPSboCQeDEJDoUhKvisGBsqymIiCSBqCIYHQoGxICKpVC6FgGAqCKxEqSeYL+YgKRFsnNkSTubDqPqaQxhaKIOLJdLZfLXIrWJbVerXZqzWSDZaqQBWVG2thiiXukAyuUapUetWxrX+HVIGn6gBsIaVrnDjqjzsTypwnsL2vNiGDlKQmdp3R4prLfoAHAGLULsyBc5Hoy7wD7416+0rGzyAJytxCCo1sACOvp5AHYJ5Dp645yOkEuq4hx-mY964+27Q7I8WEwek1y-cQ+VStz3C1njxGiyWL-PYbekC3a+uKLIgA
\begin{tikzcd}
{g\,\pi(x)} \arrow[r, Rightarrow, no head] \arrow[dd, "q"] & {g\,\pi(x(hh'))} \arrow[d, "q"] \arrow[r, Rightarrow, no head]    & {g\,\pi((xh)h')} \arrow[d, "q"]    &                                                                            &                                \\
                                                           & \pi(g(x(hh'))) \arrow[d, "\pi(b)"] \arrow[r, Rightarrow, no head] & \pi(g((xh)h')) \arrow[r, "\pi(b)"] & \pi((g(xh))h') \arrow[d, "\pi(b \cdot h')"] \arrow[r, Rightarrow, no head] & \pi(g(xh)) \arrow[d, "\pi(b)"] \\
\pi(gx) \arrow[r, Rightarrow, no head]                     & \pi((gx)(hh')) \arrow[rr, Rightarrow, no head]                    &                                    & \pi(((gx)h)h') \arrow[r, Rightarrow, no head]                              & \pi((gx)h)                    
\end{tikzcd}
    \]
    where the left commutative block is diagram \eqref{diag equivariant bundle}.
    Hence, the diagram \eqref{diag: equi-prin-bundle diag 2} is commutative if and only if the outer diagram is commutative.
    After resolving the upper row using the relation $\pi(b) \circ q = q$, we get the following diagram
    \[
    % https://tikzcd.yichuanshen.de/#N4Igdg9gJgpgziAXAbVABwnAlgFyxMJZAJgBoAGAXVJADcBDAGwFcYkQAdDtLACgHNeADwAWASjEgAvqXSZc+QijIBGanSat2XHrwFCx46bJAZseAkXIV1DFm0Qh+XUjr4Hjc84quk1NOy1HN31JGS8FSxQVGwDNBycAAhcQ0TD1GCh+eCJQADMAJwgAWyRrEBwIJBiNe21uPgAjSRpGekaYRgAFeQslEAKsfhEcTxBCkqQyCqrEAGY4uscARxBW9s6e7yiBoZGxidL5mkrqmhEYeih2SDA2Vqw79igIZkbGe4r6LEYbgjZwuMikdpqdEAAWB5PRwvN4fNYgC5XP5PE7fX6OW4AkyHJCQmZlRZBECrKSUKRAA
\begin{tikzcd}
{g\,\pi(x)} \arrow[d, "q"'] \arrow[r, Rightarrow, no head] & {g \,\pi(xh)} \arrow[r, "q"] & \pi(g(xh)) \arrow[d, "\pi(b)"'] \\
\pi(gx) \arrow[rr, Rightarrow, no head]                    &                              & \pi((gx)h)                     ,
\end{tikzcd}
    \]
    which is diagram \eqref{diag equivariant bundle}.

    Now consider the diagram \eqref{diag: equi-prin-bundle diag 3} after applying $\pi$:
    \[
    % % https://tikzcd.yichuanshen.de/#N4Igdg9gJgpgziAXAbVABwnAlgFyxMJZARgBpiBdUkANwEMAbAVxiRAB120sAKGAfWCcAtnRwALAMaNgAcQC+8zpKgQcAAh4APdctUaBwABLyAlKZDzS6TLnyEUAJnJVajFm07dtu9irXqhiYWVjbYeAREAMwu1PTMrIgcXLxaIdYgGOH2RGSOrvEeSV68fIIiYlIyCvK+-hppdfqB-EbpYXaRTqT5ce6Jyd6GFRLSDHKKTQFplhlZnQ7IMb1uCZ4p2u2ZthGLZAAMBf1sw+yio9WTegElPtcGrVvzu0T7pId9a0mn51XjNVMNLcZqFttkusg3itCgNfENymdKmMJrV7uoQa4YFAAObwIigABmACcIMIkG8QDgIEgyKsioNeJwYDg6AjfnBJETgAAFRQWagMOgAIxgDG5OxySQYMAJOFmhJJZMQFKpSBidIGtyF-JAgpFYolXRARKw2PEctBxNJNOoqsQzl1wtF4vBDl1Mrl1HEMDoUDYkDArAFWEDbFUTCF0pAtroWAY-oIrEtiqQDrtAFZPkUwEwGAwBU6Da62NLZdHdSGBuHI0GQN7fQnQ8nrYh1XaACwF-UuhYlj3l+t+pIB2sMSthiARqMxuONpMZK1KzuU6mITOO7uGt2lz0a9beJkstmVDlc3lmeUgRdIABsttXFJhbAAjuW9c6t2wTWaLQuU4gAHZ71vL0fSHcBEzfcckmradKVjeNh0g5slSAlckAADizAZXy7D9iySb9zUva9ECw9DW1AhskNDYNQxgyca3LFlZxopMKHkIA
\begin{tikzcd}
e_{\mathcal{G}} \cdot \pi(x) \arrow[r, Rightarrow, no head] \arrow[dd, "q"'] & e_{\mathcal{G}} \cdot \pi(x \cdot e_H) \arrow[d, "q"']                                              &                                                                   &                                       \\
                                                                             & \pi(e_{\mathcal{G}}\cdot (x \cdot e_{H})) \arrow[r, "\pi(\eta_{\mathscr{P}})"] \arrow[d, "\pi(b)"'] & \pi(x \cdot e_{H}) \arrow[r, Rightarrow, no head]                 & \pi(x) \arrow[d, Rightarrow, no head] \\
 \pi(e_{\mathcal{G}} \cdot x) \arrow[r, Rightarrow, no head]                 & \pi((e_{\mathcal{G}} \cdot x) \cdot e_H) \arrow[r, Rightarrow, no head]                             & \pi(e_{\mathcal{G}} \cdot x) \arrow[r, "\pi(\eta_{\mathscr{P}})"] & \pi(x)                               
\end{tikzcd}
    \]
    where the left commutative block is diagram \eqref{diag equivariant bundle}.
    Hence, the diagram \eqref{diag: equi-prin-bundle diag 2} is commutative if and only if the outer diagram is commutative.
    After resolving the upper row using the relation $\pi(\eta_\mathscr{P}) \circ q = \eta_{\mathscr{X}}$, we get the following diagram
    \[
    % % https://tikzcd.yichuanshen.de/#N4Igdg9gJgpgziAXAbVABwnAlgFyxMJZARgBoAGAXVJADcBDAGwFcYkQYB9YAHR4Ft6OABYBjJsADiAX2kACPqKgQcCnmiwAKAB5qlKuVwASAShDTS6TLnyEUAJgrU6TVuz4ade5aq7Aj0mYWVth4BETkpMTODCxsiCBqnn58giLijFKy3gbaQZYgGKG2EU40sW4JKQJCYhIy8oo+SVp55gVFNuEOUTGu8SAerUHOMFAA5vBEoABmAE4Q-EiRIDgQSGQuce48MDj0nKm1cKJzwAAa0iA0jPQARjCMAArWYXYgjDAzOO2zC0uIADMNDWyxuWDAA2UzDun2uIGEMHoUHYkEh8P2WEYqIIbGCIHmiyQwNW60Qji2lRAAEd4bcHs9XiUEnMsONhD98YSARTQYgACzlfo7Tx8PYHXg1EQnM5PWT5P5ExCbPmChFIlEJNFscHohLQ2E61b0LE49HSSjSIA
\begin{tikzcd}
e_{\mathcal{G}} \cdot \pi(x) \arrow[r, Rightarrow, no head] \arrow[d, "q"'] & e_{\mathcal{G}} \cdot \pi(x \cdot e_H) \arrow[r, "\eta_\mathscr{X}"] & \pi(x \cdot e_{H}) \arrow[d, Rightarrow, no head] \\
 \pi(e_{\mathcal{G}} \cdot x) \arrow[rr, "\pi(\eta_{\mathscr{P}})"]         &                                                                      & \pi(x)                                        
\end{tikzcd}
\]
which is diagram \eqref{diag equivariant 2}.
\end{proof}

\begin{definition}
\label{defn: equi morphism of principal bundles}
Let $\pi \colon \mathscr{P} \to \mathscr{X}$ and $\pi' \colon \mathscr{Q} \to \mathscr{X}$ be principal $H$-bundles over a stack $\mathscr{X}$.   
A \emph{morphism of $\mathcal{G}$-equivariant principal $H$-bundles} over $\mathscr{X}$ is a morphism of principal $H$-bundles $f \colon \mathscr{P} \to \mathscr{Q}$ which is also $\mathcal{G}$-equivariant such that for every chart $U$, every object $p \in \mathscr{P}(U)$, and every $g \in  \mathcal{G}(U)$ the following diagram commutes:
\[
% https://tikzcd.yichuanshen.de/#N4Igdg9gJgpgziAXAbVABwnAlgFyxMJZABgBpiBdUkANwEMAbAVxiRAB120sAKAcwAEnAMZQIOAWgCUIAL6l0mXPkIoATOSq1GLNp24ByHgDN+Q9qPGSpM+Yux4CRDQEYt9Zq0QcuWI4JExCVNpKU5SajsQDAcVIjI3ag9dbwCLIPNuHmk5BWilR1VkF1JE7U82NMsJfT8TbJs5LRgoPngiUGMAJwgAWyQyEBwIJBLylJ8ALxgcOlzOnv7EAGZqYaQAFiSdLxAqjM5p2fmQbr7RtZHEDXHdgEcAfWMTs6XB9ZWo183LpDVZCiyIA
\begin{tikzcd}
\pi(g \cdot p) \arrow[rr, "\zeta"] \arrow[d] &                              & \pi'(f(g \cdot p)) \arrow[d, "q_f"] \\
g \cdot \pi(p) \arrow[r, "g \cdot \zeta"]    & g \cdot \pi'(f(p)) \arrow[r] & {\pi'(g \cdot f(p))\,,}            
\end{tikzcd}
\]
where $\zeta \colon \pi \Rightarrow \pi' \circ f$ is a 2-isomorphism, and $q_f$ is the $2$-arrow corresponding to the $\mathcal{G}$-equivariant morphism $f$, see diagram \eqref{diag: equivariant map 2-isomorphism}.
\end{definition}

%--------------------------------------------------------
%--------------------------------------------------------
\subsection{Category of $\mathscr{G}$-equivariant principal $H$-bundles}
\label{sec: cat of principal bundles}
%--------------------------------------------------------
%--------------------------------------------------------

The 2-category of $\mathscr{G}$-equivariant principal $H$-bundles over $\mathscr{X}$ is defined as follows:
\begin{itemize}
    \item[(I)] an object is a $\mathscr{G}$-equivariant principal $H$-bundle $\pi \colon \mathscr{P} \to \mathscr{X}$.
    
    \item[(II)] a 1-morphism $(\pi_1 \colon \mathscr{P} \to \mathscr{X} ) \to (\pi_2 \colon \mathscr{Q} \to \mathscr{X} )$ is a pair $(f,\beta)$ where $f \colon \mathscr{P} \to \mathscr{Q}$ is a $\mathscr{G}$-equivariant morphism of principal $H$-bundles and $\beta \colon \pi_1 \to \pi_2 \circ f$ is a 2-morphism. 

    \item[(III)] a 2-morphism between $(f,\beta),\,(f',\beta') \colon (\pi_1 \colon \mathscr{P} \to \mathscr{X} ) \to (\pi_2 \colon \mathscr{Q} \to \mathscr{X})$ is a 2-morphism $\alpha \colon f \to f'$ such that $(\id_{\pi_2} \star \alpha) \circ \beta = \beta'$, i.e., for each object $p$ of $\mathscr{P}$ the following diagram commutes:
    \[
    % https://tikzcd.yichuanshen.de/#N4Igdg9gJgpgziAXAbVABwnAlgFyxMJZARgBoAGAXVJADcBDAGwFcYkQAdDtLAfWIAUaAJQgAvqXSZc+QinKli1Ok1bsuPXgCYBAMyHDREqdjwEiWxcoYs2iTtz47dAcgPCupAHTjlMKADm8ESgugBOEAC2SAogOBBIZCq26hwARjA49LzAaGIgNIz0GYwACtJmciBhWAEAFjjikiDhUTE08UiWyWr2XBlZLrxoTaER0YhJnYjdNr0OmjpcTGh12SK+YkA
\begin{tikzcd}
                                          & \pi_1(p) \arrow[ld, "\beta_{p}"'] \arrow[rd, "\beta'_p"] &                   \\
\pi_2(f(p)) \arrow[rr, "\pi_2(\alpha_p)"] &                                                          & {\pi_2(f'(p))\,.}
\end{tikzcd}
    \]
\end{itemize} 

Note that as $\pi_2$ is representable by schemes, it is faithful by \cite[\href{https://stacks.math.columbia.edu/tag/04Y5}{Tag 04Y5}]{stacks-project}.
Hence, for each object $p$ in $\mathscr{P}$, the morphism $\alpha_p$ is determined by the above diagram.
Consequently, the 2-morphism $\alpha$, if it exists, is unique.

By the category of $\mathscr{G}$-equivariant principal $H$-bundles over $\mathscr{X}$, we mean 
\begin{itemize}
    \item[(I)] an object is again a $\mathscr{G}$-equivariant principal $H$-bundle $\pi \colon \mathscr{P} \to \mathscr{X}$.
    \item[(II)] a morphism is a 2-isomorphism class of 1-morphisms in the 2-category of $\mathscr{G}$-equivariant principal $H$-bundles over $\mathscr{X}$.
\end{itemize}
This is indeed a category by the observation above, see also \cite[\href{https://stacks.math.columbia.edu/tag/04ZQ}{Tag 04ZQ}]{stacks-project}

% Are we considering homotopy category of 2-Category of $\mathscr{G}$-equivariant principal $H$-bundles?

% , it follows that the 2-category of isomorphism classes of $\mathcal{G}$-equivariant principal $H$-bundles on a stack $\mathscr{X}$ over $S$ is a category.

We conclude this section by presenting its main result.

\begin{theorem}
\label{thm-equivalence}
Let $G^{\bullet}=[G^{-1} \xrightarrow{d} G^{0}]$ be a morphism of abelian group schemes over $S$ such that $G^{-1},\,G^0$ are flat and locally of finite presentation over $S$. 
Let $\mathcal{G}$ be the associated Picard stack. 
Let $X$ be a scheme over $S$ and let $G^{0}$ acts on $X$ over $S$. 
Let $H$ be an group scheme over $S$.  
Then there is an equivalence of categories between the category of $\mathcal{G}$-equivariant principal $H$-bundles on $[X/G^{-1}]$ and the category of $G^0$-equivariant principal $H$-bundles on $X$.
\end{theorem}

% \begin{theorem}
% \label{thm-equivalence}
% Let $G^{\bullet}=[G^{-1} \xrightarrow{d} G^{0}]$ be a morphism of abelian group schemes over $S$ such that $G^{-1},\,G^0$ are flat and locally of finite presentation over $S$. 
% Let $\mathcal{G}$ be the associated Picard stack. 
% Let $X$ be a scheme over $S$ and let $G^{0}$ acts on $X$ over $S$. 
% Let $H$ be an group scheme over $S$.  
% Then there is a one-to-one correspondence between the isomorphism classes of $\mathcal{G}$-equivariant principal $H$-bundles on $[X/G^{-1}]$ and the isomorphism classes of $G^0$-equivariant principal $H$-bundles on $X$.
% \end{theorem}

\begin{proof}
Suppose we are given a $G^0$-equivariant principal $H$-bundle $P\to X$. 
Let $\mathscr{X}:=[X/G^{-1}]$, and define $\mathscr{P}:=[P/G^{-1}]$. 
Note that the following square is $2$-cartesian:
\[
% https://tikzcd.yichuanshen.de/#N4Igdg9gJgpgziAXAbVABwnAlgFyxMJZARgBoAGAXVJADcBDAGwFcYkQAdDgW3pwAs4AYwBOwAAoBfEJNLpMufIRRli1Ok1bsuvAcLEANSV1I1Z87HgJFyFdQxZtEIcTLkgMlpTdJqaDrWcDGXUYKABzeCJQADMRCG4kWxAcCCRicxA4hKQAJhpUpMzsxMQAZgK0xAz3ErzKpDLJSkkgA
\begin{tikzcd}
P \arrow[r] \arrow[d] & \mathscr{P} \arrow[d] \\
X \arrow[r]           & {\mathscr{X}\,,}     
\end{tikzcd}
\]
where, by the 2-Yoneda lemma, $X \to \mathscr{X}$ (similarly $P \to \mathscr{P}$) is the canonical morphism corresponding to the trivial principal $G^{-1}$-bundle over $X$ and the action map of $G$ on $X$.
% One way to prove this is first to show the same diagram at the level of prestacks and then use the fact the stackification of fibered diagram are fibered \cite[\href{https://stacks.math.columbia.edu/tag/04Y1}{Tag 04Y1}]{stacks-project}.
% https://mathoverflow.net/questions/235364/do-equivariant-morphisms-induce-representable-maps-of-quotient-stacks
We first show that the morphism $\mathscr{P} \to \mathscr{X}$ is representable by schemes. 
Suppose $T \to \mathscr{X}$ is a morphism, where $T$ is a scheme. 
Consider the following $2$-cartesian diagram
\begin{equation}
\label{diag: cart-diag for rep-by-schemes}
% https://tikzcd.yichuanshen.de/#N4Igdg9gJgpgziAXAbVABwnAlgFyxMJZAZgBoAGAXVJADcBDAGwFcYkQAdDgW3pwAs4AYwBOwAAoBfEJNLpMufIRRkAjNTpNW7LrwHCxADUldSAOhlyQGbHgJEATBQ0MWbRJx59BoiZID6hpbytkqOpOo0rtoeQbIhivYoqhEuWu4gACoABFx43PD+wLreBsDGktlxVjaJysgpVFHp7Dl5WAVwRSX6vhW5Xr1iUoHB1gp29eSpzW6tY7WTRNNNmnMebRz5hcWDPkaSlT37fjIaMFAA5vBEoABmIhDcSNMgOBBIqvEgD09ITm8PohyN9fs9EGRAZ9QY9wQD3khiDC-ogACw0BEQ5HggCsGKBDmxSDxULRRLR+KQADZySTMQB2cn0ymIGmUSRAA
\begin{tikzcd}
T \times_{\mathscr{X}} \mathscr{P} \arrow[d] & T \times_{\mathscr{X}} P \arrow[r] \arrow[d] \arrow[l] & P \arrow[r] \arrow[d] & \mathscr{P} \arrow[d] \\
T                                            & T \times_{\mathscr{X}} X \arrow[r] \arrow[l]                       & X \arrow[r]                       & {\mathscr{X}\,.}     
\end{tikzcd}
\end{equation}
% Idea for commutative diagram \cite[\href{https://stacks.math.columbia.edu/tag/04XD}{Tag 04XD}]{stacks-project}
By \cite[Theorem 3.1.10 (Algebraicity of Quotient Stacks)]{alper2025stacks}, the morphism $X \to \mathscr{X}$ is a principal $(G^{-1})$-bundle and, in particular, is smooth, surjective and representable by schemes.
Hence, $T \times_{\mathscr{X}} X$ and $T \times_{\mathscr{X}} P$ are schemes.
% \[
% % https://tikzcd.yichuanshen.de/#N4Igdg9gJgpgziAXAbVABwnAlgFyxMJZABgBoBGAXVJADcBDAGwFcYkQAVAAgB0e8AtvAD6wPgPo4AFnADGAJ2AANAL4rePCdLmKACipArS6TLnyEU5CtTpNW7DoeMgM2PASJXiNhizaJODUERMU1JGQVlNS4lJxM3cyIybxpfewDuPmC4UXFwnSj1PO1I-WFYlRsYKABzeCJQADN5CAEkMhAcCCRyIyaWtsQAJhounr6QZtakAGZR7sRiCanBuc6FocqVIA
% \begin{tikzcd}
% T \times_{\mathscr{X}} P \arrow[d] \arrow[r] & T \times_{\mathscr{X}} X \arrow[d] \\
% T \times_{\mathscr{X}} \mathscr{P} \arrow[r]             & T                                 
% \end{tikzcd}
% \]
Since $T \times_{\mathscr{X}} X \to T$ is a surjective smooth morphism of schemes, applying \cite[Proposition 2.3.17 (Descent Criterion for an fppf Sheaf to be a Scheme)]{alper2025stacks} to the left square in diagram \eqref{diag: cart-diag for rep-by-schemes}, it follows that $T \times_\mathscr{X} \mathscr{P}$ is a scheme. 
Hence, $\mathscr{P} \to \mathscr{X}$ is  representable by schemes.

% Since the morphism $X \to \mathscr{X}$ is flat, locally of finite presentation and surjective, by \cite[\href{https://stacks.math.columbia.edu/tag/04ZP}{Tag 04ZP}]{stacks-project} and \cite[\href{https://stacks.math.columbia.edu/tag/04XD}{Tag 04XD}]{stacks-project}, it follows that $\mathscr{P}\to \mathscr X$ is representable by algebraic spaces, flat, locally of finite presentation and surjective.

\noindent
To define the action of $H$ on $\mathscr P$, we first define an action on $H$ on $\mathscr{P}^{\rm pre}:=[P/G^{-1}]^{\rm pre}$ induced by the action of $H$ on $P$:
\begin{align*}
    \mathscr{P}^{\rm pre} \times_S H \to \mathscr{P}^{\rm pre}\,. 
\end{align*}
For a scheme $T \to S$, the action on objects is given by $(p \,, h) \mapsto p \cdot h$, and the action on morphisms  is given by $(p_1 \xrightarrow{g} p_2 \,, h \xrightarrow{\id} h) \mapsto (p_1 \cdot h \xrightarrow{g} p_2 \cdot h)$, where $p \,, p_1 \,, p_2 \in \mathscr{P}(T), h \in H(T)$, and $g \in G^{-1}(T)$. 
Note that the action of $H$ is well defined at the level of morphisms, since
% $$
%     g \cdot (p_1 \cdot h) 
%         = d(g) \cdot (p_1 \cdot h)
%         = (d(g) \cdot p_1) \cdot h
%         = (g \cdot p_1) \cdot h
%         = p_2 \cdot h
% $$
the action of $G^0$ on $P$ commutes with the action of $H$.
By stackification, we get a strict action of $H$ on $\mathscr{P}$.

\noindent
Consider the following 2-cartesian diagram
\begin{equation}
\label{diag: pull-diag principal bundle}
% https://tikzcd.yichuanshen.de/#N4Igdg9gJgpgziAXAbVABwnAlgFyxMJZABgBpiBdUkANwEMAbAVxiRAAUACAHW7wFt4AfQDKnABIgAvqXSZc+QigCM5KrUYs2XXgOEANTu2myQGbHgJEATGur1mrRBxNyLiomWXqHW57346HAALOABjACdgdikePixBOFEJVzN5SyVkVW97TScQAKDQyOipXQThYEKQ8Kj9KVjq4qiY1PMFKxRbHI1HNiba0rb0jxQAZjtevxB9YfdO5Ame33yBkvreUgA6aXUYKABzeCJQADMIiH4kMhAcCCQJqacwJgYGagY6ACMYBnYRzogBgwU44VLnS7Xah3JCqJ5IF5vD7fX7-eZKIEgsEyM4XK6IOEwxAAFlyfUQiPeQJRfwBGOBoPBeNh0PuiFs8Iprypnx+tPRbAZ2NMEPxHKJAFYyX5Kci+WiOvSsUzIYhHkTSZzZdT5XTBcqcSBRUhNZLpc9uXLUXrnEKVWLWUgAGyG42IJ2OxAAdldzMQUtubJ9FCkQA
\begin{tikzcd}
P \times_S H \arrow[d] \arrow[r] & P \times_X P \arrow[d] \arrow[r]                      & P \arrow[d] \arrow[r] & X \arrow[d]      \\
\mathscr{P} \times_S H \arrow[r] & \mathscr{P}\times_{\mathscr{X}} \mathscr{P} \arrow[r] & \mathscr{P} \arrow[r] & {\mathscr{X}\,.}
\end{tikzcd}
\end{equation}
Since the morphism $X \to \mathscr{X}$ is smooth and surjective, by 
% smooth implies flat \cite[\href{https://stacks.math.columbia.edu/tag/01VF}{Tag 01VF}]{stacks-project}
% smooth implies locally of finite presentation \cite[\href{https://stacks.math.columbia.edu/tag/01VE}{Tag 01VE}]{stacks-project} 
\cite[\href{https://stacks.math.columbia.edu/tag/04XD}{Tag 04XD}]{stacks-project}, it follows that $\mathscr{P}\to \mathscr X$ is flat, locally of finite presentation and surjective, as these properties hold for $P \to X$. 
Note that the morphism $P \times_X P \to \mathscr{P} \times_{\mathscr{X}} \mathscr{P}$ is smooth and surjective, since $P \to \mathscr{P}$ is as well.
Hence, by \cite[\href{https://stacks.math.columbia.edu/tag/04XD}{Tag 04XD}]{stacks-project}, it follows that the morphism $\mathscr P\times_S H\to \mathscr P\times_{\mathscr X} \mathscr P$ is an equivalence.
Hence, $\mathscr P\to \mathscr{X}$ is a principal $H$-bundle.

\noindent
% Clearly, there is an action of $\mathcal{G}$ on $\mathscr P$ such that $\mathscr P\to \mathscr{X}$ is $\mathscr{G}$-equivariant.
We now construct an action of $\mathcal{G}$ on $\mathscr{P}$ and $\mathscr{X}$ commuting with the action of $H$, such that the morphism $\mathscr{P} \to \mathscr{X}$ is $\mathcal{G}$-equivariant in the sense of \Cref{defn: equivariant map}.
Following \Cref{sec: action of picard stacks}, we first define the action of 
$\mathcal{G}^{\rm pre}$ on $\mathscr{P}^{\rm pre}$ and $\mathscr{X}^{\rm pre}$ at the prestack level and then pass to its stackification. 
This yields an action of $\mathcal{G}$ on $\mathscr{P}$ and $\mathscr{X}$ such that the morphism $\mathscr{P} \to \mathscr{X}$ is $\mathcal{G}$-equivariant, since the following diagram
\[
% https://tikzcd.yichuanshen.de/#N4Igdg9gJgpgziAXAbVABwnAlgFyxMJZABgBpiBdUkANwEMAbAVxiRAB12BbOnACwDGjYAHEAvgD1gnAE5cABGhkwx8zni7wA+gGU13XnzgCZwAAqTp7OYuViQY0uky58hFGQCMVWoxZtOHn4hBlFLWQUlFX0NbT1Aw2NTAA1w60i7BycQDGw8AiIAJnIfemZWRA4DfiTzNJso+0dnPLci0m9qMv9KhJqTYFSpCNsVBx8YKABzeCJQADMZCC4kMhAcCCRPagY6ACMYBjMXfPcQGSwpvhwshaWVxGL1zcQAZmaQReWt6g2kV66fgqIDoWisQUEwnEpH0ENqqSa2S+DzWf0egPKbFB4MMITCpD6RgGFiaFDEQA
\begin{tikzcd}
\mathcal{G}^{\rm pre} \times_S \mathscr{P}^{\rm pre} \arrow[d] \arrow[rr, "{a_{\mathcal{G},\mathscr{P}}}"] &  & \mathscr{P}^{\rm pre} \arrow[d] \\
\mathcal{G}^{\rm pre} \times_S \mathscr{X}^{\rm pre} \arrow[rr, "{a_{\mathcal{G}, \mathscr{X}}}"]          &  & \mathscr{X}^{\rm pre}          
\end{tikzcd}
\]
commutes strictly.
Similarly, the strict commutativity of the diagram
\[
% https://tikzcd.yichuanshen.de/#N4Igdg9gJgpgziAXAbVABwnAlgFyxMJZABgBpiBdUkANwEMAbAVxiRAB12BbOnACwDGjYAHEAvgD1gnAE5cABGhkwx8zni7wA+gGU13XnzgCZwAAqTp7OYuWr1WTXF3yAEiDGl0mXPkIoyAEYqWkYWNk4efiEGUUtZBSUVfQ1tPUjDY1MLKQTbFQ8vEAxsPAIiACZyEPpmVkQOA34s83jrRLsUxzS3Qu9Sv0rSYOpa8IaM5pNW3Pb8sQ8QmCgAc3giUAAzGQguJDIQHAgkQOoGOgAjGAYzHzL-EBksFb4cEFGw+saovjlgLCgYi0Vh+MTi9nYqTg8jowNcC08Wx2e0QVUOx0QAGYPnU2LDgPC+iBtrsTtQjkhsaFcQ18ZNBMJxKR6S0LAiiiSUQcKaiceMQHSmgzYkyWdM2V0nPofn8AUD3GIKGIgA
\begin{tikzcd}
\mathcal{G}^{\rm pre} \times_S \mathscr{P}^{\rm pre} \times_S H \arrow[d, "\mathrm{id}_{\mathcal{G}} \times a_{H}"'] \arrow[rr, "{a_{\mathcal{G},\mathscr{P}} \times \mathrm{id}_H}"] &  & \mathscr{P}^{\rm pre} \times_S H \arrow[d, "a_{H}"] \\
\mathcal{G}^{\rm pre} \times_S \mathscr{P}^{\rm pre} \arrow[rr, "{a_{\mathcal{G},\mathscr{P}}}"]                                                                                      &  & \mathscr{P}^{\rm pre}                              
\end{tikzcd}
\]
implies that, after stackification, the action of $\mathcal{G}$ and $H$ on $\mathscr{P}$ commutes in the sense of \Cref{def: equivariant bundle}.

Conversely, let $\pi \colon \mathscr{P} \to \mathscr{X}$ be an $\mathcal{G}$-equivariant principal $H$-bundle on $\mathscr{X}$. 
Define $P := \mathscr{P}\times_{\mathscr X}X$. 
% Using diagram~\eqref{diag: pull-diag principal bundle}, it follows from \cite[\href{https://stacks.math.columbia.edu/tag/04XD}{Tag 04XD}]{stacks-project} that the morphism $P \to X$ is a principal $H$-bundle, where the action morphism $P \times_S H \to P$ is obtained by pulling back the action morphism $\mathscr{P} \times_S H \to \mathscr{P}$ along the morphism $P \to \mathscr{P}$.
Then, by \Cref{lemma: pullback of a principal bundle is a principal bundle}, it follows that $P$ is a principal $H$-bundle with the action morphism $P \times_S H \to P$ obtained by pulling back the action morphism $\mathscr{P} \times_S H \to \mathscr{P}$ along the morphism $P \to \mathscr{P}$.

\noindent
We now define an action of $G^0$ on $P$ which commutes with the action of $H$, and for which the morphism $P \to X$ is $G^0$-equivariant.
Consider the following $2$-commutative diagram
\[
% https://tikzcd.yichuanshen.de/#N4Igdg9gJgpgziAXAbVABwnAlgFyxMJZARgBoAGAXVJADcBDAGwFcYkQAdDgW3pwAsAxk2ABxAL4ACLnm7wA+gGVJABWkdZC5QAkQ40uky58hFGWLU6TVuy68BwxmKkyscuEtV6DIDNjwERABMFJYMLGyIIGqu7p66+ob+JsGkFjThNlEqXKQ0ib5GAabI5Glh1pEgogB6wOQuGm5aXgV+xoEoZVQZley15OqaHsoxTXE6epYwUADm8ESgAGYAThDcSGUgOBBIxAWr65s0O0hBB2sbiGTbu4gAzBdHiCG3SI8+h1cALCd3+59LkgAKx-TZPK6gt6Ib7iSjiIA
\begin{tikzcd}
G^0 \times_S P \times_S H \arrow[r] \arrow[d] & \mathcal{G} \times_S P \times_S H \arrow[d] \arrow[r] & P \times_S H \arrow[d] \\
G^{0} \times_S P \arrow[r]                    & \mathcal{G} \times_S P \arrow[r]                      & {P\,,}                
\end{tikzcd}
\]
where the right $2$-commutative square is obtained by pulling back the diagram \eqref{diag: G action commutes with H action} along the morphism $P \to \mathscr{P}$, and the left $2$-commutative square is a pullback square along the natural morphism $G^{0} \times_S P \to \mathcal{G} \times_S P$.
Then the $G^{0}$-action on $P$ is defined by the composition $G^{0}\times_S P \to \mathcal{G} \times_S P \to P$. 
(Note that we still need to show that this is indeed an action. This can be done by pulling back diagram~\eqref{diag: action of a picard stack}, exactly as in the argument we are currently carrying out, which will be completed shortly.)
Since the outer diagram lies in the category of schemes, we get that the outer diagram is commutative.
Hence, the actions of $G^{0}$ and $H$ on $P$ commute.
Similarly, pulling back the diagram \eqref{diag: equivariant map 2-isomorphism} along the morphism $X \to \mathscr{X}$, one can show that $P \to X$ is $G^{0}$-equivariant.

If $P \to Q$ is a $G^0$-equivariant isomorphism of $G^0$-equivariant principal $H$-bundles over $X$, then the following $2$-cartesian diagram
\[
% https://tikzcd.yichuanshen.de/#N4Igdg9gJgpgziAXAbVABwnAlgFyxMJZABgBpiBdUkANwEMAbAVxiRAAUQBfU9TXfIRRkAjFVqMWbZOwD0AcQB6wALQiuFbrxAZseAkRHlx9Zq0QgAilr57Bh0mOqmpF5JYXK1GgDo-S1BrUMFAA5vBEoABmAE4QALZIZCA4EEgATDzRcYmIRilpiADMWSCxCUnUqUjq2uW56VWFJRRcQA
\begin{tikzcd}
P \arrow[r] \arrow[d]  & Q \arrow[d]     \\
{[P/G^{-1}]} \arrow[r] & {[Q/G^{-1}]\,,}
\end{tikzcd}
\]
together with \cite[\href{https://stacks.math.columbia.edu/tag/04XD}{Tag 04XD}]{stacks-project} implies that the morphism $[P/G^{-1}] \to [Q/G^{-1}]$ is an equivalence. 
Furthermore, using stackification arguments as before, one easily sees that this is a $\mathcal{G}$-equivariant equivalence.
The converse follows from arguments similar to those in \Cref{lemma: morphism of principal bundles is an isomorphism}.
\end{proof}

\begin{remark}
\label{rem: trivial bundle under equivalence}
If $\mathscr{X} \times_S H$ is the trivial principal $H$-bundle over $\mathscr{X}$, then its pullback to $X$ is the trivial principal $H$-bundle $X \times_S H$.
Therefore, the equivalence in \Cref{thm-equivalence} sends the trivial bundle to the trivial bundle.
\end{remark}

%--------------------------------------------------------
%--------------------------------------------------------
\section{Equivariant principal bundles on toric stacks and Tits buildings}
\label{sec: toric principal bundles on toric stacks}
%--------------------------------------------------------
%--------------------------------------------------------

In this section, we classify toric principal bundles over a toric Deligne–Mumford stack. 
We begin by recalling the construction of toric Deligne–Mumford stacks.

\begin{definition}
\label{defn: stacky fan}
A \emph{stacky fan} is a triple $\boldsymbol{\Sigma} = (N, \Sigma, \beta)$ with the following components:
\begin{enumerate}
\item $N$ is a finitely generated abelian group.

\item $\Sigma$ is a rational simplicial fan in $N_{\mathbb{Q}} = N \otimes_{\mathbb{Z}} \mathbb{Q}$ with $n$ rays, denoted by $\rho_1,\rho_2,\ldots, \rho_n$, which span $N_{\mathbb{Q}}$.

\item $\beta \colon \mathbb{Z}^n \to N$ is a homomorphism such that for any $1\leq i\leq n$, the element $\beta(e_i)$ is on the ray $\rho_i$, where $e_1,e_2,\ldots, e_n$ is the canonical basis of $\mathbb{Z}^n$.
\end{enumerate}
\end{definition}

We briefly recall the construction of the toric Deligne--Mumford stack associated with a stacky fan $\boldsymbol{\Sigma} = (N,\Sigma, \beta \colon \Z^n \to N)$. 
From the morphism $\beta \colon \mathbb{Z}^n \to N$, we can construct $\beta^{\vee} \colon (\mathbb{Z}^n)^{\vee} \to DG(\beta)$ \cite[\S 2]{BCS}. 
This, in turn, induces a morphism 
$$
    G := {\rm Hom}_{\mathbb{Z}}(DG(\beta),\mathbb{C}^*) \longrightarrow (\mathbb{C}^*)^n.
$$ 
Let ${\mathbb{C}}[z_{1}, \ldots, z_{n}]$ be the coordinate ring of ${\mathbb{A}}^n$. 
Let
\[
    J_\Sigma := 
    \left\langle \prod\limits_{\rho_i\notin \sigma (1)}z_{i} : \sigma \in \Sigma \right\rangle,
\]
where $\sigma(1)$ denotes the collection of rays of $\sigma$, for $\sigma\in \Sigma$.
Define 
$
    Z := \mathbb{A}^n \setminus V(J_\Sigma).
$ 
Then $Z$ is a smooth toric variety with the standard torus $(\C^*)^n \subset \mathbb A^n$, see \cite[Remark 7.6]{FMN}. 
% Z is smooth, see \cite[Page 198]{BCS}
Then the toric Deligne--Mumford stack associated with the stacky fan $\boldsymbol{\Sigma}$ is defined as $\mathscr X(\boldsymbol{\Sigma}):=[Z/G]$. 
Note that the stacky torus $\mc T=[(\C^*)^n/G]$ acts on $\mathscr X(\boldsymbol{\Sigma})$.
The toric Deligne--Mumford stack $\mathscr{X}(\boldsymbol{\Sigma})$ is a toric orbifold if and only if $N$ is free (see \cite[Lemma 7.15]{FMN}),
%If the torsion $N_{\rm tor} = 0$, then $\mathscr X(\boldsymbol{\Sigma})$ is a toric orbifold \cite[Lemma 7.15]{FMN}, 
and the toric Deligne Mumford structure of $\mathscr{X}(\boldsymbol{\Sigma})$ (in the sense of \cite[Definition 3.1]{FMN}) uniquely determines the stacky fan $\boldsymbol{\Sigma}$ \cite[Theorem 7.17]{FMN}. 
Otherwise, it is not unique. 
Following \cite{FMN}, we assume throughout that a Deligne--Mumford stack is separated and of finite type over $\mathbb{C}$.

For a ray $\rho\in \Sigma$, let $v_{\rho}\in N/N_{\rm tor}$ be the primitive generator of $\rho$ in $N/N_{\rm tor}\otimes \Q=N_{\Q}$. 
We define $a_{\rho}\in \N$ by $\beta(e_i) = a_{\rho_i}v_{\rho_i}$. 
Geometrically, $a_{\rho}$ is the cardinality of the generic stabilizer of the divisor corresponding to $\rho$ in the rigidification  $\mathscr{X}(\boldsymbol{\Sigma})^{\rm rig}$ of $\mathscr{X}(\boldsymbol{\Sigma})$ \cite[Remark 7.19]{FMN}. 
Note that this also gives an intrinsic characterization of $a_{\rho}$ (i.e., only depending on the toric structure of $\mathscr X(\boldsymbol{\Sigma})$ and not depending on the stacky fan $\boldsymbol{\Sigma}$).

Suppose $z_0$ is a fixed point in the open orbit of $Z$. 
% It provides an identification of the torus $(\C^*)^n$ with the open orbit via $t \mapsto t \cdot z_0$.
Since, by \Cref{thm-equivalence}, the category of $\mc{T}$-equivariant principal $H$-bundles on $\mathscr{X}(\boldsymbol{\Sigma})$ is equivalent to the category of $(\C^*)^n$-equivariant principal $H$-bundles on $Z$. 
We may introduce the following definition:

\begin{definition}
\label{defn: framed principal bundle}
A \emph{framed $\mathcal{T}$-equivariant principal $H$-bundle} $\mathscr{P}$ over $\mathscr{X}(\boldsymbol{\Sigma})$ is the framed $(\mathbb{C}^*)^n$-equivariant principal $H$-bundle $\mathscr{P}_Z := \mathscr{P} \times_{\mathscr{X}(\boldsymbol{\Sigma})} Z$ on $Z$, i.e., $\mathscr{P}_Z$ together with a choice of a point $p_0$ in the fiber $(\mathscr{P}_Z)_{z_0}$.
\end{definition}

A framing of a $\mathcal{T}$-equivariant principal $H$-bundle $\mathscr{P}$ on $\mathscr{X}(\boldsymbol{\Sigma})$ can also be viewed as a choice of a point $[p_0]$ in the fiber $(\mathscr{P})_{[z_0]}$, where $[z_0]$ denotes the image $z_0$ under the map $(\C^*)^n \to \mathcal{T}$.
This can be seen from the following 2-cartesian cube:
\[
% https://tikzcd.yichuanshen.de/#N4Igdg9gJgpgziAXAbVABwnAlgFyxMJZABgBoAWAXVJADcBDAGwFcYkQAtEAX1PU1z5CKAEwVqdJq3YAdGQFt6OABZwAxgCdgADW4AKOQCMIjKHACe8442ByAylgDmi7gEoefEBmx4CRMiISDCxsiCAGMgDCAHoAVK7RhLz8PkJEYoE0wdJhcooqakzAACrcHimCfigAzKTVQVKhIHlKqprAAAplyV4CvsLItQCMDSGyCq0a8sB2aDBq+nKR7j3elQNDpCNZjeP5bVpdAPocAARyePLwR7YTKuqH3Nznd8pTM3MLEcvlvalVyE29R2Y1yrwenW4J1+a36RE2xFGOWar3es3miyiKwkMCgjngRFAADMNBB5EgyCAcBAkEMeiSyUhyDRqbT6aTyYgAKwsmmIWqSUEgZAALyOiN+DM5ADZeUgROzGYgAOxyxDERWcgWsxCbKn0LCMdjKCAQADWko58rVlJwBqNYRN5stStlVL5ys1SB57qZIORyDQ4soLplaq5XsQYl9-MjAA41dHDDAwFAkABaaqUxj0ZOMDp9NJhDROZQ4ED+ppi4ihpAJmM+5OppBZyvjLBp+Nqt2MLBgJpQehwZS42sqtXkbiUbhAA
\begin{tikzcd}
                                     & \mathrm{Spec}(\C) \arrow[ldd, "z_0"', bend right] \arrow[rrd, "\id", bend left] \arrow[d, dashed] &                                  &                                                               \\
                                     & \mathscr{P}_Z \times_{\mathscr{P}} \mathrm{Spec}(\C) \arrow[ld] \arrow[dd] \arrow[rr]             &                                  & \mathrm{Spec}(\C) \arrow[ld, "{[z_0]}"] \arrow[dd, "{[p_0]}"] \\
(\C^*)^n \arrow[dd, hook] \arrow[rr] &                                                                                                   & \mathcal{T} \arrow[dd, hook]     &                                                               \\
                                     & \mathscr{P}_Z \arrow[ld] \arrow[rr]                                                               &                                  & \mathscr{P} \arrow[ld]                                        \\
Z \arrow[rr]                         &                                                                                                   & \mathscr{X}(\boldsymbol{\Sigma}) &                                                              
\end{tikzcd}
\]
and the universal property of pullbacks.
Intuitively, we pick the unique point lying in the intersection of the fiber $(\mathscr{P}_Z)_{z_0}$ and the orbit $[p_0] = G \cdot p_0 \subset \mathscr{P}_Z$.
Hence, we will denote a framing of a $\mathcal{T}$-equivariant principal $H$-bundle $\mathscr{P}$ over $\mathscr{X}(\boldsymbol{\Sigma})$ by either $(\mathscr{P},[p_0])$ or $(\mathscr{P}_Z,p_0)$, depending on the situation.

% \begin{remark}
% Note that a framing of a $\mathcal{T}$-equivariant principal $H$-bundle $\mathscr{P}$ on $\mathscr{X}(\boldsymbol{\Sigma})$ gives a point $[z_0] \in \mathcal{T}$ and a point in the fiber $[p_0] \in (\mathscr{P})_{[z_0]}$.
% Conversely, 
% % in certain cases, a framing can be obtained from a point in the stacky torus together with a point in the fiber above it.
% % Let $\mathscr{X}(\boldsymbol{\Sigma})$ be a Deligne--Mumford orbifold with Deligne--Mumford torus $\mathcal{T}$.
% % Note that $\mathcal{T}$ is an ordinary torus.
% suppose $[z_0]$ is a fixed point in open orbit of $\mathscr{X}(\boldsymbol{\Sigma})$, that is, a point in the image of the inclusion $\mathcal{T} \hookrightarrow \mathscr{X}(\boldsymbol{\Sigma})$.
% If there exists a section $s \colon \mathcal{T} \to (\C^*)^n$
% % of the geometric quotient $(\C^*)^n \to \mathcal{T}$
% , then a framing of $\mathcal{T}$-equivariant principal $H$-bundle $\mathscr{P}$ on $\mathscr{X}(\boldsymbol{\Sigma})$ is a point $[p_0]$ in the fiber $(\mathscr{P})_{[z_0]}$.
% This can be seen by replacing the map $z_0 \colon \mathrm{Spec}(\C) \to (\C^*)^n$ with $s \circ z_0 \colon \mathrm{Spec}(\C) \to (\C^*)^n$ in the above diagram.
% Note that the framing we obtain here is with respect to the point $s([z_0]) \in Z$.
% If needed, we can use the $(\C^*)^n$-action to a obtain a framing over a fixed point $z_0 \in Z$.
% \end{remark}

\begin{theorem}
\label{thm: equivalence with piecewise linear maps}
Let $H$ be a connected reductive algebraic group over $\C$. 
There is a one-to-one correspondence between the isomorphism classes of framed $\mathcal{T}$-equivariant principal $H$-bundles on $\mathscr{X}(\boldsymbol{\Sigma})$ and piecewise linear maps $\Phi \colon |\Sigma| \to \widetilde{\mathfrak{B}}(H)$ such that $\Phi(a_{\rho} v_{\rho})$ is a lattice point of the building, where $\widetilde{\mathfrak {B}}(H)$ is the cone over the Tits building of $H$.
\end{theorem}

% \begin{theorem}
% There is an equivalence of categories between the category of framed $\mathcal{T}$-equivariant principal $H$-bundles on $\mathscr{X}(\boldsymbol{\Sigma})$ and the category of piecewise linear maps $\Phi \colon |\Sigma| \to \widetilde{\mathfrak {B}}_{\rm sph}(H)$ such that $\Phi(a_{\rho} v_{\rho})$ is a lattice point of the building.
% \end{theorem}

% Here $\widetilde{\mathfrak {B}}(H)$ is the cone over the Tits building of $H$ (\cite[Definition 1.5]{MR4492498}). 
% See \cite[Definition 2.1]{MR4492498} for the definition of piecewise linear maps to buildings.

\begin{proof}
By \Cref{thm-equivalence} and \Cref{defn: framed principal bundle}, the category of framed $\mc{T}$-equivariant principal $H$-bundles on $\mathscr{X}(\boldsymbol{\Sigma})$ is equivalent to the category of framed $(\C^*)^n$-equivariant principal $H$-bundles on $Z$.
By \cite[Proposition 5.1.9]{CLS}, the fan $\widetilde \Sigma$ of $Z$ is given by the collection $\{\widetilde \sigma \,|\, \sigma \in \Sigma \}$, where $\widetilde \sigma$ is the cone generated by $e_i$'s such that $\beta(e_i)\in \sigma(1)$. 
By \cite[Theorem 2.2]{MR4492498}, we get that the category of $(\C^*)^n$-equivariant principal $H$-bundles on $Z$ is equivalent to the category of integral piecewise linear maps from $|\widetilde{\Sigma}| \to \widetilde{\mathfrak{B}}(H)$.
Since $\widetilde{\Sigma}$ is a smooth fan, the integrality condition on the piecewise linear map is equivalent to saying that the images of $e_{\rho}$ are lattice points in the building.
% (Here the integrality condition means that the images of $e_{\rho}$ are lattice points of the building \cite[Definition 1.5, Definition 2.1]{MR4492498}). 
Note that the map $\widetilde \Sigma \to \Sigma$ is in fact a bijection. 
This follows from the fact that $\widetilde \sigma \to \sigma$ is a bijection, since $\Sigma$ is simplicial, i.e., the rays of $\sigma$ are linearly independent. 
Hence, we get a piecewise linear map $|\Sigma|\to \widetilde{\mathfrak{B}}(H)$.  
The last property follows from the fact that 
under $\widetilde{\sigma}\to \sigma$, the basis vectors $e_{\rho}\mapsto a_{\rho}v_{\rho}$. 
This completes the proof.
\end{proof}

Let $a,b \in \Z_{> 0}$ and let $d = \gcd(a,b)$.
The \emph{stacky projective line} is defined as the quotient stack
$$
    \mathbb{P}(a,b) 
        := \left[\left(\mathbb{C}^{2}\setminus \{(0,0)\}\right)/\mathbb{C}^*\right]\,
$$
where the action of $\C^*$ is given by
$$
    t \cdot (z_0,z_1) = (t^{a}z_0,t^bz_1)
$$
for any $t \in \C^*$ and $(z_0,z_1) \in \C^{2}\setminus \{(0,0)\}$.
The stack $\mathbb{P}(a,b)$ is a complete toric Deligne--Mumford stack with Deligne--Mumford torus $\mathcal{T} = \left[(\C^*)^2/\C^*\right] \simeq \C \times B\mu_d$, where $\mu_d$ is the finite cyclic group of order $d$.

\noindent
The projective stack $\mathbb{P}(a,b)$ is described by the following stacky fan $(N,\Sigma,\beta)$:
\begin{itemize}
\item The lattice $N$ is $\mathbb{Z}\oplus \mathbb{Z}/d\mathbb{Z}$.

\item The fan $\Sigma$ is the complete fan in $N_{\mathbb Q}\cong \mathbb Q$ with two rays $\rho_1 = \langle 1 \rangle = \mathbb Q_{\ge 0}$ and $\rho_2 = \langle -1 \rangle = \mathbb Q_{\le 0}$, together with the zero cone.

\item The map $\beta \colon \mathbb{Z}^2 \to N$ is given by $\beta(e_1)=(b',1)$ and $\beta(e_2)=(-a',1)$, where $a' = a/d$ and $b' = b/d$. 
Note that we can take $\beta(e_2) = (-a',0)$ as well.
\end{itemize}

\begin{corollary}
The set of isomorphism classes of (framed) $\mathcal{T}$-equivariant line bundles over the stacky projective line $\mathbb{P}(a,b)$ is in bijection with the set $\Z \oplus \Z$.
\end{corollary}

\begin{proof}
By the preceding theorem, the set $S$ of isomorphism classes of framed $\mathcal{T}$-equivariant line bundles over $\mathbb{P}(a,b)$ is in bijection with the set of piecewise linear maps $\Phi \colon |\Sigma| \to \B(\Gm)$ such that $\Phi(b')$ and $\Phi(a')$ are lattice points of the building.
By \cite[Example 1.14]{MR4492498}, the cone over the Tits building $\B(\Gm)$ is isomorphic to $\R$ with its lattice points are identified with the integers $\Z$.
Since there are no compatibility conditions to check, the set $S$ is in bijection with $\Z \oplus \Z$.

Let $L$ be a $\mathcal{T}$-equivariant line bundle over $\mathbb{P}(a,b)$.
We claim that any two framings of $L$ are isomorphic as framed equivariant line bundles.
A change of framing of $L$ can be realized by a bundle homomorphism of $L$ with respect to a conjugation homomorphism $\alpha$ of the group $\Gm$, see \Cref{rem: change of framing}.
Hence, by \Cref{thm: main-theorem kaveh-manon}, the piecewise linear map corresponding to two different framings of $L$ are related by composition with $\widehat{\alpha}$.
Since $\Gm$ is abelian, $\alpha$ is identity, and hence so is $\widehat{\alpha}$.
Hence, the piecewise linear maps corresponding to two different frames of $L$ are identical.
Therefore, they are isomorphic are framed equivariant line bundles. 
Thus, the set of isomorphism classes of $\mathcal{T}$-equivariant line bundles over $\mathbb{P}(a,b)$ is also in bijection with $\Z \oplus \Z$.
\end{proof}

%--------------------------------------------------------
%--------------------------------------------------------
\section{Equivariant automorphism group}
%--------------------------------------------------------
%--------------------------------------------------------

Let $H$ be a connected reductive algebraic group over $\C$.
Recall from \Cref{sec: one-para-subgroups and lattice points} that the lattice points $\B_{\Z}(H)$ in the cone over the Tits building of $H$ can be identified with the set of equivalence classes of one-parameter subgroups of $H$.
Furthermore, given an equivalence class of one-parameter subgroups of $H$, represented by $\lambda \colon \mathbb{G}_m \to H$, we can associate it with the parabolic subgroup
$$
    P_{\lambda} 
        = \left\{ h \in H 
        \mid h \lambda h^{-1} \sim \lambda \right\}.
$$
Hence, each point in $\B_{\Z}(H)$ determines a parabolic subgroup of $H$.
With this setup, we can now state the following result, which generalizes \cite[Theorem 5.5]{MR4960069}.

\begin{theorem}
Let $H$ be a connected reductive algebraic group over $\C$, and let $\scrP$ be a $\mathcal{T}$-equivariant principal $H$-bundle over $\scrXE$.
Let $(\scrP_Z,p_0)$ be a framing of $\scrP$, and let $\Phi \colon |\Sigma| \to \widetilde{\mathfrak{B}}(H)$ be the corresponding piecewise linear map, see \Cref{thm: equivalence with piecewise linear maps}. 
If $\Aut_{\mathcal{T}}(\scrP)$ is  the group of $\mathcal{T}$-equivariant automorphisms of $\scrP$ in the sense of \Cref{defn: equi morphism of principal bundles}, then
$$
    \Aut_{\mathcal{T}}(\scrP) 
        \cong \bigcap_{\rho \in \Sigma(1)} P_{\rho}\,,
$$
where $P_{\rho}$ is the parabolic subgroup in $H$ corresponding to the lattice point $\Phi(a_{\rho}v_{\rho}) \in \widetilde{\mathfrak{B}}_{\Z}(H)$. 
\end{theorem}

\begin{proof}
Following the notations of \Cref{sec: toric principal bundles on toric stacks}, by \Cref{thm-equivalence}, it follows that 
$$
    \Aut_{(\C^*)^n}(\scrP_Z) \cong \Aut_{\mathcal{T}}(\scrP)\,.
$$ 
If $\widetilde{\Phi} \colon |\widetilde{\Sigma}| \to \widetilde{\mathfrak{B}}(H)$ is the piecewise linear map corresponding to the framed bundle $(\scrP_Z,p_0)$, then \cite[Theorem 4.2]{MR4767486} implies that 
$$
    \Aut_{(\C^*)^n}(Z)
        \cong \bigcap_{\rho \in \widetilde{\Sigma}(1)} P_{\rho}
$$
where $P_{\rho}$ is the parabolic subgroup in $H$ corresponding to the lattice point $\widetilde{\Phi}(e_{\rho}) \in \widetilde{\mathfrak{B}}_{\Z}(H)$. 
Since the piecewise linear map $\Phi \colon |\Sigma| \to \B(H)$ corresponding to the framed bundle $\scrP$ is given by the composition $|\Sigma| \to |\widetilde{\Sigma}| \xrightarrow{\widetilde{\Phi}} \widetilde{\mathfrak{B}}(H)$, where $\Sigma \to \widetilde{\Sigma}$ is the bijection which maps $e_{\rho}$ to $a_{\rho}v_{\rho}$, the desired result follows.
\end{proof}

%--------------------------------------------------------
%--------------------------------------------------------
\section{Equivariant reduction of structure group}
%--------------------------------------------------------
%--------------------------------------------------------

%--------------------------------------------------------
%--------------------------------------------------------
%\subsection{Associated bundle}
%--------------------------------------------------------
%--------------------------------------------------------

Let $\mathscr{Q}$ be a principal $K$-bundle over an algebraic stack $\mathscr{X}$.
Suppose $F$ is a scheme with a (right) $K$-action.
We define the associated stack $\mathscr{Q} \times^K F$ as follows:
for each smooth atlas $X \to \mathscr{X}$, we have $(\mathscr{Q} \times^K F)(X) := \mathscr{Q}_X \times^K F$, where 
$$
    \mathscr{Q}_X \times^K F  = \frac{\mathscr{Q}_X \times F}{(q,f) \sim (q \cdot k ,\, k^{-1} \cdot f)}\, \quad \forall\, q \in \mathscr{Q}_X, \, k \in K\,, f \in F\,,
$$
and $\mathscr{Q}_X = \mathscr{Q} \times_{\mathscr{X}} X$.
This construction indeed defines a stack, see \cite[Section 1.4]{Biswas-Majumdar-Wong}.

The proof of the following proposition is similar to the arguments used to show that $\mathscr{P}$ is a principal $H$-bundle over $\mathscr{X}$ in \Cref{thm-equivalence}, and is therefore omitted.

\begin{proposition}
Let $K$ be a closed subgroup of a linear algebraic group $H$ over $\C$. 
Suppose $X \to \mathscr{X}$ is a smooth atlas of an algebraic stack $\mathscr{X}$ representable by schemes.
If $\mathscr{Q}$ be a principal $K$-bundle over $\mathscr{X}$, then $\mathscr{Q} \times^K H$ is a principal $H$-bundle over $\mathscr{X}$ in the sense of \Cref{defn: principal bundle}.
\end{proposition}

\begin{definition}[Equivariant reduction of structure group]
Let $K$ be a closed subgroup of a linear algebraic group $H$ over $\C$. 
We say that a $\mathcal{T}$-equivariant principal $H$-bundle $\mathscr{P}$ over $\mathscr{X}(\boldsymbol{\Sigma})$ admits an \emph{equivariant reduction of structure group to $K$} if there exists a $\mathcal{T}$-equivariant principal $K$-bundle $\mathscr{Q}$ over $\mathscr{X}(\boldsymbol{\Sigma})$ such that $\mathscr{Q} \times^K H$ is isomorphic to $\mathscr{P}$ as $\mathcal{T}$-equivariant principal $H$-bundles.

If $\mathscr{P}$ admits an equivariant reduction of the structure group to a maximal torus of $H$, then we say that $\mathscr{P}$ splits equivariantly.
\end{definition}

\begin{lemma}
\label{lemma: reduction of structure group}
Let $K$ be a closed subgroup of a linear algebraic group $H$ over $\C$. 
A $\mathcal{T}$-equivariant principal $H$-bundle $\mathscr{P}$ over $\mathscr{X}(\boldsymbol{\Sigma})$ has an equivariant reduction of structure group to $K$ if and only if $(\C^*)^n$-equivariant principal $H$-bundle $\mathscr{P}_Z := \mathscr{P} \times_{\mathscr{X}(\boldsymbol{\Sigma})} Z$ over $Z$ has an equivariant reduction of structure group to $K$.
In particular, $\mathscr{P}$ splits equivariantly if and only if $\mathscr{P}_Z$ splits equivariantly.
\end{lemma}

\begin{proof}
Suppose $\mathcal{T}$-equivariant principal $H$-bundle $\mathscr{P}$ over $\mathscr{X}(\boldsymbol{\Sigma})$ has an equivariant reduction of structure group to $K$.
Then there exists a $\mathcal{T}$-equivariant principal $K$-bundle $\mathscr{Q}$ over $\mathscr{X}(\boldsymbol{\Sigma})$ and $\mathcal{T}$-equivariant isomorphism $\mathscr{Q} \times^K H \to \mathscr{P}$ of principal $H$-bundles.
Pulling back this isomorphism along $Z \to \mathscr{X}(\boldsymbol{\Sigma})$, by \Cref{thm-equivalence}, we get $(\C^*)^n$-equivariant isomorphism $\mathscr{Q}_Z \times^K H \to \mathscr{P}_Z$ of principal $H$-bundles over $Z$. 
Hence, $\mathscr{P}_Z$ admits an equivariant reduction of structure group to $K$.

Conversely, suppose $(\C^*)^n$-equivariant principal $H$-bundle $\mathscr{P}_Z$ has an equivariant reduction of structure group to $K$. 
Then there exists a $(\C^*)^n$-equivariant principal $K$-bundle $Q$ over $Z$ and $(\C^*)^n$-equivariant isomorphism $Q \times^K H \to \mathscr{P}_Z$ of principal $H$-bundles.
Then, by \Cref{thm-equivalence}, $\mathscr{Q} := [Q/G]$ is a $\mathcal{T}$-equivariant principal $K$-bundle over $\mathscr{X}(\boldsymbol{\Sigma})$ such that the following diagram is 2-cartesian
\[
% https://tikzcd.yichuanshen.de/#N4Igdg9gJgpgziAXAbVABwnAlgFyxMJZABgBpiBdUkANwEMAbAVxiRAEUQBfU9TXfIRQBGclVqMWbADrSAtnRwALOAGMATsHZduvEBmx4CRMsPH1mrRCABauvocFFRZ6hanXZC5Ws0ANLgAKWQAjCAYoOABPOTCGYFkAZSwAcwUuAEpZUgA6bnEYKBT4IlAAM3UIOSQAJmocCCQAZh5yyurEURAG5taQCqqkMm7Gzr6BjuGexBquCi4gA
\begin{tikzcd}
Q \arrow[r] \arrow[d] & \mathscr{Q} \arrow[d]                 \\
Z \arrow[r]           & {\mathscr{X}(\boldsymbol{\Sigma})\,.}
\end{tikzcd}
\]
Hence, the following diagram is $2$-cartesian
\[
% https://tikzcd.yichuanshen.de/#N4Igdg9gJgpgziAXAbVABwnAlgFyxMJZABgBpiBdUkANwEMAbAVxiRAEUACAHW7wFt4APQDSnABIgAvqXSZc+QigCM5KrUYs2vfnRwALOAGMATsHZSefLILiiJ02SAzY8BImWXr6zVohAAWo5yropEql7UPlr+OnqGpsAAGlIAFLwARhAMUHAAnvxZDMC8AMpYAOa6UgCUvKQAdNLqMFAV8ESgAGYmEPxIAEzUOBBIAMwy3b39iKogI+OTID19SGTzo7NLKzPrC4gDUhRSQA
\begin{tikzcd}
Q \times^K H \arrow[r] \arrow[d] & \mathscr{Q} \times^K H \arrow[d]      \\
Z \arrow[r]                      & {\mathscr{X}(\boldsymbol{\Sigma})\,.}
\end{tikzcd}
\]
Hence, by \Cref{thm-equivalence}, it follows that $\mathscr{Q} \times^K H$ is isomorphic to $\mathscr{P}$ as $\mathcal{T}$-equivariant principal $H$-bundles.
\end{proof}

% As a consequence of the above result and \cite[Theorem 5.4]{MR4767486}, we obtain the following theorem.

\begin{theorem}[Criterion for equivariant reduction of structure group]
Let $H$ be a connected reductive algebraic group over $\C$, and let $K$ be a closed connected reductive subgroup of $H$. 
A $\mathcal{T}$-equivariant principal $H$-bundle $\mathscr{P}$ over $\mathscr{X}(\boldsymbol{\Sigma})$ has an equivariant reduction of structure group to $K$ if and only if there exists a framing $(\mathscr{P}_Z,p_0)$ of $\mathscr{P}$ such that the image of the corresponding piecewise linear map $\Phi \colon |\Sigma| \to \B(H)$, satisfying $\Phi(a_{\rho}v_{\rho})$ is a lattice point of the building, lies in $\B(K)$.
\end{theorem}

% \begin{theorem}[Criterion for equivariant reduction of structure group]
% Let $K$ be a closed subgroup of a linear algebraic group $H$ over $\C$. 
% Let $\mathcal{T}$-equivariant principal $H$-bundle $\mathscr{P}$ over $\mathscr{X}(\boldsymbol{\Sigma})$.
% Then the following are equivalent:
% \begin{enumerate}
% \item $\mathscr{P}$ admits a reduction of structure group to $K$.

% \item $\mathscr{P}_Z$ admits a reduction of structure group to $K$.

% \item There exists a $p_0 \in (\mathscr{P}_Z)_{z_0}$ such that the image of $\Phi_{p_0}$ lies in $\widetilde{\mathfrak {B}}(K)$, where $\Phi_{p_0} \colon |\widetilde{\Sigma}| \to \widetilde{\mathfrak {B}}(H)$ is the integral piecewise linear map corresponding to the framed bundle $(\mathscr{P}_Z,p_0)$ and $\widetilde{\Sigma}$ is the fan of $Z$.

% \item There exists a $[p_0] \in (\mathscr{P})_{[z_0]}$ such that the image of $\Phi_{[p_0]}$ lies in $\widetilde{\mathfrak {B}}(K)$, where $\Phi_{[p_0]} \colon |\Sigma| \to \widetilde{\mathfrak {B}}(H)$ is the piecewise linear map, satisfying $\Phi_{[p_0]}(a_{\rho}v_{\rho})$ is a lattice point of the building, corresponding to the framed bundle $(\mathscr{P},[p_0])$.
% \end{enumerate}
% \end{theorem}

\begin{proof}
% By \Cref{lemma: reduction of structure group}, it suffices to show that $\mathscr{P}_Z$ has an equivariant reduction of structure group to $K$ if and only if there exists a $p_0 \in (\mathscr{P}_Z)_{z_0}$ such that the image of the corresponding piecewise linear map $\Phi \colon |\Sigma| \to \B(H)$ lies in $\B(K)$.
By \cite[Theorem 5.4]{MR4767486}, $\mathscr{P}_Z$ has an equivariant reduction of structure group to $K$ if and only if there exists a framing $(\mathscr{P}_Z, p_0)$ of $\mathscr{P}_Z$ such that the image of the corresponding integral piecewise linear map $\widetilde{\Phi} \colon |\widetilde{\Sigma}| \to \widetilde{\mathfrak {B}}(H)$ lies in $\widetilde{\mathfrak {B}}(K)$, where $\widetilde{\Sigma}$ is the fan of $Z$.
% Following the proof of \Cref{thm-equivalence}, we have $\widetilde{\Sigma} \to \Sigma$ is a bijection; hence we obtain a piecewise linear map with the required properties.
Since the piecewise linear map $\Phi \colon |\Sigma| \to \B(H)$ corresponding to the framing $(\mathscr{P}_Z, p_0)$ of $\scrP$ is given by the composition $|\Sigma| \to |\widetilde{\Sigma}| \xrightarrow{\widetilde{\Phi}} \widetilde{\mathfrak{B}}(H)$, the desired result follows follows from \Cref{lemma: reduction of structure group}.
\end{proof}

The following corollary follows from the fact that the cone over the tits building $\widetilde{\mathfrak{B}}(T)$ of a maximal torus $T$ of $H$ is the extended apartment $\widetilde{A}_T$.

\begin{corollary}[Criterion for equivariant splitting]
\label{cor: equivariant splitting}
Let $H$ be a connected reductive algebraic group over $\C$.
A $\mathcal{T}$-invariant principal $H$-bundle $\mathscr{P}$ over $\mathscr{X}(\boldsymbol{\Sigma})$ splits equivariantly if and only if for some (and hence any) framing $(\mathscr{P}_Z,p_0)$ of $\mathscr{P}$ such that the image of the corresponding piecewise linear map $\Phi \colon |\Sigma| \to \B(H)$, satisfying $\Phi(a_{\rho}v_{\rho})$ is a lattice point of the building, lies in an extended apartment $\widetilde{A}_T$ for some maximal torus $T$ of $H$.
\end{corollary}

% \begin{corollary}[Criterion for equivariant splitting]
% Let $K$ be a closed subgroup of a linear algebraic group $H$ over $\C$. 
% Let $\mathcal{T}$-equivariant principal $H$-bundle $\mathscr{P}$ over $\mathscr{X}(\boldsymbol{\Sigma})$.
% Then the following are equivalent:
% \begin{enumerate}
% \item $\mathscr{P}$ splits equivariantly.

% \item $\mathscr{P}_Z$ splits equivariantly.

% \item For some (and hence any) $p_0 \in (\mathscr{P}_Z)_{z_0}$ the image of $\Phi_{p_0}$ lies in an extended apartment $\widetilde{A}_K$ for some maximal torus $K$ in $H$, where $\Phi_{p_0} \colon |\Sigma| \to \widetilde{\mathfrak {B}}(H)$ is the integral piecewise linear map corresponding to the framed bundle $(\mathscr{P}_Z,p_0)$ and $\widetilde{\Sigma}$ is the fan of $Z$.

% \item For some (and hence any) $[p_0] \in (\mathscr{P})_{[z_0]}$ the image of $\Phi_{[p_0]}$ lies in an extended apartment $\widetilde{A}_K$ for some maximal torus $K$ in $H$, where $\Phi_{[p_0]} \colon |\Sigma| \to \widetilde{\mathfrak {B}}(H)$ is the piecewise linear map, satisfying $\Phi_{[p_0]}(a_{\rho}v_{\rho})$ is a lattice point of the building, corresponding to the framed bundle $(\mathscr{P},[p_0])$.
% \end{enumerate}
% \end{corollary}

\begin{corollary}
\label{cor: if associated bundle splits then bundle splits}
Let $H$ be a connected reductive algebraic group over $\C$, and let $K$ be a closed connected reductive subgroup of $H$. 
Let $\mathscr{Q}$ be a $\mathcal{T}$-equivariant principal $K$-bundle $\mathscr{P}$ over $\mathscr{X}(\boldsymbol{\Sigma})$. 
If $\mathscr{Q} \times^K H$ splits equivariantly as a $H$-bundle, then $\mathscr{Q}$ splits equivariantly as a $K$-bundle.
\end{corollary}

\begin{proof}
Suppose $\mathscr{Q} \times^K H$ splits equivariantly as a $H$-bundle.
Equivalently, by \Cref{lemma: reduction of structure group}, $\mathscr{Q}_Z \times^K H$ splits equivariantly as a $H$-bundle.
Then, by \cite[Corollary 5.6]{MR4767486}, it follows that $\mathscr{Q}_Z$ splits equivariantly as a $K$-bundle.
Hence, again by \Cref{lemma: reduction of structure group}, we obtain that $\mathscr{Q}$ splits equivariantly as a $K$-bundle.
\end{proof}

% NOTE FOR REDUCTIVE GROUPS: THE LEVI DECOMPOSITION IS TRIVIAL SINCE R_u(G) = \{e\}. 

% \begin{corollary}[Equivariant reduction of structure group to a Levi]
% Suppose $G$ is an algebraic group over $\C$ admitting a semidirect product decomposition $G = P \ltimes U$, where $U$ is unipotent.
% Let $\scrP$ be an $\mathcal{T}$-equivariant principal $G$-bundle over $\scrX(\boldsymbol{\Sigma})$, then $\scrP$ admits an equivariant reduction of structure group to $L$.
% In particular, $\scrP$ admits an equivariant reduction of structure group to a Levi subgroup.
% \end{corollary}

% \begin{proof}
% By \cite[Corollary 5.8]{MR4767486}, the pullback bundle $\scrP_Z$ admits equivariant reduction of structure group to $L$. 
% Hence, by \Cref{lemma: reduction of structure group}, $\scrP$ admits equivariant reduction of structure group to $L$.
% \end{proof}

\begin{example}
\label{ex: low rank bundle splits over weighted projective stack}
Let $w_0,w_1,\dots,w_n \in \Z_{>0}$. 
Consider the weighted projective stack
$$
    \mathbb{P}(w) 
        := \left[\left(\mathbb{C}^{n+1}\setminus \{(0,\dots,0)\}\right) / \mathbb{C}^*\right]\,
$$
where the action of $\C^*$ is given by
$$
    t \cdot (z_0,z_1,\dots,z_n) = (t^{w_0}z_0,t^{w_1}z_1,\dots,t^{w_n}z_n)
$$
for any $t \in \C^*$ and $(z_0,z_1,\dots,z_n) \in \mathbb{C}^{n+1}\setminus \{(0,\dots,0)\}$.
The stack $\mathbb{P}(w)$ is a complete toric Deligne--Mumford stack with Deligne--Mumford torus $\mathcal{T} = \left[(\C^*)^{n+1}/\C^*\right] \simeq \C^n \times B\mu_d$, where $\mu_d$ is the finite cyclic group of order $d = \gcd(w_0,w_1,\dots,w_n)$.

Let $\mathscr{P}$ be a $\mathcal{T}$-equivariant principal $\GL(r)$-bundle over $\mathbb{P}(w)$, where $r < n$.
Let $\mathscr{P}_Z$ be the corresponding $(\C^*)^{n+1}$-equivariant principal $\GL(r)$-bundle over $Z = \mathbb{C}^{n+1}\setminus \{(0,\dots,0)\}$.
Since $r <n$, it follows from \cite[Corollary 6.1.4]{Klyachko} that $\mathscr{P}_Z$ splits equivariantly.
Hence, by \Cref{lemma: reduction of structure group}, we obtain that $\mathscr{P}$ splits equivariantly.
Hence, any toric vector bundle over the weighted projective stack $\mathbb{P}(w)$ splits equivariantly if $r < n$. 
Moreover, if $K$ is a closed connected reductive subgroup of $\GL(r)$, then by \Cref{cor: if associated bundle splits then bundle splits} it follows that any toric principal $K$-bundle over the weighted projective stack $\mathbb{P}(w)$ splits equivariantly if $r < n$. 
\end{example}

\begin{example}[Equivariant bundles on the stacky projective line $\mathbb{P}(a,b)$]
\label{ex: equivariant bundle on stacky projective line}
% Let $a,b \in \Z_{> 0}$ and let $d = \gcd(a,b)$.
% The \emph{stacky projective line} is defined as the quotient stack
% $$
%     \mathbb{P}(a,b) 
%         := \left[\left(\mathbb{C}^{2}\setminus \{(0,0)\}\right)/\mathbb{C}^*\right]\,
% $$
% where the action of $\C^*$ is given by
% $$
%     t \cdot (z_0,z_1) = (t^{a}z_0,t^bz_1)
% $$
% for any $t \in \C^*$ and $(z_0,z_1) \in \C^{2}\setminus \{(0,0)\}$.
% The stack $\mathbb{P}(a,b)$ is a complete toric Deligne--Mumford stack with Deligne--Mumford torus $\mathcal{T} = \left[(\C^*)^2/\C^*\right] \simeq \C \times B\mu_d$, where $\mu_d$ is the finite cyclic group of order $d$.

% \noindent
% The projective stack $\mathbb{P}(a,b)$ is described by the following stacky fan $(N,\Sigma,\beta)$:
% \begin{itemize}
% \item The lattice $N$ is $\mathbb{Z}\oplus \mathbb{Z}/d\mathbb{Z}$.

% \item The fan $\Sigma$ is the complete fan in $N_{\mathbb Q}\cong \mathbb Q$ with two rays $\rho_1 = \langle 1 \rangle = \mathbb Q_{\ge 0}$ and $\rho_2 = \langle -1 \rangle = \mathbb Q_{\le 0}$, together with the zero cone.

% \item The map $\beta \colon \mathbb{Z}^2 \to N$ is given by $\beta(e_1)=(b',1)$ and $\beta(e_2)=(-a',1)$, where $a' = a/d$ and $b' = b/d$. 
% Note we can take $\beta(e_2) = (-a',0)$ as well.
% \end{itemize}

% \noindent
Let $H$ be a connected reductive linear algebraic group over $\C$, and let $\mathscr{P}$ be a $\mathcal{T}$-equivariant principal $H$-bundle over the stacky projective line $\mathbb{P}(a,b)$.
For any framing $(\mathscr{P}_Z,p_0)$ of $\mathscr{P}$, the corresponding piecewise linear map $\Phi \colon |\Sigma| \to \B(H)$, satisfying $\Phi(a_{\rho}v_{\rho})$ is a lattice point of the building, gives us two cones $\Phi(\rho_1)$ and $\Phi(\rho_2)$.
These cones (corresponding to the parabolic subgroups of $H$) are contained inside the Weyl chambers (corresponding to the Borel subgroups of $H$).
Since the intersection of two Borel subgroups in a reductive group contains a maximal tori (see \cite[Page 173 Corollary]{HumphreysLAG}), it follows that there exists a maximal torus $T$ of $H$ such that the parabolic subgroups corresponding to $\Phi(\rho_1)$ and $\Phi(\rho_2)$ contain $T$.
Hence, $\Phi(| \Sigma |) \subset \widetilde{A}_T$, and hence $\mathscr{P}$ splits equivariantly.
% \noindent
% Let $Z = \C^2 \setminus \{0\}$, and choose $p_0 \in (\mathscr{P}_Z)_{z_0}$.
% The fan $\widetilde{\Sigma}$ of the toric variety $Z$ consists of two rays, $\rho_1 = \langle 1 \rangle$ and $\rho_2 = \langle -1 \rangle$, together with the zero cone in two-dimensional space.
\end{example}

\begin{theorem}
Let $H$ be a linear algebraic group over $\C$.
If $H$ is a nilpotent group, then any $\mathcal{T}$-equivariant principal $H$-bundle over $\scrX(\boldsymbol{\Sigma})$ splits equivariantly. 
In particular, if $H$ is unipotent, then the bundle is trivial, with $\mathcal{T}$ acting trivially on $H$.
\end{theorem}

\begin{proof}
Suppose $\scrP$ is a $\mathcal{T}$-equivariant principal $H$-bundle over $\scrX(\boldsymbol{\Sigma})$.
Since $Z$ is a smooth toric variety, it follows from \cite[Corollary 3.4]{BDP2016} that $\scrP_Z$ splits equivariantly.
Hence, by \Cref{lemma: reduction of structure group}, we obtain $\scrP$ splits equivariantly.

If $H$ is unipotent, then by \cite[Corollary 3.4]{BDP2016}, the principal $H$-bundle $\mathscr{P}_Z$ is trivial, i.e., isomorphic to $X \times H$ with $(\mathbb{C}^*)^n$ acting trivially on $H$. 
Therefore, by \Cref{rem: trivial bundle under equivalence}, the principal $H$-bundle $\scrP$ is isomorphic to the trivial $H$-bundle $\scrX(\boldsymbol{\Sigma}) \times H$, with $\mathcal{T}$ acting trivially on $H$.
\end{proof}

%--------------------------------------------------------
%--------------------------------------------------------
\section{Appendix}
\label{sec: Appendix (Strict Actions)}
%--------------------------------------------------------
%--------------------------------------------------------

In this section, we study the strictification of actions on groupoids. 
We prove that the strictification functor sends equivariant morphisms to strict morphisms.

Let $S$ be a scheme, and let $H$ be a group-valued functor on the category of $S$-schemes.
Let $H\hbox{-}\mathfrak{Grpd}/S$ denote the $2$-category of $H$-groupoids over $S$, see \cite[\S 1.2]{Romagny}. 
We note that this 2-category is defined in the same way as the one introduced in \Cref{sec: action of picard stacks}.

\begin{proposition}[{\cite[Proposition 1.5]{Romagny}}]
The inclusion functor of the $2$-category of strict $H$-groupoids, as a fully faithful sub-2-category of $H\hbox{-}\mathfrak{Grpd}/S$, is a $2$-equivalence.
More precisely, there is a ``strictification'' functor 
$$
    \varphi \colon H\hbox{-}\mathfrak{Grpd}/S \to H\hbox{-}\mathfrak{Grpd}/S
$$
% $H\hbox{-}\mathfrak{Grpd}/S \to H\hbox{-}\mathfrak{Grpd}/S$ 
sending any $H$-groupoid to an isomorphic $H$-groupoid with strict action.
\end{proposition}

Let us describe the functor $\varphi$. 
Let $\widetilde{\mathcal{M}}$ denote the strictification of the $H$-groupoid $\mathcal{M}$ over $S$; that is, $\widetilde{\mathcal{M}} : = \varphi(\mathcal{M})$. 
For any morphism $T \to S$, the groupoid $\widetilde{\mathcal{M}}(T)$ is defined as follows:
\begin{itemize}
\item The objects are
$$
    {\rm{Ob}}(\widetilde{\mc M}) (T) 
        = \{(g,x) \mid g \in {\rm{Ob}}(H(T)) \text{ and } x \in {\rm{Ob}}(\mc{M}(T)) \}.
$$ 

\item A morphism $\widetilde{\phi} \colon (g,x) \to (h,y)$ in $\widetilde{\mc M}(T)$ is represented by a morphism
$$
    \phi \colon x \to (g^{-1}h)\cdot y
$$
in $\mc M (T)$.
The composition $(g,x) \xrightarrow{\widetilde{\phi}} (h,y) \xrightarrow{\widetilde{\psi}} (k,z)$ of morphisms in $\widetilde{\mc M} (T)$, is given by the composite
$$
    x \xrightarrow{\phi} (g^{-1}h)\cdot y 
        \xrightarrow{(g^{-1}h) \cdot \psi} (g^{-1}h)\cdot((h^{-1}k) \cdot z)
        \xrightarrow{\mu_{\mc M}^{-1}} ((g^{-1}h)(h^{-1}k))\cdot z 
        = (g^{-1}k) \cdot z
$$
in $\mc M (T)$.
The identity morphism of $(g,x)$ is represented by $\eta_{\mc M}^{-1} \colon x \to e \cdot x$.

\item The strict action of $H$ on $\widetilde{\mc M}$ given by
$$
    \gamma \cdot (g,x) = (\gamma \cdot g,x) 
        \text{ and } 
    \gamma \cdot \left(\widetilde{\phi} \colon (g,x) \to (h,y) \right)
        = \widetilde{\phi} \colon (\gamma g,x) \to (\gamma h,y),
$$
where $\gamma \cdot \widetilde{\phi} = \widetilde{\phi}$ is now considered as a map from $(\gamma g,x)$ to $(\gamma h,y)$.
This is well-defined since 
$$
    x 
        \xrightarrow{\phi} (g^{-1}h) \cdot x
        = ((\gamma g)^{-1} (\gamma h)) \cdot x\,.
$$
\end{itemize}
In \cite[Proposition 1.5]{Romagny}, it is shown that the functor $S_{\mc M} \colon \widetilde{\mc M} \to \mc{M}$, defined by
$$
    (g,x) \mapsto g \cdot x 
        \text{ and } 
    \left(\widetilde{\phi}  \colon (g,x) \to (h,y)\right) 
        \mapsto 
    \left(g\cdot x \xrightarrow{g \cdot \phi} g \cdot ((g^{-1}h) \cdot y) \xrightarrow{\mu_M^{-1}} (g \cdot (g^{-1}h) ) \cdot y = h \cdot y\right),
$$
is an $H$-isomorphism.

Note that if $\mathscr{X}$ is an $H$-stack, then by the universal property of stackification, 
% \[
% % https://tikzcd.yichuanshen.de/#N4Igdg9gJgpgziAXAbVABwnAlgFyxMJZABgBpiBdUkANwEMAbAVxiRAB12B3LWPB2ME4BbOjgAWcAMYAnYAA0AvopCLS6TLnyEUZAIxVajFmxFjJshSrUbseAkT3lD9Zq0Qd2oidLlKA5KrqIBh22o6kBtSuJh6qhjBQAObwRKAAZjIQwkhkIDgQSABMNiCZ2UhO+YWIRdQMdABGMAwACpr2OiAyWEniOCD1WGDuIFAQODiJQRlZOYh5BUgAzEMjbONMjQys1OIwdFBIYEwMDNQ4dFgMbJDrpeXzizVVMaMAygD6Ql7mvlYqepNFrtMIODw9PoDRQURRAA
% \begin{tikzcd}
% \widetilde{\mathscr{X}} \arrow[r] \arrow[rd, Rightarrow, pos= .1, shorten >=40pt] \arrow[d, "S_{\mathscr{X}}"'] & \mathscr{X}' \\
% \mathscr{X} \arrow[ru, dotted]                                                         & {}          
% \end{tikzcd}
% \]
it follows that the stackification of the strictification $\widetilde{\mathscr{X}}$ is isomorphic to $\mathscr{X}$.
Since stackification commutes with (finite) fibered products, we may
%, by the above proposition, 
restrict our attention to strict actions. 

\begin{proposition}
\label{prop: strictification morphism}
% Let $H$ be a group scheme over $S$, and 
Let $\mc{M}$ and $\mc{N}$ be $H$-groupoids over $S$. 
If $\pi \colon \mc{M} \to \mc{N}$ is an $H$-equivariant morphism, then there exists a strict $H$-morphism $\widetilde{\pi} \colon \widetilde{\mc M} \to \widetilde{\mc N}$ such that the following diagram commutes
\begin{equation}
\label{diag: (non)-strict-action-comm-diag}
% https://tikzcd.yichuanshen.de/#N4Igdg9gJgpgziAXAbVABwnAlgFyxMJZABgBpiBdUkANwEMAbAVxiRAB12B3LWPB2ME4BbOjgAWAY0bAAsgF95IeaXSZc+QigCM5KrUYs2IsVJkLlqkBmx4CRMtv31mrRB268Y-QSYnSGYAA5RUs1W00iXSdqFyN3PzNAkM5Sanl9GCgAc3giUAAzACcIYSQyEBwIJAAmWMM3Dx4+LAEYIXY0LCVqBjoAIxgGAAV1Oy0QIqxs8RwQXqwwRrgIBl55kHEYOigkMCYGBmocOla2SCWwkGLS8uPqxF0DVzYAZQB9DtF-c1DegaGowi9ncDBgBTmKkKJTKj3uSAAzPUXglOlgNn1BiMxpFQeDIVYbrC6pUHkjnvEQB8vqYAsE-iBMYCcSDJtNZlcibV4XCKY0AI4YgHY4ETMEQjGLRpQCBMfpg5QUeRAA
\begin{tikzcd}
\widetilde{\mathcal{M}} \arrow[d, "\widetilde{\pi}"'] \arrow[r, "S_{\mathcal{M}}"] & \mathcal{M} \arrow[d, "\pi"] \\
\widetilde{\mathcal{N}} \arrow[r, "S_{\mathcal{N}}"'] \arrow[ru, "q", Rightarrow]  & {\mathcal{N}\,,}            
\end{tikzcd}
\end{equation}
where $q \colon a_{\mc N} \circ (\mathrm{id}_{H} \times \pi) \implies \pi \circ a_{\mc {M}}$ is the 2-arrow corresponding to the equivariant map $\pi$, see the diagram \eqref{diag: equivariant map 2-isomorphism}.
% $q \colon S_{\mc N} \circ \widetilde{\pi} \to \pi \circ S_{M}$ is a $2$-isomorphism induced from \eqref{diag: equivariant map 2-isomorphism}.
\end{proposition}

\begin{proof}
% Suppose $q \colon a_{\mc N} \circ (\mathrm{id}_{H} \times \pi) \implies \pi \circ a_{\mc {M}}$ be the 2-arrow corresponding to the equivariant map $\pi$, see diagram \eqref{diag: equivariant map 2-isomorphism}.
Let us define the functor $\widetilde{\pi} \colon \widetilde{\mc M} \to \widetilde{\mc N}$ by $\widetilde{\pi}(g,x) = (g,\pi(x))$ and for a morphism $\widetilde{\phi} \colon (g,x) \to (h,y)$ represented by $\phi \colon x \to (g^{-1}h)\cdot y$ in $\mc M (T)$, the morphism $\widetilde{\pi}(\widetilde{\phi}) \colon (g,\pi(x)) \to (h,\pi(y))$ is represented by the composition
$$
        \pi(x) \xrightarrow{\pi(\phi)} \pi((g^{-1}h)\cdot y) \xrightarrow{q^{-1}} (g^{-1}h) \cdot \pi(y)
$$
in $\mc N (T)$.

Let us first check that $\widetilde{\pi}$ is indeed a functor. 
The map $\widetilde{\pi}(\widetilde{\psi} \circ \widetilde{\phi})$ is represented by the composition 
% $q^{-1} \circ \pi (\mu_{\mc M}^{-1} \circ ((g^{-1}h) \cdot \psi) \circ \phi)$.
$$
\pi(x) \xrightarrow{\pi(\phi)} \pi((g^{-1}h)\cdot y)
        \xrightarrow{\pi((g^{-1}h)\cdot \psi)} \pi((g^{-1}h)\cdot((h^{-1}k) \cdot z))
        \xrightarrow{\pi(\mu_{\mc M}^{-1})} \pi((g^{-1}k) \cdot z)
        \xrightarrow{q^{-1}} (g^{-1}k) \cdot \pi(z)\,,
$$
and the map $\widetilde{\pi}(\widetilde{\psi}) \circ \widetilde{\pi}(\widetilde{\phi})$ is represented by the composition
$$
    \pi(x) \xrightarrow{\pi(\phi)} \pi((g^{-1}h)\cdot y) \xrightarrow{q^{-1}} (g^{-1}h) \cdot \pi(y) 
    \xrightarrow{(g^{-1}h) \cdot (q^{-1}\circ \pi(\psi))} (g^{-1}h) \cdot ((h^{-1}k) \cdot \pi(z))
    \xrightarrow{\mu_{\mc N}^{-1}} (g^{-1}k)\cdot \pi(z)\,.
$$
We want to show that the following diagram is commutative:
\[
% https://tikzcd.yichuanshen.de/#N4Igdg9gJgpgziAXAbVABwnAlgFyxMJZABgBpiBdUkANwEMAbAVxiRAB120sAKHgcwB6wALQBGAL4ALAJScAxlAg4ABAE8ZICaXSZc+QigDM5KrUYs2nbnyGjJshUpx8pw8RIDWMlU+UqALxlNbV1sPAIiAFZTanpmVkQOLl4Bd0lvX3ZFfyCtHRAMcIMiMjEzeMsktPtpHz9Va14NfLD9SONScriLRJAaj1ksnNVXdK967OcsmyCQgqL2w2QY7vMEtgGMuSn-Jp4gzlIAOi0zGCh+eCJQADMAJwgAWyQyEBwIJBN1qpAAR3GIGoDDoACMYAwAAp6CKGEAMGC3HCtEAPZ6vagfJBiHobJL7LZ1BozbCaYFgiHQ4odeGI5GhVGPF6IHHvT6IABMuN++04TyYAH1gHz5CoALIScZk+EUqEwkpJBFIlFo5lctlIGI-PoA2pAmXguXUuFK+kFVVfTHsgAs3L6hKGxJ4uo8Ciw91FvK4pOlIMNVKWbFNKqZSFtGsQWsqfT5guF7CeooAcpK9eT-fKacGJBQJEA
\begin{tikzcd}
\pi((g^{-1}h)\cdot y) \arrow[d, "q^{-1}"] \arrow[rrr, "\pi((g^{-1}h)\cdot \psi)"] &  &  & \pi((g^{-1}h)\cdot((h^{-1}k) \cdot z)) \arrow[rr, "\pi(\mu_{\mc M}^{-1})"] &  & \pi((g^{-1}k) \cdot z) \arrow[d, "q^{-1}"] \\
(g^{-1}h) \cdot \pi(y) \arrow[rrr, "(g^{-1}h) \cdot (q^{-1}\circ \pi(\psi))"]     &  &  & (g^{-1}h) \cdot ((h^{-1}k) \cdot \pi(z)) \arrow[rr, "\mu_{\mc N}^{-1}"]    &  & {(g^{-1}k)\cdot \pi(z)\,.}                
\end{tikzcd}
\]
Using diagram \eqref{diag equivariant 1}, this is equivalent to showing that the following diagram is commutative:
\[
% https://tikzcd.yichuanshen.de/#N4Igdg9gJgpgziAXAbVABwnAlgFyxMJZABgBpiBdUkANwEMAbAVxiRAB120sAKHgcwB6wALQBGAL4ALAJScAxlAg4ABAE8ZICaXSZc+QigDM5KrUYs2nbnyGjJshUpx8pw8RIDWMlU+UqALxlNbV1sPAIiMjEzemZWRBABdwcfP1VrXg0tHRAMcIMiExjqOMtE5PtpNPZFf1cUrxq6jK5eIJDc-P1IlABWU1KLBKS7D1lfWudJmx43Ku90wM6wnsNkAZLzeLZK8ebphoWD-0yeDs5SagkzGCh+eCJQADMAJwgAWyQyEBwIJAATEMdokAI6NEDUBh0ABGMAYAAU9BFDCAGDBnjgci93l9ED8-kgxMDyhw2rZGo4pqcuNhNFDYfCkQVemiMVjQiA3p9AdRCYgTNtSXtUpMWioeOCqgosK95DNeNY6fS0YzEcjCol0ZjsVzcUS+f9EAAWEkjKUeXXcvGm35GgZCkYi6pi6YWyRW-UCw1IB0MLBgEZKJgw9GQkBSGB0KBsSCB8M4OhYBixgisG4SIA
\begin{tikzcd}
\pi((g^{-1}h)\cdot y) \arrow[d, "q^{-1}"] \arrow[rrr, "\pi((g^{-1}h)\cdot \psi)"] &  &  & \pi((g^{-1}h)\cdot((h^{-1}k) \cdot z)) \arrow[rr, "q^{-1}"]              &  & (g^{-1}h) \cdot \pi((h^{-1}k)\cdot z) \arrow[d, "(g^{-1}h) \cdot q^{-1}"] \\
(g^{-1}h) \cdot \pi(y) \arrow[rrr, "(g^{-1}h) \cdot (q^{-1}\circ \pi(\psi))"]     &  &  & (g^{-1}h) \cdot ((h^{-1}k) \cdot \pi(z)) \arrow[rr, Rightarrow, no head] &  & {(g^{-1}h) \cdot ((h^{-1}k) \cdot \pi(z))\,,}                            
\end{tikzcd}
\]
which is equivalent to the commutativity of the following diagram:
\[
% https://tikzcd.yichuanshen.de/#N4Igdg9gJgpgziAXAbVABwnAlgFyxMJZABgBpiBdUkANwEMAbAVxiRAB120sAKHgcwB6wALQBGAL4ALAJScAxlAg4ABAE8ZICaXSZc+QigDM5KrUYs2nbnyGjJshUpx8pw8RIDWMlU+UqALxlNbV1sPAIiMjEzemZWRBABdwcfP1VrXg0tHRAMcIMiExjqOMtE5PtpNPZFf0zXFK85WudAltIAOi0zGCh+eCJQADMAJwgAWyQyEBwIJAAmUosEkABHJpBqBjoAIxgGAAU9CMMQBhhhnByR8anEGbmkMWX4qy5eSo9HVvqubE02z2B2OBUiiQuVxuIDGk0W1CeiBM5jeFTs3xqdQyHx41gBW3OwKOJ0KEMu11CMLuzwR8yRr3K602EgoEiAA
\begin{tikzcd}
\pi((g^{-1}h)\cdot y) \arrow[d, "q^{-1}"] \arrow[rrr, "\pi((g^{-1}h)\cdot \psi)"] &  &  & \pi((g^{-1}h)\cdot((h^{-1}k) \cdot z)) \arrow[d, "q^{-1}"] \\
(g^{-1}h) \cdot \pi(y) \arrow[rrr, "(g^{-1}h) \cdot \pi(\psi)"]                   &  &  & {(g^{-1}h) \cdot \pi((h^{-1}k)\cdot z)\,.}                
\end{tikzcd}
\]
This follows since $q \colon a_{\mc N} \circ (\mathrm{id}_{H} \times \pi) \implies \pi \circ a_{\mc {M}}$ is a $2$-arrow.
% (apply naturality of $q$ to the morphism $(\mathrm{id}_{g^{-1}h} \colon g^{-1}h \to g^{-1}h,(\psi \colon y \to (h^{-1}k) \cdot y))$
Hence, $\widetilde{\pi}(\widetilde{\psi} \circ \widetilde{\phi}) = \widetilde{\pi}(\widetilde{\psi}) \circ \widetilde{\pi}(\widetilde{\phi})$.
Note that the map $\widetilde{\pi}(\mathrm{id}_{(g,x)})$ is represented by the composition
$$
    \pi(x) 
        \xrightarrow{\pi(\eta_{\mc M}^{-1})} \pi(e \cdot x) 
        \xrightarrow{q^{-1}} e \cdot \pi(x)\,,
$$
which is the morphism $\eta_{\mc N}^{-1}$ by diagram \eqref{diag equivariant 2}, the representative for $\mathrm{id}_{(g,\pi(x))}$.
Hence, $\widetilde{\pi}(\mathrm{id}_{(g,x)}) = \mathrm{id}_{(g,\pi(x))}$, and $\widetilde{\pi}$ is a functor.

Now we show that the morphism $\widetilde{\pi}$ is strict, i.e. the following diagram is commutative:
% $\widetilde{q}$ is identity.
\[
% https://tikzcd.yichuanshen.de/#N4Igdg9gJgpgziAXAbVABwnAlgFyxMJZABgBpiBdUkANwEMAbAVxiRAAkACAHW7wFt4PbgHcssPA1jBe-OjgAWAY0bAAsgF8NIDaXSZc+QigCM5KrUYs2vMRKxSYM7nMUqG6rTr0gM2PAREZCYW9MysiBzCAvC24jCS0rLyyqoAcl66+v5GRGYh1GHWkXH2js6uqR4ZGrykAHQ6FjBQAObwRKAAZgBOEPxIZCA4EEgATIVWESDJij38wOIaAPpcvDFwwnYJDkncaFja1Ax0AEYwDAAKBgHGID1YrQo43t19A4hDI0hmluFsdGWzm2iScsyqni8xzOF2uOUCkQYMC6LyyIF6-R+1G+iAAzJN-iVRPFQc4DkcQCdzlcbrlEcjUT4MR8JsNRniCcUQIDgSTdmCXCl3MAahSqbDaQjKQymhogA
\begin{tikzcd}
H \times \widetilde{\mathcal{M}} \arrow[d, "\mathrm{id}_H \times \widetilde{\pi}"'] \arrow[r, "a_{\widetilde{\mathcal{M}}}"] & \widetilde{\mathcal{M}} \arrow[d, "\widetilde{\pi}"] \\
H \times\widetilde{\mathcal{N}} \arrow[r, "a_{\widetilde{\mathcal{N}}}"]                                                     & {\widetilde{\mathcal{N}}\,.}                        
\end{tikzcd}
\]
% \[
% % https://tikzcd.yichuanshen.de/#N4Igdg9gJgpgziAXAbVABwnAlgFyxMJZABgBpiBdUkANwEMAbAVxiRAAkACAHW7wFt4PbgHcssPA1jBe-OjgAWAY0bAAsgF8NIDaXSZc+QigCM5KrUYs2vMRKxSYM7nMUqG6rTr0gM2PAREZCYW9MysiBzCAvC24jCS0rLyyqoAcl66+v5GRGYh1GHWkXH2js6uqR4Z2hoWMFAA5vBEoABmAE4Q-EhkIDgQSABMhVYRIMmKHfzA4hoA+gDi0ViCcMJ2CQ5J3GhY2tQMdABGMAwACgYBxiAdWI0KON7tXT2IfQNIZpbhbHTzzk2iSckyqni8hxOZ0uOUCkQYMDaTyyIE63S+1E+iAAzKNfiVRPFgc49gcQEdThcrrl4YjkT40W8Rv1Bji8cUQP9AUTtiCXCl3MAamSKdDqXDyXTnqjXsNMazvkVxqUtuUAI4iqFU2E3BFIkCHLBgcZQCBMY4InQUDRAA
% \begin{tikzcd}
% H \times \widetilde{\mathcal{M}} \arrow[d, "\mathrm{id}_G \times \widetilde{\pi}"'] \arrow[r, "a_{\widetilde{\mathcal{M}}}"] & \widetilde{\mathcal{M}} \arrow[d, "\widetilde{\pi}"] \\
% H \times\widetilde{\mathcal{N}} \arrow[r, "a_{\widetilde{\mathcal{N}}}"] \arrow[ru, "\widetilde{q}", Rightarrow]             & \widetilde{\mathcal{N}}                             
% \end{tikzcd}
% \]
For an object $(\gamma,(g,x))$ in $H(T) \times \widetilde{M}(T)$, we have 
$
(\widetilde{\pi} \circ a_{\widetilde{\mc M}}) (\gamma,(g,x)) 
    = \widetilde{\pi} (\gamma g,x)
    = (\gamma g , \pi(x)),
$
and
$
(a_{\widetilde{\mc N}} \circ (\mathrm{id}_H \times \widetilde{\pi})) (\gamma, (g,x))
        = a_{\widetilde{\mc N}} (\gamma, (g, \pi(x)))
        = (\gamma g ,\pi(x)).
$
As the action of $\gamma \in H(T)$ on a morphism $\widetilde{\phi} \colon (g,x) \to (h,y)$ in $\widetilde{\mc M} (T)$ does not change its representative in $\mc M (T)$, it follows that for a morphism $(\mathrm{id}_{\gamma}, \widetilde{\phi})$ in $H(T) \times \widetilde{M}(T)$, the morphisms
% NOTE: Only morphisms in $G(T)$ are identity maps. 
% Hence, above we have $(\mathrm{id}_{\gamma}, \widetilde{\phi}) \colon (\gamma,(g,x)) \to (\gamma,(h,y))$, where $\widetilde{\phi} \colon (g,x) \to (h,y)$ in $\widetilde{\mc M}(T)$ instead of  $(s, \widetilde{\phi}) \colon (\gamma,(g,x)) \to (\gamma',(h,y))$, where $s \colon \gamma \to \gamma'$ in $G(T)$ and $\widetilde{\phi} \colon (g,x) \to (h,y)$ in $\widetilde{\mc M}(T)$.
$$
(\widetilde{\pi} \circ a_{\widetilde{\mc M}}) (\mathrm{id}_{\gamma}, \widetilde{\phi})
    = \widetilde{\pi} \left(\widetilde{\phi} \colon (\gamma g,x) \to (\gamma h,y)\right)
    =  \widetilde{\pi} (\widetilde{\phi}) \colon (\gamma g,\pi(x)) \to (\gamma h,\pi(y))\,,
$$
and
$$
(a_{\widetilde{\mc N}} \circ (\mathrm{id}_H \times \widetilde{\pi}))(\mathrm{id}_{\gamma}, \widetilde{\phi})
    = a_{\widetilde{\mc N}}\left(\mathrm{id}_{\gamma}, \left(\widetilde{\pi}(\widetilde{\phi}) \colon (g,\pi(x)) \to (h,\pi(y))\right)\right)
    = \widetilde{\pi}(\widetilde{\phi}) \colon (\gamma g,\pi(x)) \to (\gamma h,\pi(y))
$$
are represented by the same morphism $\pi(x) \xrightarrow{\pi(\phi)} \pi((g^{-1}h)\cdot y) \xrightarrow{q^{-1}} (g^{-1}h) \cdot \pi(y)$ in $\mc M (T)$.
Hence, $(\widetilde{\pi} \circ a_{\widetilde{\mc M}}) (\mathrm{id}_{\gamma}, \widetilde{\phi}) = (a_{\widetilde{\mc N}} \circ (\mathrm{id}_H \times \widetilde{\pi}))(\mathrm{id}_{\gamma}, \widetilde{\phi})$, and the map $\widetilde{\pi}$ is strict.

Finally, we show that diagram \eqref{diag: (non)-strict-action-comm-diag} is $2$-commutative.
Suppose $\widetilde{\phi} \colon (g,x) \to (h,y)$ is a morphism in $\widetilde{\mc M}(T)$.
% represented by the morphism $\phi \colon x \to (g^{-1}h)\cdot y$ in $\mc M (T)$.
Then $S_{\mc N}(\widetilde{\pi}(\widetilde{\phi}))$ is the composition
$$
    g \cdot \pi(x) 
        \xrightarrow{g \cdot \pi(\phi)} g \cdot \pi((g^{-1}h) \cdot y)
        \xrightarrow{g \cdot q^{-1}} g \cdot ((g^{-1}h) \cdot \pi(y))
        \xrightarrow{\mu_{\mc N}^{-1}} h \cdot \pi(y)\,,
$$
and $\pi (S_{\mc M}(\widetilde{\phi}))$ is the composition
$$
    \pi(g\cdot x) 
        \xrightarrow{\pi(g \cdot \phi)} \pi(g \cdot ((g^{-1}h) \cdot y)) \xrightarrow{\pi(\mu_{\mc M}^{-1})}  \pi (h \cdot y)\,.
$$
We want to show that the following diagram is commutative:
\[
% https://tikzcd.yichuanshen.de/#N4Igdg9gJgpgziAXAbVABwnAlgFyxMJZABgBpiBdUkANwEMAbAVxiRAHMACAHW4GMoEHD25osACgAeASk4gAvqXSZc+QigCM5KrUYs2XXgKEix48ewB6wALQb5AC1lHBwgJ7SFSkBmx4CRABM2tT0zKyIHCLGwuZWtvZO0a6mEh6eisp+akQAzCG64WwOySa8Zh5eWaoBKGQaOmH6keUS7C4mMlU+Kv7qyPkNoXoRIK2c4iUd7tK8pAB03b41-cFDhc1jom2lsRbWdo7O-CnpCjowUOzwRKAAZgBOEAC2SGQgOBBIWhujhidlbbicoOLAZbyPF7faifJDBX4GXacACOB3s3Uhr0Q8NhiHyCJa3GeTAA+sBeM8+JwAHLyNHyDFPLHvXEAFmGRUiyJA1AYdAARjAGAAFXo5SIPLDsBw4RlQvEwr6IACsHM23MyIExSHZHyVADY1aNWhYkSCwXKsYa9UhVQStmYKaTyUSqQBZOkJeQZCjyIA
\begin{tikzcd}
g \cdot \pi(x)  \arrow[r, "g \cdot \pi(\phi)"] \arrow[d, "q"'] & g \cdot \pi((g^{-1}h) \cdot y) \arrow[r, "g \cdot q^{-1}"] & g \cdot ((g^{-1}h) \cdot \pi(y)) \arrow[r, "\mu_{\mc N}^{-1}"]      & h \cdot \pi(y) \arrow[d, "q"] \\
\pi(g\cdot x) \arrow[rr, "\pi(g \cdot \phi)"]                  &                                                            & \pi(g \cdot ((g^{-1}h) \cdot y)) \arrow[r, "\pi(\mu_{\mc M}^{-1})"] & {\pi (h \cdot y)\,.}         
\end{tikzcd}
\]
Using diagram \ref{diag equivariant 1}, this is equivalent to showing that the following diagram is commutative:
\[
% https://tikzcd.yichuanshen.de/#N4Igdg9gJgpgziAXAbVABwnAlgFyxMJZABgBpiBdUkANwEMAbAVxiRAHMACAHW4GMoEHD25osACgAeASk4gAvqXSZc+QijIBGKrUYs2vMePa8BQzjIVKQGbHgJEATKW3V6zVohCGJXU4OFxYwA9YABaTXkAC1l-cwBPaWleUgA6K2U7NSdyHXd9Lz9+AJEjEPDImJEzYUSFHRgodngiUAAzACcIAFskMhAcCCRNNz1PEABHEGoGOgAjGAYABRV7dRAOrHYonAyQTp7h6kGkZ10PA1FfapLDKKxpPYPexABmY6HEM-zxqcV2rovfonN6jC6FG7mHziO4PeryIA
\begin{tikzcd}
g \cdot \pi(x)  \arrow[d, "q"'] \arrow[rr, "g \cdot \pi(\phi)"] &  & g \cdot \pi((g^{-1}h) \cdot y) \arrow[d, "q"] \\
\pi(g\cdot x) \arrow[rr, "\pi(g \cdot \phi)"]                   &  & {\pi(g \cdot ((g^{-1}h) \cdot y))\,.}        
\end{tikzcd}
\]
This follows since $q \colon a_{\mc N} \circ (\mathrm{id}_{H} \times \pi) \implies \pi \circ a_{\mc {M}}$ is a $2$-arrow. 
% (apply naturality of $q$ to the morphism $(\mathrm{id}_{g} \colon g \to g,(\phi \colon x \to (g^{-1}h) \cdot y))$.
\end{proof}

%--------------------------------------------------------
%--------------------------------------------------------
\section{Acknowledgments}
%--------------------------------------------------------
%--------------------------------------------------------
R.S. Arora would like to acknowledge the support of IISER Pune - IDeaS Scholarship and Siemens-IISER Ph.D. fellowship. M. Poddar thanks Kiumars Kaveh for explaining his work and sharing many insights with him. His research was supported by a SERB MATRICS Grant: MTR/2019/001613. C. Gangopadhyay was supported by ANRF ECRG Grant: \\
ANRF/ECRG/2025/000347/PMS.
%--------------------------------------------------------
%--------------------------------------------------------
\bibliographystyle{halpha}
\bibliography{Toric}

@misc{stacks-project,
  author       = {The {Stacks project authors}},
  title        = {The Stacks project},
  howpublished = {\url{https://stacks.math.columbia.edu}},
  year         = {2024},
}

@book {Olsson,
    AUTHOR = {Olsson, Martin},
     TITLE = {Algebraic spaces and stacks},
    SERIES = {American Mathematical Society Colloquium Publications},
    VOLUME = {62},
 PUBLISHER = {American Mathematical Society, Providence, RI},
      YEAR = {2016},
     PAGES = {xi+298},
      ISBN = {978-1-4704-2798-6},
   MRCLASS = {14D23 (14D22)},
  MRNUMBER = {3495343},
MRREVIEWER = {Stefan\ Schr\"oer},
       DOI = {10.1090/coll/062},
       URL = {https://doi.org/10.1090/coll/062},
}

@article {MR4492498,
    AUTHOR = {Kaveh, Kiumars and Manon, Christopher},
     TITLE = {Toric principal bundles, piecewise linear maps and {T}its
              buildings},
   JOURNAL = {Math. Z.},
  FJOURNAL = {Mathematische Zeitschrift},
    VOLUME = {302},
      YEAR = {2022},
    NUMBER = {3},
     PAGES = {1367--1392},
      ISSN = {0025-5874,1432-1823},
   MRCLASS = {14M25 (14J60 20E42 51E24)},
  MRNUMBER = {4492498},
MRREVIEWER = {Zhuang\ He},
       DOI = {10.1007/s00209-022-03094-5},
       URL = {https://doi.org/10.1007/s00209-022-03094-5},
}

@article {MR4767486,
    AUTHOR = {Huang, Shaoyu and Kaveh, Kiumars},
     TITLE = {Toric principal bundles, {T}its buildings and reduction of
              structure group},
   JOURNAL = {Doc. Math.},
  FJOURNAL = {Documenta Mathematica},
    VOLUME = {29},
      YEAR = {2024},
    NUMBER = {4},
     PAGES = {815--830},
      ISSN = {1431-0635,1431-0643},
   MRCLASS = {14M25 (32L05)},
  MRNUMBER = {4767486},
MRREVIEWER = {Howard\ M.\ Thompson},
       DOI = {10.4171/dm/962},
       URL = {https://doi.org/10.4171/dm/962},
}

@article {Biswas-Majumdar-Wong,
    AUTHOR = {Biswas, Indranil and Majumder, Souradeep and Wong, Michael
              Lennox},
     TITLE = {Root stacks, principal bundles and connections},
   JOURNAL = {Bull. Sci. Math.},
  FJOURNAL = {Bulletin des Sciences Math\'ematiques},
    VOLUME = {136},
      YEAR = {2012},
    NUMBER = {4},
     PAGES = {369--398},
      ISSN = {0007-4497,1952-4773},
   MRCLASS = {14D23 (14H60 53B15)},
  MRNUMBER = {2923408},
MRREVIEWER = {Francesco\ Bottacin},
       DOI = {10.1016/j.bulsci.2012.03.006},
       URL = {https://doi.org/10.1016/j.bulsci.2012.03.006},
}

@article{alper2025stacks,
  title={Stacks and moduli},
  author={Alper, Jarod},
  year={July 7, 2025},
  url={https://sites.math.washington.edu/~jarod/lecture-notes.html}
}

@article {Romagny,
    AUTHOR = {Romagny, Matthieu},
     TITLE = {Group actions on stacks and applications},
   JOURNAL = {Michigan Math. J.},
  FJOURNAL = {Michigan Mathematical Journal},
    VOLUME = {53},
      YEAR = {2005},
    NUMBER = {1},
     PAGES = {209--236},
      ISSN = {0026-2285,1945-2365},
   MRCLASS = {14A20 (14H10)},
  MRNUMBER = {2125542},
MRREVIEWER = {Ivan\ S.\ Kausz},
       DOI = {10.1307/mmj/1114021093},
       URL = {https://doi.org/10.1307/mmj/1114021093},
}

@article {aldrovandi-NoohiII,
    AUTHOR = {Aldrovandi, Ettore and Noohi, Behrang},
     TITLE = {Butterflies {II}: torsors for 2-group stacks},
   JOURNAL = {Adv. Math.},
  FJOURNAL = {Advances in Mathematics},
    VOLUME = {225},
      YEAR = {2010},
    NUMBER = {2},
     PAGES = {922--976},
      ISSN = {0001-8708,1090-2082},
   MRCLASS = {18D05 (18D30 18E30 18G35)},
  MRNUMBER = {2671184},
MRREVIEWER = {Vasil\cprime\ \=I.\ Andr\=\i\u ichuk},
       DOI = {10.1016/j.aim.2010.03.011},
       URL = {https://doi.org/10.1016/j.aim.2010.03.011},
}

@incollection {Breen-bitorseur,
    AUTHOR = {Breen, Lawrence},
     TITLE = {Bitorseurs et cohomologie non ab\'elienne},
 BOOKTITLE = {The {G}rothendieck {F}estschrift, {V}ol.\ {I}},
    SERIES = {Progr. Math.},
    VOLUME = {86},
     PAGES = {401--476},
 PUBLISHER = {Birkh\"auser Boston, Boston, MA},
      YEAR = {1990},
      ISBN = {0-8176-3427-4},
   MRCLASS = {18G50 (18G55 55U05 55U15)},
  MRNUMBER = {1086889},
MRREVIEWER = {Antonio\ M.\ Cegarra},
}

@article {Breen-Theorie,
    AUTHOR = {Breen, Lawrence},
     TITLE = {Th\'eorie de {S}chreier sup\'erieure},
   JOURNAL = {Ann. Sci. \'Ecole Norm. Sup. (4)},
  FJOURNAL = {Annales Scientifiques de l'\'Ecole Normale Sup\'erieure.
              Quatri\`eme S\'erie},
    VOLUME = {25},
      YEAR = {1992},
    NUMBER = {5},
     PAGES = {465--514},
      ISSN = {0012-9593},
   MRCLASS = {18G50},
  MRNUMBER = {1191733},
MRREVIEWER = {Graham\ J.\ Ellis},
       URL = {http://www.numdam.org/item?id=ASENS_1992_4_25_5_465_0},
}

@book {SGA4,
     TITLE = {Th\'eorie des topos et cohomologie \'etale des sch\'emas},
    SERIES = {Lecture Notes in Mathematics},
    VOLUME = {Vol. 269},
      NOTE = {S\'eminaire de G\'eom\'etrie Alg\'ebrique du Bois-Marie
              1963--1964 (SGA 4),
              Dirig\'e{} par M. Artin, A. Grothendieck, et J. L. Verdier.
              Avec la collaboration de N. Bourbaki, P. Deligne et B.
              Saint-Donat},
 PUBLISHER = {Springer-Verlag, Berlin-New York},
      YEAR = {1972},
     PAGES = {xix+525},
   MRCLASS = {14-06},
  MRNUMBER = {354652},
}

@article {BCS,
    AUTHOR = {Borisov, Lev A. and Chen, Linda and Smith, Gregory G.},
     TITLE = {The orbifold {C}how ring of toric {D}eligne-{M}umford stacks},
   JOURNAL = {J. Amer. Math. Soc.},
  FJOURNAL = {Journal of the American Mathematical Society},
    VOLUME = {18},
      YEAR = {2005},
    NUMBER = {1},
     PAGES = {193--215},
      ISSN = {0894-0347,1088-6834},
   MRCLASS = {14N35 (14C15 14M25)},
  MRNUMBER = {2114820},
MRREVIEWER = {Domenico\ Fiorenza},
       DOI = {10.1090/S0894-0347-04-00471-0},
       URL = {https://doi.org/10.1090/S0894-0347-04-00471-0},
}

@article {FMN,
    AUTHOR = {Fantechi, Barbara and Mann, Etienne and Nironi, Fabio},
     TITLE = {Smooth toric {D}eligne-{M}umford stacks},
   JOURNAL = {J. Reine Angew. Math.},
  FJOURNAL = {Journal f\"ur die Reine und Angewandte Mathematik. [Crelle's
              Journal]},
    VOLUME = {648},
      YEAR = {2010},
     PAGES = {201--244},
      ISSN = {0075-4102,1435-5345},
   MRCLASS = {14M25 (14D23)},
  MRNUMBER = {2774310},
MRREVIEWER = {Hsian-Hua\ Tseng},
       DOI = {10.1515/CRELLE.2010.084},
       URL = {https://doi.org/10.1515/CRELLE.2010.084},
}

@book {CLS,
    AUTHOR = {Cox, David A. and Little, John B. and Schenck, Henry K.},
     TITLE = {Toric varieties},
    SERIES = {Graduate Studies in Mathematics},
    VOLUME = {124},
 PUBLISHER = {American Mathematical Society, Providence, RI},
      YEAR = {2011},
     PAGES = {xxiv+841},
      ISBN = {978-0-8218-4819-7},
   MRCLASS = {14M25 (05A15 05E45 52B12)},
  MRNUMBER = {2810322},
MRREVIEWER = {Ivan\ Arzhantsev},
       DOI = {10.1090/gsm/124},
       URL = {https://doi.org/10.1090/gsm/124},
}

@article {Klyachko,
    AUTHOR = {Klyachko, A. A.},
     TITLE = {Equivariant bundles over toric varieties},
   JOURNAL = {Izv. Akad. Nauk SSSR Ser. Mat.},
  FJOURNAL = {Izvestiya Akademii Nauk SSSR. Seriya Matematicheskaya},
    VOLUME = {53},
      YEAR = {1989},
    NUMBER = {5},
     PAGES = {1001--1039, 1135},
      ISSN = {0373-2436},
   MRCLASS = {14M25 (14F05)},
  MRNUMBER = {1024452},
MRREVIEWER = {I.\ Dolgachev},
       DOI = {10.1070/IM1990v035n02ABEH000707},
       URL = {https://doi.org/10.1070/IM1990v035n02ABEH000707},
}

@article {IS,
    AUTHOR = {Ilten, Nathan and S\"{u}ss, Hendrik},
     TITLE = {Equivariant vector bundles on {$T$}-varieties},
   JOURNAL = {Transform. Groups},
  FJOURNAL = {Transformation Groups},
    VOLUME = {20},
      YEAR = {2015},
    NUMBER = {4},
     PAGES = {1043--1073},
      ISSN = {1083-4362,1531-586X},
   MRCLASS = {14L30 (14J60)},
  MRNUMBER = {3416439},
MRREVIEWER = {P.\ E.\ Newstead},
       DOI = {10.1007/s00031-015-9312-2},
       URL = {https://doi.org/10.1007/s00031-015-9312-2},
}

@article{Kaneyama, 
   author={Kaneyama, T.},
   title={On equivariant vector bundles on an almost homogeneous variety},
   journal={Nagoya Math. J.},
   volume={57},
   year={1975},
   pages={65--86},
   issn={0027-7630},
}

@article {BDP2016,
    AUTHOR = {Biswas, Indranil and Dey, Arijit and Poddar, Mainak},
     TITLE = {A classification of equivariant principal bundles over
              nonsingular toric varieties},
   JOURNAL = {Internat. J. Math.},
  FJOURNAL = {International Journal of Mathematics},
    VOLUME = {27},
      YEAR = {2016},
    NUMBER = {14},
     PAGES = {1650115, 16},
      ISSN = {0129-167X,1793-6519},
   MRCLASS = {32L05 (14M25 55R91)},
  MRNUMBER = {3593677},
MRREVIEWER = {Ragni\ Piene},
       DOI = {10.1142/S0129167X16501159},
       URL = {https://doi.org/10.1142/S0129167X16501159},
}

@book {TitsBuilding,
    AUTHOR = {Tits, Jacques},
     TITLE = {Buildings of spherical type and finite {BN}-pairs},
    SERIES = {Lecture Notes in Mathematics},
    VOLUME = {Vol. 386},
 PUBLISHER = {Springer-Verlag, Berlin-New York},
      YEAR = {1974},
     PAGES = {x+299},
   MRCLASS = {20G15 (50A20)},
  MRNUMBER = {470099},
MRREVIEWER = {Bruce\ Cooperstein},
}

@book {BrownBuildings,
    AUTHOR = {Brown, Kenneth S.},
     TITLE = {Buildings},
 PUBLISHER = {Springer-Verlag, New York},
      YEAR = {1989},
     PAGES = {viii+215},
      ISBN = {0-387-96876-8},
   MRCLASS = {20-02 (20E32 20G15 22E99 51B25)},
  MRNUMBER = {969123},
MRREVIEWER = {W.\ M.\ Kantor},
       DOI = {10.1007/978-1-4612-1019-1},
       URL = {https://doi.org/10.1007/978-1-4612-1019-1},
}

@book{garrett1997buildings,
  title={Buildings and classical groups},
  author={Garrett, Paul B},
  year={1997},
  publisher={CRC Press}
}

@book {HumphreysLAG,
    AUTHOR = {Humphreys, James E.},
     TITLE = {Linear algebraic groups},
    SERIES = {Graduate Texts in Mathematics},
    VOLUME = {No. 21},
 PUBLISHER = {Springer-Verlag, New York-Heidelberg},
      YEAR = {1975},
     PAGES = {xiv+247},
   MRCLASS = {20GXX (14LXX)},
  MRNUMBER = {396773},
MRREVIEWER = {T.\ Ono},
}

@article {MR3453974,
    AUTHOR = {Biswas, Indranil and Dey, Arijit and Poddar, Mainak},
     TITLE = {Equivariant principal bundles and logarithmic connections on
              toric varieties},
   JOURNAL = {Pacific J. Math.},
  FJOURNAL = {Pacific Journal of Mathematics},
    VOLUME = {280},
      YEAR = {2016},
    NUMBER = {2},
     PAGES = {315--325},
      ISSN = {0030-8730,1945-5844},
   MRCLASS = {14L30 (14M25 14M27)},
  MRNUMBER = {3453974},
MRREVIEWER = {Justin\ Brown},
       DOI = {10.2140/pjm.2016.280.315},
       URL = {https://doi.org/10.2140/pjm.2016.280.315},
}

@book {GIT,
    AUTHOR = {Mumford, D. and Fogarty, J. and Kirwan, F.},
     TITLE = {Geometric invariant theory},
    SERIES = {Ergebnisse der Mathematik und ihrer Grenzgebiete (2) [Results
              in Mathematics and Related Areas (2)]},
    VOLUME = {34},
   EDITION = {Third},
 PUBLISHER = {Springer-Verlag, Berlin},
      YEAR = {1994},
     PAGES = {xiv+292},
      ISBN = {3-540-56963-4},
   MRCLASS = {14D25 (58E05 58F05)},
  MRNUMBER = {1304906},
MRREVIEWER = {Yi\ Hu},
}

@article {MR4960069,
    AUTHOR = {Dasgupta, Jyoti and Khan, Bivas and Biswas, Indranil and Dey,
              Arijit and Poddar, Mainak},
     TITLE = {Classification, reduction, and stability of toric principal
              bundles},
   JOURNAL = {Transform. Groups},
  FJOURNAL = {Transformation Groups},
    VOLUME = {30},
      YEAR = {2025},
    NUMBER = {3},
     PAGES = {1083--1133},
      ISSN = {1083-4362,1531-586X},
   MRCLASS = {14M25 (14J60 32L05)},
  MRNUMBER = {4960069},
       DOI = {10.1007/s00031-023-09812-5},
       URL = {https://doi.org/10.1007/s00031-023-09812-5},
}

@article {MR4166679,
    AUTHOR = {Biswas, Indranil and Dey, Arijit and Poddar, Mainak},
     TITLE = {Tannakian classification of equivariant principal bundles on
              toric varieties},
   JOURNAL = {Transform. Groups},
  FJOURNAL = {Transformation Groups},
    VOLUME = {25},
      YEAR = {2020},
    NUMBER = {4},
     PAGES = {1009--1035},
      ISSN = {1083-4362,1531-586X},
   MRCLASS = {14M15 (14J60 14L15)},
  MRNUMBER = {4166679},
MRREVIEWER = {Jongbaek\ Song},
       DOI = {10.1007/s00031-020-09557-5},
       URL = {https://doi.org/10.1007/s00031-020-09557-5},
}

@phdthesis{Singh2026,
  author       = {Singh, Yash Vardhan},
  title        = {Vector Bundles on Toric Stacks},
  school       = {University of Waterloo},
  year         = {2026},
  url          = {https://hdl.handle.net/10012/23545}
}

@article {ToricStacksI,
    AUTHOR = {Geraschenko, Anton and Satriano, Matthew},
     TITLE = {Toric stacks {I}: {T}he theory of stacky fans},
   JOURNAL = {Trans. Amer. Math. Soc.},
  FJOURNAL = {Transactions of the American Mathematical Society},
    VOLUME = {367},
      YEAR = {2015},
    NUMBER = {2},
     PAGES = {1033--1071},
      ISSN = {0002-9947,1088-6850},
   MRCLASS = {14D23 (14M25)},
  MRNUMBER = {3280036},
MRREVIEWER = {Fabio\ Perroni},
       DOI = {10.1090/S0002-9947-2014-06063-7},
       URL = {https://doi.org/10.1090/S0002-9947-2014-06063-7},
}

@article {ToricStacksII,
    AUTHOR = {Geraschenko, Anton and Satriano, Matthew},
     TITLE = {Toric stacks {II}: {I}ntrinsic characterization of toric
              stacks},
   JOURNAL = {Trans. Amer. Math. Soc.},
  FJOURNAL = {Transactions of the American Mathematical Society},
    VOLUME = {367},
      YEAR = {2015},
    NUMBER = {2},
     PAGES = {1073--1094},
      ISSN = {0002-9947,1088-6850},
   MRCLASS = {14D23 (14M25)},
  MRNUMBER = {3280037},
MRREVIEWER = {Fabio\ Perroni},
       DOI = {10.1090/S0002-9947-2014-06064-9},
       URL = {https://doi.org/10.1090/S0002-9947-2014-06064-9},
}

@article{kundu2026fixed,
  title={Fixed point locus of Moduli spaces of Sheaves on Toric DM stacks},
  author={Kundu, Promit},
  journal={arXiv preprint arXiv:2605.01730},
  year={2026}
}
%--------------------------------------------------------
%--------------------------------------------------------

\end{document}